\documentclass[aos,noinfoline]{imsart}
\usepackage{amsmath,amssymb,amsthm,booktabs}
\usepackage{mathrsfs}
\usepackage[colorlinks,citecolor=blue,urlcolor=blue]{hyperref}

\usepackage{bm}\usepackage{euscript}
\usepackage{graphicx}
\usepackage{multicol}
\usepackage[usenames,dvipsnames,svgnames,table]{xcolor}

\usepackage{natbib}
\startlocaldefs
\numberwithin{equation}{section} \theoremstyle{plain}
\newtheorem{theorem}{Theorem}[section]
\newtheorem{lemma}{Lemma}[section]
\newtheorem{corollary}{Corollary}[section]
\newtheorem{proposition}{Proposition}[section]

\endlocaldefs

\def\blue{}

\allowdisplaybreaks

\begin{document}

\newcommand{\cov}{\text{Cov}}
\newcommand{\var}{\text{Var}}
\newcommand\tr{\mathop{tr}}
\newcommand\E{\mathbb{E}}
\newcommand\N{\mathcal{N}}
\newcommand\EN{\EuScript{N}}
\newcommand\lb{\left(}
\newcommand\rb{\right)}
\newcommand\veps{\varepsilon}
\newcommand\norm[1]{\left\lVert#1\right\rVert}
\newcommand\x{\mathbf{x}}
\newcommand\y{\mathbf{y}}
\newcommand\z{\mathbf{z}}
\renewcommand\a{\mathbf{a}}
\renewcommand{\d}[1]{\ensuremath{\operatorname{d}\!{#1}}}

\newcommand\gai[1]{{#1}}
\newcommand\ggai[1]{#1}

\def\cL{\mathcal{L}}
\def\cM{\mathcal{M}}
\def\X{\mathbf{X}}
\def\p{\mathbf{p}}
\def\wh{\widehat}
\def\wt{\widetilde}
\def\P{\mathbf{P}}
\def\D{\mathbf{D}}
\def\Q{\mathbf{Q}}
\def\Tr{\mathrm{Tr}}
\def\diag{\text{diag}}
\begin{frontmatter}

  \title{On testing for high-dimensional white noise}
  \runtitle{On testing a  high-dimensional white noise}

  \begin{aug}
    \author{\fnms{~ Zeng} \snm{Li}\ead[label=e1]{zxl278@psu.edu}}
    \footnote{Li's  research was supported by was supported by NIDA, NIH grants
      P50 DA039838,  a NSF grant DMS 1512422 and National Nature Science Foundation of
      China (NNSFC), 11690015.}
    \and
    \author{\fnms{~ Clifford} \snm{Lam~}\ead[label=e3]{c.lam2@lse.ac.uk}}
    \and
    \author{\fnms{~ Jianfeng} \snm{Yao~}\ead[label=e2]{jeffyao@hku.hk}}
    \and
    \author{\fnms{~ Qiwei} \snm{Yao}\ead[label=e4]{q.yao@lse.ac.uk}}
    
    \affiliation{London School of Economics, Pennsylvania State University, The University of Hong Kong}
    \runauthor{Z. Li, C. Lam,  J. Yao \& Q. Yao}

    \address{ Department of Statistics\\
      Pennsylvania State University\\
      \printead{e1}}
    
    \address{ 
      Department of Statistics and Actuarial Science\\
      The University of Hong Kong\\
      \printead{e2}
    }

    \address{
      Department of Statistics \\
      London School of Economics and Political Science\\
      \printead{e3,e4}
    }
  \end{aug}

\begin{abstract}
Testing for white noise is a classical yet important problem in
statistics, especially for diagnostic checks in time series modeling and linear 
regression.
For high-dimensional time series in the sense that the dimension $p$ is
large in relation to the sample size $T$, 
the popular omnibus tests
including the multivariate Hosking and Li-McLeod tests are extremely
conservative, leading to substantial power loss.
To develop more relevant tests for high-dimensional cases,
we propose a portmanteau-type test statistic which is the sum of squared singular values of
the first $q$ lagged sample autocovariance matrices. It, therefore,
encapsulates all
the serial correlations (upto the time lag $q$) within and across all
component series.
Using the tools from random matrix theory and assuming   both $p$ and 
$T$ diverge to infinity, 
we derive the asymptotic normality of  the test statistic  under both
the null and a specific  VMA(1) alternative hypothesis.
As the  actual implementation of the test requires the knowledge of three
characteristic constants of the population cross-sectional covariance
matrix and the value of the fourth moment of the standardized
innovations, non trivial estimations are proposed for these parameters
and their integration leads to a practically usable test. 
Extensive simulation confirms the
excellent finite-sample performance of the new test 
with accurate size and satisfactory power for 
a large  range of
finite $(p,T)$ combinations, 
therefore ensuring wide applicability 
in practice. 
In particular, the
new tests are consistently superior to the traditional Hosking
and Li-McLeod tests.
\end{abstract}

 \end{frontmatter}

\section{Introduction}

Testing for white noise is an important problem in
statistics. It is indispensable in diagnostic checking for linear regression and linear time series modeling in particular.
The surge of recent interests in modeling high-dimensional time series 
adds a further challenge: diagnostic checking
demands the testing for high-dimensional white noise in the sense that
the dimension of time series is comparable to or
even greater than the sample size (i.e.,
the observed length of the time series).
One prominent example showing the need for diagnostic checking in
high-dimensional time series concerns the vector autoregressive model,
which has a large literature.
When the dimension is large, most existing works regularize the fitted models by Lasso \citep{HHC2008, HNMK2009, SM2010,
BM2015},
Dantzig penalization \citep{HL2013}, banded autocovariances \citep{bickel2011},
or banded auto-coefficient matrices \citep{GuoWangYao16}. However, none of them have developed any residual-based diagnostic tools. Another popular approach is to represent high-dimensional time series
by lower-dimensional factors. See for example,
\cite{SW89,SW98,SW99},  \cite{Fornietal_2000,Fornietal_2005}, \cite{BaiNg_Econometrica_2002}, \cite{LamYao_AOS_2012} and
\cite{ChangGuoYao_2015}.
Again, there is a pertinent need to develop appropriate tools for checking the validity of
the fitted factor models through careful examination of the
residuals.

There are several well-established  white noise tests for univariate time series \citep{Li03}.
Some of them have been extended for testing vector time series
\citep{Hosking80,Li81,Lutkepohl_2005}.
However, these methods are designed for the cases where the
dimension of the time series is small or relatively small compared to the sample size.
For the purpose of model diagnostic checking, the so-called omnibus tests are
often adopted which are designed to detect any forms of departure from white noise.
The celebrated Box-Pierce portmanteau test and its variations are the most popular omnibus
tests. The fact that the Box-Pierce test and its variations are asymptotically
distribution-free and $\chi^2$-distributed under the null hypothesis
makes them particularly easy to use in practice. However, it is well
known in the literature that the slow convergence to
their asymptotic null distributions is particularly pronounced in
multivariate cases. On the other hand, testing for high-dimensional time
series is still in an infancy stage. To our best knowledge, the only
available methods are \cite{CYZ17}  and \cite{Tsay17}.

To appreciate the challenge in testing for a high-dimensional white
noise, 
we refer to an  example reported in Section 3.1 below where say, 
we have to check the residuals from a fitted  multivariate volatility for a portfolio containing $p=50$
stocks using their daily returns over a period of one semester. The
length of the returns time series is then approximately
$T=100$. Table~\ref{Tab:cLSizeCom}
shows that the two variants of the multivariate portmanteau test,
namely the Hosking and Li-McLeod tests, all have actual sizes around
0.1\%, instead of the nominal level of 5\%. These omnibus tests are thus
extremely conservative and they will not be able to detect an eventual misfitting of the volatility
model.

The above example 
illustrates the following fact, which is now better understood: many popular tools in multivariate statistics are severely challenged by the
emergence of high-dimensional data, and they need to be {\em re-examined} or {\em corrected}.
Recent advances in  high-dimensional statistics  demonstrate that
feasible and quality  solutions to
these high-dimensional challenges can be obtained by exploiting
tools of  random matrix theory
via  a precise  spectral analysis of large
sample covariance or sample autocovariance matrices.
For a review on such progress, we
refer to \cite{John07}, \cite{PaulAue14} and 
monograph \cite{YZB15a}.
In particular, asymptotic results found in this context using random
matrix theory exhibit 
fast convergence rates, and hence provide satisfactory finite sample approximation
for data analysis.

This paper proposes a new method for testing high-dimensional white
noise. 
 The test statistic  encapsulates the
serial correlations within and across all component series.
Precisely, the statistic is the sum of the squared
singular values of several lagged   sample autocovariance matrices.
Using random matrix theory, asymptotic normality for the test statistics
\ggai{under the null} is established under the Mar\v{c}enko-Pastur asymptotic regime where $p$
and $T$ are large and comparable.
Next, original methods are proposed for estimation of a few parameters
in the limiting distribution in order to get a fully implementable
version of the test.
{The asymptotic power of the  test under a
  specific alternative of first-order vector moving average process (VMA(1))
  has also been derived}.
Extensive simulation 
demonstrates excellent
behavior of the proposed tests for a wide array of
combinations of $(p,T)$, with
accurate size and satisfactory power.
In this paper, we also explore the reasons why
the popular multivariate Hosking and Li-McLeod tests are no longer reliable
when the dimension is large in relation to the sample size.

The rest of the paper is organized as follows.
Section~\ref{sec:multlag} presents the main contributions of the
paper. A new high-dimensional test for white noise is introduced, \ggai{its 
asymptotic distributions under both the null and the VMA(1)
alternative hypothesis are  established.}
Section~\ref{sec:simul} reports extensive Monte-Carlo experiments
which assess the finite sample behavior of the tests.
Whenever possible, comparison is made with the
popular Hosking and Li-McLeod tests.
Numerical evidence also indicates
that the new test is more powerful than that of \cite{CYZ17}. 
 Section~\ref{sec:proofs} collects all the
technical proofs.

\section{A test for high-dimensional white noise}\label{sec:multlag}

Let $\x_1,\cdots,\x_T$ be observations from a $p \times 1$
complex-valued linear process of the form 
\begin{equation*}
  \x_t=\sum_{l\geq0}A_l{\bf z}_{t-l},
\end{equation*}
where $A_l$ are $p\times p$ coefficient matrices, $\{{\bf z}_t\}$ is a
sequence of  $p$-dimensional 
random vectors such that, if the coordinates of ${\bf z}_t$ are
$\{z_{it}\}$, then the two dimensional array $\{z_{it}: 1\le i\le
p,t\ge 1\}$  of variables  are i.i.d.  satisfying
the moment conditions $\E z_{it}=0,~\E |z_{it}|^2=1$ and $
\E |z_{it}|^4=\nu_4<\infty$.
Hence $\E\x_t = {\bf 0}$, and  $\Sigma_\tau\equiv \cov(\x_{t+\tau},\x_t)$ depends on $\tau$
only. Note that $\Sigma_0=\mathop{\text{var}}(\x_t)$ is the population
covariance matrix of the time series.
The goal is to test  the null hypothesis 
\begin{equation} \label{eq:H0}
H_0:~ \x_t=A_0\z_t
\end{equation}
where $A_0$ is unknown.
This in fact tests the independence instead of linear independence
(i.e. $\Sigma_\tau=0$ for all $\tau \ne 0$), which is however
a common practice in the literature of white noise tests.
 Throughout the paper, the complex adjoint
of a matrix (or vector) $A$ is denoted by $A^*$.  For $\tau\ge 1$, let
$\widehat{\Sigma}_{\tau}$ be the {\em lag $\tau$ sample
	autocovariance matrix}
\begin{equation*}
\widehat{\Sigma}_{\tau}=\frac{1}{T}\sum_{t=1}^{T} \x_t\x_{t-\tau}^*.
\end{equation*}
\ggai{where by convention $\x_t=\x_{T+t}$ when $t\leq 0$.}
Under the null hypothesis, $\E(\widehat{\Sigma}_\tau)=0$ for $\tau\ne 0$, and
a natural test statistic
is \gai{the sum of squared singular values of the first $q$ lagged
  sample autocovariance matrices}:
\[G_q =
\sum_{\tau=1}^{q}\Tr\lb\widehat{\Sigma}_\tau^*\widehat{\Sigma}_\tau\rb
= \sum_{\tau=1}^{q}  \sum_{j} \alpha_{\tau,j}^2,
\]
\gai{where
  $\{\alpha_{\tau,j}\}$ are the singular values of $\widehat{\Sigma}_{\tau}$,}
and $\Tr$ denotes the trace operation for square matrices.
We reject the null hypothesis $H_0$ for large values of $G_q$.

Notice that the setting here allows for complex-valued observations: this is
important for applications in areas such as signal processing where
signal time series are usually complex-valued. 
However, for the sake of presentation, 
we mostly focus on the real-valued case in the subsequent
sections. 
Directions on how the tests can be extended to accommodate
complex-valued observations will be given in the last Section~\ref{subsec:complextests}.

\subsection{High dimensional asymptotics}\label{subsec:asymptotics}
We adopt  the so-called {\em  Mar\v{c}enko-Pastur
  regime} for asymptotic analysis, i.e. we assume $c_p=p/T\rightarrow c>0$
when $p, T\rightarrow \infty$.
\gai{This asymptotic framework has been
  widely employed in the literature on high-dimensional
  statistics,
  see,  \cite{John07}, \cite{PaulAue14},  also monograph  \cite{YZB15a} 
and the references within.
  Most of the results in this area concern sample
  covariance matrices only. However our test statistic $G_q$ is based on the
  sample autocovariance matrices, which is much less studied; see
\cite{LiuAuePaul15} and \cite{BB16}.
}

As a main contribution of the paper,
we characterize the asymptotic distribution of $G_q$ in this
high-dimensional setting when the observations are real-valued.
We introduce the following limits whenever
they exist: for $\ell\ge 1$,
\begin{equation}
  \label{eq:limits}
  s_\ell = \lim_{p\rightarrow \infty} \frac{1}{p}\Tr(\Sigma_0^\ell),
  \quad
  s_{d,\ell} = \lim_{p\rightarrow \infty}\frac{1}{p}\Tr(D^\ell(\Sigma_0)),
\end{equation}
where $D(A)$ denotes the diagonal matrix consisting of the main diagonal elements
 of $A$ (here the $d$ in the index is a reminder of this diagonal
 structure). 

\begin{theorem}\label{MainThmMult}
  Let $q\geq 1$ be a fixed integer, and the following assertions hold.
  \begin{enumerate}
  \item $\{{\bf z}_t\}$ is a sequence of real-valued independent $p\times 1$
    random vectors with independent components
    $\z_t=(z_{it})$
    satisfying
    $\E z_{it}=0,~\E z_{it}^2=1$ and $\E z_{it}^4=\nu_4<\infty$;
  \item
    $\left\{\Sigma_0\right\}$ is a sequence of $p\times p$
    semi-positive definite matrices with bounded spectral norm
    \gai{such that the limits $\{s_1, s_2\}$ and $\{s_{d,2}\}$ exist;}
  \item (Mar\v{c}enko-Pastur  regime).
    The dimension $p$ and the sample size $T$ grow  to infinity in
    a related way such that     $c_p:=p/T\rightarrow c>0$.
  \end{enumerate}
  Then when  $\x_t=\Sigma_0^{1/2}\z_t$, the limiting distribution of
  the test statistic $G_q$ is
  \begin{equation}
    \label{eq:Gq}
    G_q - {qTc_p^2}{s_1^2}
    ~~\xrightarrow{d}~~ \EN\lb 0,  \sigma^2(c) \rb,
  \end{equation}
  where
  \begin{equation}
    \label{eq:sigma2}
    \sigma^2(c)=
    2qc^2s_2^2 + 4q^2c^3(\nu_4-3)s_1^2s_{d,2} + 8q^2c^3s_1^2s_2 .
  \end{equation}
\end{theorem}
The proof of this theorem is given in Section \ref{sec:proofs}. {\blue It's worth mentioning here that in \cite{BB16}, they considered a simpler case when $\Sigma_0={\bf I}_p$, $q=1$ and $p=T$ with Gaussian population distribution, which is consistent with the results above.}

\medskip
Let $Z_\alpha$ be the upper-$\alpha$ quantile of the standard normal
distribution at level $\alpha$.
Based on Theorem~\ref{MainThmMult}, we obtain a  procedure for  testing  the null hypothesis in
(\ref{eq:H0}) as follows.
\begin{equation}\label{eq:Gqtest}
\text{\em Reject ~ $H_0$ ~~if ~~}
\left\{G_q-qTc_p^2s_1^2 > Z_{\alpha}   \gai{\sigma(c)}  \right\}.
\end{equation}

The illustration in Section
\ref{sec:simul} indicates that the test above is much more powerful
than  some classical alternatives, especially
when
 \gai{the dimension $p$ is growing linearly with the sample size $T$}.  The power of this
test is gained from gathering together the serial correlations
from the first $q$ lags within and across all $p$ component series; 
see the definition of $G_q$. Also note that the
asymptotic mean of $G_q$ is $qTc_p^2s_1^2$,
which grows linearly with $T$ (and $p$), while its
asymptotic variance  \gai{ $\sigma^2(c)$ is a
constant}.
This implies that even for moderately large  $T$, departure from
white noise in the first $q$ \ggai{lags of the} autocovariance matrices is
likely to result in a large and different mean, which will be a large multiple 
standard deviation away from $qTc_p^2s_1^2$ since 
the standard deviation
$\sigma(c)$ is constant.

However the test $G_q$ in \eqref{eq:Gqtest} is not yet practically
usable as it depends on 
(i) three characteristic constants,  $s_1$, $s_2$
and $s_{d,2}$ of the (population) cross-sectional
  covariance matrix $\Sigma_0$ and 
(ii) the fourth moment $\nu_4$ of the innovations $\{{\bf z}_t\}$.
These issues are addressed below.

\subsection{Estimation of the covariance characteristics $s_1$ and $s_2$}
\label{subsec:unknownSigma}

If the cross-sectional  covariance matrix $\Sigma_0$ is known, {\blue reasonable approximations} of these characteristics are
readily calculated from $\Sigma_0$. By Slutsky's Theorem, these estimates can substitute for the true ones in
the asymptotic variance $\sigma^2(c)$ {\em and} the
centering term  $qTc^2_p s_1^2$. The
test~\eqref{eq:Gqtest} still applies.

However, the population covariance matrix $\Sigma_0$ is in general unknown and the situation becomes
  challenging  as estimating a general $\Sigma_0$ is somehow out of
  reach without specific assumptions on its structure. Luckily, as
  observed previously, we only need  consistent estimates of
  the three characteristics.
  First of all, in the setting of Theorem~\ref{MainThmMult} and under
  the null, it is not difficult to find  {\em consistent} estimators for these characteristics, thus a consistent estimator of the limiting
  variance $\sigma^2(c)$. The situation is much more intricate for the
  centering term $qTc^2_p s_1^2$.
  Suppose ${\hat s}_1^2$ is a consistent estimator of
  $s_1^2$. Plugging it into the centering term leads to
  \begin{equation}
    \label{eq:decomp}
   G_{q,1}:= G_q - qT c^2_p  \hat{s}_1^2  =
    \left\{ G_q - qT c^2_p s_1^2  \right\}  +
    qT c^2_p     \left\{ s_1^2 -\hat{s}_1^2 \right\}~.
  \end{equation}
  Because of the multiplication by $T$ here,
  the asymptotic distribution would remain the same only if
  the estimation error   $\left\{\hat{s}_1^2   - s_1^2\right\}~$
  is of order $o_P(1/T)$.
  This is however not the case and in general the error is exactly of
  the order $O_p(1/T)$ and $T \left\{\hat{s}_1^2   -
    s_1^2\right\}$ converges to some other normal distribution.

  Our method is  as follows. First we establish the joint asymptotic
  distribution of $G_q-qTc_p^2s_1^2$ and
  $p \left\{\hat{s}_1^2   -    s_1^2\right\}$
  for a natural estimator $\hat{s}_1^2$.  This result extends Theorem
  \ref{MainThmMult} which addresses  the statistic $G_q-qTc_p^2s_1^2$ only. Next, the asymptotic null distribution of the ``feasible" test statistic $G_{q,1}$ is readily obtained as a simple consequence.

Precisely, consider the sample covariance matrix
  $\widehat{\Sigma}_0=\frac1T\sum_{t=1}^T\x_t\x_t^*$ and define the
  natural estimators of $s_1$ and $s_2$ as
  \[    {\hat s}_1 = \frac1p \Tr (\widehat{\Sigma}_0),   \quad {\hat s}_2 = \frac1p \Tr (\widehat{\Sigma}_0^2).
  \]
\begin{theorem}\label{Thm:s1}
  Assume the same conditions as in Theorem~\ref{MainThmMult}, then when  $\x_t=\Sigma_0^{1/2}\z_t$, we have
  \[
  \arraycolsep=1pt\def\arraystretch{1.8}
  \left(
	\begin{array}{c}
      p\lb \hat{s}_1^2-s_1^2\rb\\
      G_q - {qTc_p^2}{s_1^2}
	\end{array}\right)\xrightarrow{d}
	\]
  \[
  \arraycolsep=1pt\def\arraystretch{1.8}
  \EN_2\left(
	\left(
      \begin{array}{c}
        0\\
      0
      \end{array}
	\right),~\left(
      \begin{array}{cc}
        4c(\nu_4-3)s_1^2s_{d,2}+8cs_1^2s_2 & 4qc^2(\nu_4-3)s_1^2s_{d,2}+8qc^2s_1^2s_2\\
        4qc^2(\nu_4-3)s_1^2s_{d,2}+8qc^2s_1^2s_2& \quad \sigma^2(c)
      \end{array}
	\right)
  \right),
  \]
  where the variance $\sigma^2(c)$ is given in \eqref{eq:sigma2}.
\end{theorem}

The proof of this theorem is relegated to Section
\ref{sec:proofs}. 

Applying Theorem~\ref{Thm:s1} to the decomposition \eqref{eq:decomp}, the following proposition establishes the asymptotic null distribution of the feasible statistic $G_{q,1}$. Second order terms of the mean and variance of $G_{q,1}$ are also provided to improve finite sample performance.  

\begin{proposition}\label{prop:Gq1}
	Assume the same conditions as in Theorem~\ref{Thm:s1} and the
    observations are real-valued, we have
    \begin{equation}
      \label{eq:Gq1}
        \arraycolsep=1pt\def\arraystretch{1.8}
      G_{q,1}=G_q-qTc_{p}^2\hat{s}_1^2\xrightarrow{d}\EN(0, \xi^2(c)
      ),
    \end{equation}
    where $\xi^2(c)=2qc^2s_2^2$. Meanwhile,
    \begin{gather*}
    \mathbb{E}\lb G_{q,1}\rb=-\frac{q}{T^2}\lb 2\Tr(\Sigma_0^2)+(\nu_4-3)\Tr(D^2(\Sigma_0))\rb, ~\mathbb{E}(\hat{s}_1)=\frac{1}{p}\Tr(\Sigma_0),~\\
    \var\lb G_{q,1}\rb =\frac{2q}{T^2}\Tr^2(\Sigma_0^2)+\frac{q}{T^3}\lb 2\Tr(\Sigma_0^2)+(\nu_4-3)\Tr(D^2(\Sigma_0))\rb^2+o(\frac{1}{T}),\\
\mathbb{E}(\hat{s}_2)=\frac{1}{p}\Tr(\Sigma_0^2)+\frac{1}{pT}\Tr^2(\Sigma_0)+\frac{1}{pT}\lb \Tr(\Sigma_0^2)+(\nu_4-3)\Tr(D^2(\Sigma_0))\rb.\\
    \end{gather*}
\end{proposition}

Now we aim at  consistent estimates for the unknown
quantity $s_2$ in the asymptotic variance $\xi^2(c)$. It is well known
that almost surely (\cite{Bai10}), 
\[ \hat{s}_1\rightarrow s_1,~\hat{s}_2\rightarrow s_2+cs_1^2.\]
Therefore
$\tilde{s}_2=\hat{s}_2-c_p\hat{s}_1^2$ is a strongly consistent
estimator of $s_2$.

In summary,
when $\Sigma_0$ is unknown, we obtain a  procedure for  testing  the
null hypothesis of white noise
(\ref{eq:H0}) as follows:
\begin{equation}
\label{eq:Gq1test0}
\text{\em  Reject $H_0$ if ~~}
\left\{ G_{q}-qTc_{p}^2\hat{s}_1^2 > Z_{\alpha}
  {\tilde \xi}  \right\}
\end{equation}
where
$ {\tilde \xi}^2  =
2qc_{p}^2\tilde{s}_2^2.
$

\subsection{Finite sample correction and estimation for  non-Gaussian  innovations}\label{subsec:nu4}

Although the test procedure \eqref{eq:Gq1test0} is already
practically  usable,  it can be further improved by finite sample
corrections provided in  Proposition~\ref{prop:Gq1} which are 
especially useful for non-Gaussian population
where $\nu_4\ne 3$.  To this goal,  it remains to obtain a consistent
estimate for
(i) the covariance  characteristic
\[
  s_{d,2}=   \frac1p  \sum_{i=1}^p d_{i}^2=\frac1p  \Tr(D^2(\Sigma_0)), 
\]
where $d_{i}=  \Sigma_{0,ii}$ is the $i$th diagonal element of $\Sigma_0$,
and (ii) 
the fourth moment $\nu_4$ of the innovations.

\medskip \noindent {\em (i) Estimation of $s_{d,2}$}. \quad 
By its very definition,
$d_i$ can be consistently estimated
by its sample counterpart
\[  \tilde d_{i}=     \frac1{T} \sum_{t=1}^T  x_{it}^2.
\]
It follows that a consistent estimator for $s_{d,2}$ is simply
${\tilde s}_{d,2} =p^{-1} \sum_{i=1}^p  {\tilde  d}_i^2 $.

\medskip \noindent {\em (ii) Estimation of $\nu_4$}. \quad 
This is again a non trivial problem which has not been
touched yet in the literature (to our best knowledge). 
In order to get rid of the role of the unknown cross-sectional covariance
matrix $\Sigma_0$, we adopt  the following splitting strategy:
the original data $\{\x_t,~t=1,\cdots,T\}$  are split into two halves of
length
$T_1$ and $T_2$, respectively  ($T=T_1+T_2$). 
Define the two corresponding sample cross-sectional covariance
matrices 
\begin{equation}
  \label{eq:cov}
   S_{n,1}=\frac{1}{T_1}\sum_{t=1}^{T_1}\x_{t}\x_{t}^*,\quad 
   S_{n,2}=\frac{1}{T_2}\sum_{t=1}^{T_2}\x_{t+T_1}\x_{t+T_1}^*.
\end{equation}
This yields the corresponding $F$-ratio, or Fisher matrix,  $F_n =S_{n,1}^{-1}S_{n,2}$. 
Observe that this matrix does not depend on the value of the
cross-sectional covariance $\Sigma_0$ so that in what follows we can
assume $\Sigma_0=I$. 

Let $(\lambda_j)_{1\le j\le p}$ be the eigenvalues of $F_n$. Define
$K$ test functions  $f_k(x)=\log(a_k + b_k x)$ where 
$(a_k,b_k)_{1\le k\le K}$ are some positive constants. 
For each $k$, we have an eigenvalue  statistic of the Fisher matrix
\[    X_{T,k}  = f_k(\lambda_1)+\cdots+ f_k(\lambda_p) - p \int f_k(x) 
dF_{c_{p,1},c_{p,2}} (x)~, 
\] 
where $c_{p,i}= p/T_i$ ($i=1,2$) and $F_{c,c'}$ is the limiting
\ggai{Wachter} distribution with index $(c,c')$, see formula (3.1) in \cite{Zheng12}.
It is proved on page 452 of the reference, when $p,T_1,T_2$ grow
proportionally to infinity,
\begin{equation}
  \label{eq:reg}
  X_{T,k} =  u_{T,k} + v_{T,k} \nu_4 +\varepsilon_{T,k}    ~,
\end{equation}
where $\{u_{T,k}, v_{T,k}\}$ are  constants depending on 
$\{c_{p_i}\}$ and $(a_k,b_k)$, and $\{\varepsilon_{T,k}\} $ is a
centered and asymptotically Gaussian error. 
Then the least squares estimator of $\nu_4$ using the above regression
model leads to a consistent estimate, say ${\hat\nu}_4$ for the
unknown parameter.

Under the null hypothesis, the observations are independent. We may
repeat this estimation procedure, say $B$ times, by taking $B$ random splits of
the initial sample.  The final estimate of $\nu_4$ is then taken to be
the average of the estimates $\{{\hat\nu}_{4,b}\}_{1\le b\le B}$.

Finally we can implement the following test procedure with finite sample correction for the 
null hypothesis of white noise \eqref{eq:H0}:

\begin{equation}
  \label{eq:Gq1test}
  \text{\em  Reject $H_0$ if ~~}
  \left\{ G_{q,1}^*=G_{q}-qTc_{p}^2\hat{s}_1^2+\frac{1}{T}\cdot qc_p\lb 2\tilde{s}_2+(\hat{\nu}_4-3)\tilde{s}_{d,2}\rb > Z_{\alpha}
    {\hat \xi}  \right\}
\end{equation}
where
\begin{gather*}
 {\hat \xi}^2  =
2qc_{p}^2\tilde{s}_2^2+\frac{1}{T}\cdot qc_p^2\lb 2\tilde{s}_2+(\hat{\nu}_4-3)\tilde{s}_{d,2}\rb^2
\end{gather*}
with the above estimator $\hat \nu_4$ for the fourth moment. Note that
the estimation procedure proposed above for $\nu_4$ is only feasible
when $p<T$, thus we can only implement test \eqref{eq:Gq1test} when
$p<T$. However, our primary test statistic is $G_{q,1}$ in
\eqref{eq:Gq1test0} which doesn't require estimation of $\nu_4$. In
fact simulation results  in Section~\ref{subsec:comp} and
\ref{subsec:VMAComp} show that the statistic $G_{q,1}$ already
performs well. Therefore we can directly use $G_{q,1}$ when $p>T$.

\subsection{Tests when the observations are complex-valued}\label{subsec:complextests}
To proceed, we first define $\x_t = \Sigma_0^{1/2}\z_t$ where $\z_t$
is a {\em proper} complex random vector, and $\Sigma_0^{1/2}$ is such
that $\Sigma_0^{1/2}$ is Hermitian with $\Sigma_0 =
\Sigma_0^{1/2}(\Sigma_0^{1/2})^*$ ({\em Properness} of a complex
random vector $\z_t$ means that $\E(\z_t\z_t^T) = 0$). We immediately have
\[0 = \E(\z_t\z_t^T) = \E(z_{it}^2)I_p,\]
so that $\E(z_{it}^2) = 0$ for all $i=1,\ldots,p$ and $t=1,\ldots,T$. It also implies that $b = |\E(z_{it}^2)|^2 = 0$.
Since $\x_t = \Sigma_0^{1/2}\z_t$, we have 
\[\E(\x_t\x_t^T) = \E(\Sigma_0^{1/2}\z_t\z_t^T\Sigma_0^{T/2}) = 0,\]
 so that we are also assuming an observed vector $\x_t$ is proper. 

From Corollary \ref{Cor:Gq}, since $b=0$ from the properness of $\z_t$, the asymptotic covariance of $G_q$ is then
\[\var(G_q) \rightarrow qc^2s_2^2 + 4q^2c^3s_1^2[(\nu_4-2)s_{d,2}-s_2' + 2s_{r,2}],\]
where $s_2' = \lim_{p\rightarrow \infty}\Tr(\Sigma_0\Sigma_0^T)/p$, $s_{r,2} = \lim_{p\rightarrow \infty} \Tr(\Re^2(\Sigma_0))/p$, with $\Re(A) = (\Re(a_{ij}))$, the matrix of the real parts of all entries in $A$. 

Using Lemma~1.1 of the supplemental paper~\cite{supp},
defining $\Im(A) = (\Im(a_{ij}))$ to be the matrix of the imaginary parts of all entries in $A$, we have
\begin{align*}
  2\Tr(\Re^2(\Sigma_0)) - \Tr(\Sigma_0\Sigma_0^T) &= 2\Tr(\Sigma_0\Re(\Sigma_0)) - \Tr(\Sigma_0(\Re(\Sigma_0) - i\Im(\Sigma_0)))\\
  &= \Tr(\Sigma_0(\Re(\Sigma_0) + i\Im(\Sigma_0))) = \Tr(\Sigma_0^2), 
\end{align*}
so that $2s_{r,2} - s_2' = s_2$. The asymptotic variance for $G_q$ is then
\[\var(G_q) \rightarrow \sigma^2(c) = qc^2s_2^2 + 4q^2c^3s_1^2[(\nu_4-2)s_{d,2} + s_2],\]
which can be estimated consistently using the estimators suggested in Section \ref{subsec:unknownSigma}.

\subsection{Testing power of $G_{q,1}$} In this section, we look into
the power function of the tests when an alternative hypothesis $H_1$
is specified. Here we assume that under $H_1$, the observations
$\x_1,\cdots,\x_T$ follows from a $p\times 1$ real-valued
$p$-dimensional first-order vector moving average  process, VMA(1) in
short,  of the form
\begin{equation}\label{eq:H1}
  H_1:~\x_t=A_0\z_t+A_1\z_{t-1},
\end{equation}
where $A_0$, $A_1$ are $p\times p$ coefficient matrices. Now we only
consider the asymptotic behavior of our test statistic $G_q$ and
$G_{q,1}$ when $q=1$ since higher order autocorrelations of $\x_t$ are
null under both $H_0$ and $H_1$.

Denote 
$$\widetilde{\Sigma}_0=A_0^*A_0,\quad\widetilde{\Sigma}_1=A_1^*A_1,\quad\widetilde{\Sigma}_{01}=A_0^*A_1,$$
we characterize the joint limiting distribution of $\hat{s}_1^2$ and
$G_1$ under the VMA(1) alternative \eqref{eq:H1} as follows.

\begin{theorem}\label{Thm:jointH1}
	Assume that
  \begin{enumerate}
	\item $\{{\bf z}_t\}$ is a sequence of real-valued independent $p\times 1$
	random vectors with independent components
	$\z_t=(z_{it})$
	satisfying
	$\E z_{it}=0,~\E z_{it}^2=1$ and $\E z_{it}^4=\nu_4<\infty$;
	\item $\widetilde{\Sigma}_0$, $\widetilde{\Sigma}_1$ and $\widetilde{\Sigma}_{01}\widetilde{\Sigma}_{01}^*$ all have bounded spectral norm and for integers $i,j,k\geq 0$, $1\leq i+j+k\leq 4$, the limits 
$	\lim_{T\rightarrow\infty}\frac{1}{T}\Tr\lb\widetilde{\Sigma}_0^{i}\widetilde{\Sigma}_1^j\widetilde{\Sigma}_{01}^k\rb$ 
	exist;
		\item (Mar\v{c}enko-Pastur  regime).
	The dimension $p$ and the sample size $T$ grow  to infinity in
	a related way such that     $c_p:=p/T\rightarrow c>0$.
\end{enumerate}
Then under the VMA(1) alternative \eqref{eq:H1}, the joint limiting distribution of
the $G_1$ and $\hat{s}_1^2$ is
\[
\lb
\begin{array}{cc}
\sigma_G^2&\sigma_{GS}\\
\sigma_{GS}&\sigma_S^2
\end{array}
\rb^{-1/2}\lb
\begin{array}{c}
G_1-\mu_G\\
Tc_p^2\hat{s}_1^2-\mu_S
\end{array}
\rb\xrightarrow{d} \EN_2\lb
{\bf 0},~ {I}_2\rb,
\]
where
\begin{align*}
\mu_G&=\frac{1}{T}\Tr^2\lb\widetilde{\Sigma}_0+\widetilde{\Sigma}_1\rb+\Tr\lb\widetilde{\Sigma}_0\widetilde{\Sigma}_1\rb+\frac{2}{T}\Tr^2\lb\widetilde{\Sigma}_{01}\rb\\
&+\frac{1}{T}\left[ \Tr\lb\widetilde{\Sigma}_0\widetilde{\Sigma}_1\rb+(\nu_4-3)\Tr\lb D(\widetilde{\Sigma}_0)D(\widetilde{\Sigma}_1)\rb\right],\\
\mu_S&=\frac{1}{T}\Tr^2\lb\widetilde{\Sigma}_0+\widetilde{\Sigma}_1\rb+\frac{4}{T^2}\Tr\lb\widetilde{\Sigma}_{01}\widetilde{\Sigma}_{01}^*\rb\\
&+\frac{1}{T^2}\left[2\Tr\lb\widetilde{\Sigma}_0+\widetilde{\Sigma}_1\rb^2+(\nu_4-3)\Tr\lb D^2(\widetilde{\Sigma}_0+\widetilde{\Sigma}_1)\rb\right],
\end{align*}

\begin{align*}
\sigma_S^2=&~\frac{4}{T^3}\Tr^2\lb\widetilde{\Sigma}_0+\widetilde{\Sigma}_1\rb\left[2\Tr\lb\widetilde{\Sigma}_0+\widetilde{\Sigma}_1\rb^2+(\nu_4-3)\Tr\lb D^2\lb \widetilde{\Sigma}_0+\widetilde{\Sigma}_1\rb\rb\right]\\ &
+\frac{16}{T^3}\Tr^2\lb\widetilde{\Sigma}_0+\widetilde{\Sigma}_1\rb\Tr\lb\widetilde{\Sigma}_{01}\widetilde{\Sigma}_{01}^*\rb+ R_n,
\end{align*}
and
{\small
\begin{align*}
~&\sigma_G^2=\frac{4}{T^3}\Tr^2\lb \widetilde{\Sigma}_0+\widetilde{\Sigma}_1\rb\left[2\Tr\lb \widetilde{\Sigma}_0+\widetilde{\Sigma}_1\rb^2+(\nu_4-3)\Tr\lb D^2\lb \widetilde{\Sigma}_0+\widetilde{\Sigma}_1\rb\rb\right]\\ &
+\frac{8}{T^2}\Tr\lb \widetilde{\Sigma}_0+\widetilde{\Sigma}_1\rb \left[2\Tr\lb \widetilde{\Sigma}_0\widetilde{\Sigma}_1\lb\widetilde{\Sigma}_0+\widetilde{\Sigma}_1\rb\rb+(\nu_4-3)\Tr\lb D\lb\widetilde{\Sigma}_0\widetilde{\Sigma}_1\rb D\lb \widetilde{\Sigma}_0+\widetilde{\Sigma}_1\rb \rb\right]\\ &
+\frac{2}{T^2}\Tr^2\lb \widetilde{\Sigma}_0^2+\widetilde{\Sigma}_1^2\rb+\frac{6}{T^2}\Tr^2\lb\widetilde{\Sigma}_0\widetilde{\Sigma}_1 \rb+\frac{4}{T}\left[ 2\Tr\lb \widetilde{\Sigma}_0\widetilde{\Sigma}_1\rb^2+(\nu_4-3)\Tr\lb D^2\lb \widetilde{\Sigma}_0\widetilde{\Sigma}_1\rb\rb\right]\\ &
+\frac{8}{T^2}\Tr\lb \widetilde{\Sigma}_{01}\widetilde{\Sigma}_{01}^*\rb\Tr\lb\widetilde{\Sigma}_0^2+\widetilde{\Sigma}_1^2\rb+\frac{16}{T^2}\Tr\lb \widetilde{\Sigma}_{01}\widetilde{\Sigma}_1\rb\Tr\lb \widetilde{\Sigma}_{01}\widetilde{\Sigma}_0\rb\\ &
+\frac{16}{T^2}\Tr\lb \widetilde{\Sigma}_0+\widetilde{\Sigma}_1\rb \left[ \Tr\lb \widetilde{\Sigma}_{01}^*\widetilde{\Sigma}_{01}\widetilde{\Sigma}_0\rb+ \Tr\lb \widetilde{\Sigma}_{01}\widetilde{\Sigma}_{01}^*\widetilde{\Sigma}_1\rb\right]\\ &
+\frac{16}{T^2}\Tr(\widetilde{\Sigma}_{01})\left[\Tr\lb\widetilde{\Sigma}_0^2\widetilde{\Sigma}_{01}^*\rb+\Tr\lb\widetilde{\Sigma}_1^2\widetilde{\Sigma}_{01}\rb+2\Tr\lb \widetilde{\Sigma}_1\widetilde{\Sigma}_{01}\widetilde{\Sigma}_0\rb\right]\\ &
+\frac{4}{T}\Tr\lb \widetilde{\Sigma}_{01}^*\widetilde{\Sigma}_{01}\widetilde{\Sigma}_0^2+\widetilde{\Sigma}_{01}\widetilde{\Sigma}_{01}^*\widetilde{\Sigma}_1^2+2\widetilde{\Sigma}_{01}^*\widetilde{\Sigma}_1\widetilde{\Sigma}_{01}\widetilde{\Sigma}_0\rb\\ &
+\frac{16}{T^3}\Tr^2\lb \widetilde{\Sigma}_0+\widetilde{\Sigma}_1 \rb\Tr\lb \widetilde{\Sigma}_{01}\widetilde{\Sigma}_{01}^*\rb+\frac{16}{T^3}\Tr^2(\widetilde{\Sigma}_{01})\Tr\lb \widetilde{\Sigma}_0+\widetilde{\Sigma}_1\rb^2\\ &
+\frac{32}{T^3}\Tr\lb\widetilde{\Sigma}_0+\widetilde{\Sigma}_1\rb\Tr\lb\widetilde{\Sigma}_{01}\rb\Tr\lb \widetilde{\Sigma}_{01}\lb\widetilde{\Sigma}_0+\widetilde{\Sigma}_1\rb\rb\\ &
+\frac{4}{T}\Tr\lb\widetilde{\Sigma}_{01}\widetilde{\Sigma}_{01}^*\widetilde{\Sigma}_{01}^*\widetilde{\Sigma}_{01}\rb+\frac{12}{T^2}\Tr^2\lb\widetilde{\Sigma}_{01}\widetilde{\Sigma}_{01}^*\rb+\frac{16}{T^2}\Tr(\widetilde{\Sigma}_{01})\Tr\lb\widetilde{\Sigma}_{01}\widetilde{\Sigma}_{01}^*\widetilde{\Sigma}_{01}^* \rb\\ &
+\frac{16}{T^3}\Tr^2\lb \widetilde{\Sigma}_{01}\rb \left[ \Tr\lb\widetilde{\Sigma}_{01}\rb^2+2\Tr\lb\widetilde{\Sigma}_{01}\widetilde{\Sigma}_{01}^*\rb+(\nu_4-3)\Tr\lb D^2\lb \widetilde{\Sigma}_{01}\rb\rb\right]+\frac{8}{T^2}\Tr^2\lb \widetilde{\Sigma}_1 \widetilde{\Sigma}_{01}\rb\\ &
+\frac{16}{T^3}\Tr\lb\widetilde{\Sigma}_{01}\rb\Tr\lb\widetilde{\Sigma}_0+\widetilde{\Sigma}_1\rb\left[2\Tr\lb \widetilde{\Sigma}_{01}\lb\widetilde{\Sigma}_0+\widetilde{\Sigma}_1\rb\rb+(\nu_4-3)\Tr\lb D\lb \widetilde{\Sigma}_{01} \rb D\lb \widetilde{\Sigma}_0+\widetilde{\Sigma}_1 \rb\rb\right]\\ &
+\frac{8}{T^2}\Tr^2\lb \widetilde{\Sigma}_0 \widetilde{\Sigma}_{01}\rb +\frac{16}{T^2}\Tr\lb\widetilde{\Sigma}_{01}\rb\left[2\Tr\lb\widetilde{\Sigma}_0\widetilde{\Sigma}_1\widetilde{\Sigma}_{01}\rb+(\nu_4-3)\Tr\lb D\lb\widetilde{\Sigma}_0\widetilde{\Sigma}_1 \rb D\lb \widetilde{\Sigma}_{01}\rb\rb\right]+R_n,
\end{align*}
}
{\small
\begin{align*}
&~\sigma_{GS}~=\frac{4}{T^3}\Tr^2\lb\widetilde{\Sigma}_0+\widetilde{\Sigma}_1\rb\left[2\Tr\lb \widetilde{\Sigma}_0+\widetilde{\Sigma}_1\rb^2+(\nu_4-3)\Tr\lb D^2\lb \widetilde{\Sigma}_0+\widetilde{\Sigma}_1\rb\rb\right]\\ &
+\frac{4}{T^2}\Tr\lb\widetilde{\Sigma}_0+\widetilde{\Sigma}_1\rb\left[2\Tr\lb \widetilde{\Sigma}_0\widetilde{\Sigma}_1\lb\widetilde{\Sigma}_0+\widetilde{\Sigma}_1\rb\rb+(\nu_4-3)\Tr\lb D\lb\widetilde{\Sigma}_0\widetilde{\Sigma}_1\rb D\lb \widetilde{\Sigma}_0+\widetilde{\Sigma}_1\rb \rb\right]\\ &
+\frac{8}{T^2}\Tr\lb\widetilde{\Sigma}_0+\widetilde{\Sigma}_1\rb\left[ \Tr\lb \widetilde{\Sigma}_{01}^*\widetilde{\Sigma}_{01}\widetilde{\Sigma}_0\rb+ \Tr\lb \widetilde{\Sigma}_{01}\widetilde{\Sigma}_{01}^*\widetilde{\Sigma}_1\rb\right]+\frac{16}{T^3}\Tr^2\lb \widetilde{\Sigma}_0+\widetilde{\Sigma}_1 \rb\Tr\lb \widetilde{\Sigma}_{01}\widetilde{\Sigma}_{01}^*\rb\\ &
+\frac{8}{T^3}\Tr\lb\widetilde{\Sigma}_{01}\rb\Tr\lb\widetilde{\Sigma}_0+\widetilde{\Sigma}_1\rb\left[2\Tr\lb \widetilde{\Sigma}_{01}\lb\widetilde{\Sigma}_0+\widetilde{\Sigma}_1\rb\rb+(\nu_4-3)\Tr\lb D\lb \widetilde{\Sigma}_{01} \rb D\lb \widetilde{\Sigma}_0+\widetilde{\Sigma}_1 \rb\rb\right]\\ &
+\frac{16}{T^3}\Tr\lb\widetilde{\Sigma}_0+\widetilde{\Sigma}_1\rb\Tr\lb\widetilde{\Sigma}_{01}\rb\Tr\lb \widetilde{\Sigma}_{01}\lb\widetilde{\Sigma}_0+\widetilde{\Sigma}_1\rb\rb+R_n.
\end{align*}}
Here the $R_n$'s, possibly different, represent remainders which have
smaller orders than the other terms listed  in $\sigma_S^2$,
$\sigma_G^2$ and $\sigma_{GS}$, respectively.
\end{theorem}

The proof of this theorem is relegated to
Section~\ref{sec:proofs}. Similarly, applying
Theorem~\ref{Thm:jointH1} to the decomposition~\eqref{eq:decomp}, the
following proposition  establishes the asymptotic distribution of our
test statistic $G_{q,1}$ under the VMA(1) alternative~\eqref{eq:H1} when $q=1$. 

\begin{proposition}\label{prop:H1}
	Assume the same conditions as in Theorem~\ref{Thm:jointH1}, when  $\x_t=A_0\z_t+A_1\z_{t-1}$ and the observables are real-valued, we have
	\begin{equation}\label{eq:G11H1}
	\sigma_{G_{1,1}}^{-1}\lb G_1-Tc_p^2\hat{s}_1^2-\mu_{G_{1,1}}\rb\xrightarrow{d} \EN\lb 0,~ 1\rb,
	\end{equation}
	where
	\begin{align*}
	\mu_{G_{1,1}}=&\Tr\lb\widetilde{\Sigma}_0\widetilde{\Sigma}_1\rb+\frac{2}{T}\Tr^2\lb\widetilde{\Sigma}_{01}\rb+\frac{1}{T}\left[ \Tr\lb\widetilde{\Sigma}_0\widetilde{\Sigma}_1\rb+(\nu_4-3)\Tr\lb D(\widetilde{\Sigma}_0)D(\widetilde{\Sigma}_1)\rb\right]\\ &
-\frac{4}{T^2}\Tr\lb\widetilde{\Sigma}_{01}\widetilde{\Sigma}_{01}^*\rb	-\frac{1}{T^2}\left[2\Tr\lb\widetilde{\Sigma}_0+\widetilde{\Sigma}_1\rb^2+(\nu_4-3)\Tr\lb D^2(\widetilde{\Sigma}_0+\widetilde{\Sigma}_1)\rb\right],
\end{align*}
{\small
\begin{align*}
~&\sigma_{G_{1,1}}^2=\frac{2}{T^2}\Tr^2\lb \widetilde{\Sigma}_0^2+\widetilde{\Sigma}_1^2\rb+\frac{4}{T}\left[ 2\Tr\lb \widetilde{\Sigma}_0\widetilde{\Sigma}_1\rb^2+(\nu_4-3)\Tr\lb D^2\lb \widetilde{\Sigma}_0\widetilde{\Sigma}_1\rb\rb\right]\\ 
	+&\frac{6}{T^2}\Tr^2\lb\widetilde{\Sigma}_0\widetilde{\Sigma}_1 \rb+\frac{8}{T^2}\Tr\lb \widetilde{\Sigma}_{01}\widetilde{\Sigma}_{01}^*\rb\Tr\lb\widetilde{\Sigma}_0^2+\widetilde{\Sigma}_1^2\rb+\frac{16}{T^2}\Tr\lb \widetilde{\Sigma}_{01}\widetilde{\Sigma}_1\rb\Tr\lb \widetilde{\Sigma}_{01}\widetilde{\Sigma}_0\rb\\ 
	+&\frac{16}{T^2}\Tr(\widetilde{\Sigma}_{01})\left[\Tr\lb\widetilde{\Sigma}_0^2\widetilde{\Sigma}_{01}^*\rb+\Tr\lb\widetilde{\Sigma}_1^2\widetilde{\Sigma}_{01}\rb+2\Tr\lb \widetilde{\Sigma}_1\widetilde{\Sigma}_{01}\widetilde{\Sigma}_0\rb\right]\\ 
+&\frac{4}{T}\Tr\lb \widetilde{\Sigma}_{01}^*\widetilde{\Sigma}_{01}\widetilde{\Sigma}_0^2+\widetilde{\Sigma}_{01}\widetilde{\Sigma}_{01}^*\widetilde{\Sigma}_1^2+2\widetilde{\Sigma}_{01}^*\widetilde{\Sigma}_1\widetilde{\Sigma}_{01}\widetilde{\Sigma}_0\rb+\frac{16}{T^3}\Tr^2(\widetilde{\Sigma}_{01})\Tr\lb \widetilde{\Sigma}_0+\widetilde{\Sigma}_1\rb^2\\ 
+&\frac{4}{T}\Tr\lb\widetilde{\Sigma}_{01}\widetilde{\Sigma}_{01}^*\widetilde{\Sigma}_{01}^*\widetilde{\Sigma}_{01}\rb+\frac{12}{T^2}\Tr^2\lb\widetilde{\Sigma}_{01}\widetilde{\Sigma}_{01}^*\rb+\frac{16}{T^2}\Tr(\widetilde{\Sigma}_{01})\Tr\lb\widetilde{\Sigma}_{01}\widetilde{\Sigma}_{01}^*\widetilde{\Sigma}_{01}^* \rb\\ 
+&\frac{16}{T^3}\Tr^2\lb \widetilde{\Sigma}_{01}\rb \left[ \Tr\lb\widetilde{\Sigma}_{01}\rb^2+2\Tr\lb\widetilde{\Sigma}_{01}\widetilde{\Sigma}_{01}^*\rb+(\nu_4-3)\Tr\lb D^2\lb \widetilde{\Sigma}_{01}\rb\rb\right]+\frac{8}{T^2}\Tr^2\lb \widetilde{\Sigma}_1 \widetilde{\Sigma}_{01}\rb\\ 
+&\frac{8}{T^2}\Tr^2\lb \widetilde{\Sigma}_0
  \widetilde{\Sigma}_{01}\rb+\frac{16}{T^2}\Tr\lb\widetilde{\Sigma}_{01}\rb\left[2\Tr\lb\widetilde{\Sigma}_0\widetilde{\Sigma}_1\widetilde{\Sigma}_{01}\rb+(\nu_4-3)\Tr\lb
  D\lb\widetilde{\Sigma}_0\widetilde{\Sigma}_1 \rb D\lb
  \widetilde{\Sigma}_{01}\rb\rb\right]+R_n.	
\end{align*}}
Here $R_n$ represents  a remainder  of  smaller order than  the other terms listed  in $\sigma^2_{G_{1,1}}$ .
\end{proposition}
Notice that if $A_1={\bf 0}$, $\widetilde\Sigma_1=0$  and
$\widetilde\Sigma_{01}=0$. Then Theorem~\ref{Thm:jointH1} and
Proposition~\ref{prop:H1} reduce to Theorem~\ref{Thm:s1} and
Proposition~\ref{prop:Gq1}, respectively. 

Acturally, under the VMA(1) alternative \eqref{eq:H1} with $q=1$, we have almost surely, $\tilde{\xi}=\sqrt{2} c_p\tilde{s}_2\rightarrow \xi_0$ as $p,T\rightarrow \infty$, where
\begin{equation}\label{eq:xiH1}
\xi_0=\lim_{T\rightarrow\infty}\sqrt{2}\left[\frac{1}{T}\Tr\lb \widetilde{\Sigma}_0^2+\widetilde{\Sigma}_1^2\rb+\frac{2}{T}\Tr\lb \widetilde{\Sigma}_{01}\widetilde{\Sigma}_{01}^*\rb+\frac{2}{T^2}\Tr^2\lb\widetilde{\Sigma}_{01}\rb\right].
\end{equation}
With Proposition~\ref{prop:Gq1} and \ref{prop:H1}, the power function
of the test  \eqref{eq:Gq1test0}  is then easily derived. 

\begin{proposition}\label{prop:power}
	Assume the same conditions as in Theorem~\ref{Thm:jointH1}, then under $H_1:~\x_t=A_0\z_t+A_1\z_{t-1}$,  $ \mbox{as } p,T\rightarrow\infty,$ the power function 
	\[\beta_{\alpha}=Pr\lb\left. G_1- Tc_p^2\hat{s}_1^2> Z_\alpha\tilde{\xi}~\right\rvert H_1\rb\rightarrow Pr\lb Z>Z_{\alpha}\frac{\xi_0}{\widetilde{\sigma}_{G_{1,1}}}-\frac{\widetilde{\mu}_{G_{1,1}}}{\widetilde{\sigma}_{G_{1,1}}} \rb,
	\]	
	where $Z$ represents a standard normal random variable, $Z_\alpha$ is the upper-$\alpha$ quantile of the standard normal
	distribution, $\widetilde{\mu}_{G_{1,1}}$ and $\widetilde{\sigma}_{G_{1,1}}$ are limits of $\mu_{G_{1,1}}$ and $\sigma_{G_{1,1}}$ as $T\rightarrow \infty$.
\end{proposition}

In fact, under $H_1$, when $\widetilde{\Sigma}_0$ and
$\widetilde{\Sigma}_1$ have  bounded spectral norm, both $\widetilde{\sigma}_{G_{1,1}}$  and $\xi_0$ are of order $O(1)$  and $0<\frac{\xi_0}{\widetilde{\sigma}_{G_{1,1}}}\leq 1$, while the leading order term of $\widetilde{\mu}_{G_{1,1}}$ is $$\lim_{T\rightarrow \infty} \Tr\lb \widetilde{\Sigma}_0 \widetilde{\Sigma}_1\rb=\lim_{T\rightarrow \infty} \Tr\lb A_1A_0^*A_0A_1^*\rb>0.$$
Consequently, 
\begin{itemize}
	\item[Case 1.] If $ \Tr\lb \widetilde{\Sigma}_0 \widetilde{\Sigma}_1\rb$ diverges as $T\rightarrow\infty$, then the power function $\beta_{\alpha}\rightarrow 1$ ;
	\item[Case 2.]  If $\Tr\lb \widetilde{\Sigma}_0
      \widetilde{\Sigma}_1\rb$ is of order $\Omega(1)$ (bounded from below and above),  then the power function
      $\beta_\alpha$ converges to the constant \\
      $\beta=Pr\lb
      Z>Z_{\alpha}\frac{\xi_0}{\widetilde{\sigma}_{G_{1,1}}}-\frac{\widetilde{\mu}_{G_{1,1}}}{\widetilde{\sigma}_{G_{1,1}}}\rb$ and $\alpha\le \beta\le 1$.
\end{itemize}

Therefore, as expected the asymptoic power of the  test
~\eqref{eq:Gq1test0} under the VMA(1) alternative \eqref{eq:H1}
depends on the eigenstructure of the coefficient matrix $A_1$.
{ To illustrate, assume that 
  (i) $A_0A_0^*$ is of rank $r_{0p} \sim r p$ for some constant $0
  <r\le 1$;  (ii) $A_0A_0^*$ is of rank $1  \ll r_{1p} \ll p$, for example
  $r_{1p}\sim r' \log p $ for some constant $r'>0$, and that the non-null
  eigenvalues of both matrices are of order $\Omega(1)$. Then $
  \Tr\lb \widetilde{\Sigma}_0 \widetilde{\Sigma}_1\rb \sim r'' r_{1p} \to
  \infty$ for some constant $r''>0$, and the asymptotic power is equal to 1 (Case 1).
  If instead,  $r_{1p} = \Omega(1) $, then the asymptotic power can be
  smaller than 1 (Case 2).  Both situations correspond to a low-rank
  alternative for $A_1$,  with exploding ranks in Case 1  and constant
  order ranks in
  Case 2.

Finally, as here the alternative is a VMA(1), one would expect that
$G_{q,1}$ with $q>1$ might have  smaller power than $G_{1,1}$. This is 
indeed true because  $\widetilde{\mu}_{G_{q,1}}$ remains the same with $\widetilde{\mu}_{G_{1,1}}$ under $H_1$,  while $\widetilde{\sigma}_{G_{q,1}}$ is larger than $\widetilde{\sigma}_{G_{1,1}}$ and $\xi_0$ increases with $q$ as well.
}

\section{Simulation experiments}\label{sec:simul}

Most of the experiments of this section are designed
in order to  compare the test procedures in \eqref{eq:Gqtest} and \eqref{eq:Gq1test0} based
on the statistics $G_q$ and $G_{q,1}$,
with two well known classical white noise tests, namely the Hosking test
\citep{Hosking80} and the Li-McLeod test \citep{Li81}.

To introduce the Hosking and  Li-McLeod tests and using  their notations,
consider a $p$-dimensional VARMA($u,v$) process of the form
\[\x_t-\Phi_1\x_{t-1}-\cdots-\Phi_u\x_{t-u}=\a_t-\Theta_1\a_{t-1}-\cdots-\Theta_v\a_{t-v},\]
where $\a_t$ is a $p-$dimensional white noise with mean zero and
variance $\Sigma$. Since $\x_t$ is observed, with an initial guess of
$u$ and $v$, by assuming $\a_t$ to be Gaussian, estimation of
parameters $\{\Phi,~\Theta\}$ is conducted by the method of maximum likelihood. The
initial estimates of $u$ and $v$ are further refined at the diagnostic
checking stage based on the autocovariance  matrices $\hat{C}_\tau$ of  the
residuals $\{\hat{\a}_t\}$:
\[\hat{C}_\tau=\frac 1T\sum_{t=\tau+1}^T \hat{\a}_t\hat{\a}_{t-\tau}^*,  \quad, \tau=0,1,2,\ldots.
\]
\citet{Hosking80} proposed the portmanteau statistic
\[\widetilde{Q}_{q}=T^2\sum_{\tau=1}^q \frac{1}{T-\tau}\mathrm{Tr}\lb\hat{C}_\tau^*\hat{C}_0^{-1}\hat{C}_\tau\hat{C}_0^{-1}\rb,\]
while
\citet{Li81} recommended  the use of the statistic
\[Q_q^*=T\sum_{\tau=1}^q \mathrm{Tr}\lb\hat{C}_\tau^*\hat{C}_0^{-1}\hat{C}_\tau\hat{C}_0^{-1}\rb+\frac{p^2q(q+1)}{2T}.\]
When $\{\x_t\}$ follows a  VARMA$(u,v)$ model, both $\widetilde{Q}_q$ and $Q_q^*$ converge to $\chi^2(p^2(q-u-v))$ distribution as $T\rightarrow \infty$, while the dimension $p$ remains fixed.

To compare with our multi-lag $q$ test statistics $G_q$ and $G_{q,1}$ when $\Sigma_0$ is either known or unknown, we set $u=v=0$.
 All tests use 5\% significance level and the critical regions of the three tests are as follows:
\gai{
 \begin{itemize}
 \item[(i)]$G_q$ when all the limiting parameters are  known as defined in   \eqref{eq:Gqtest} with $\alpha=5\%$;
 \item[(ii)]$G_{q,1}$ with estimated limiting parameters  as defined in   \eqref{eq:Gq1test0} with $\alpha=5\%$;
 \item[(iii)] Hosking  test: \quad
   $\displaystyle\left\{~\widetilde{Q}_{q}>\chi^2_{0.05, ~q p^2} ~\right\}$;
 \item[(iv)] Li-McLeod test: \quad
   $\displaystyle\left\{~Q_q^*>\chi^2_{0.05,~q p^2} ~\right\}$.
 \end{itemize}
}
\noindent Here $Z_{0.05}$ and $\chi^2_{0.05,~m}$ denote the upper-5\% quantile  of the standard normal distribution and the chi-squared
distribution with degrees of freedom $m$, respectively.
Empirical statistics are obtained  using 2000 independent
replicates. 

\subsection{Empirical sizes}\label{subsec:empiricalsize}

The data is generated as $\x_t=\Sigma_0^{1/2}\z_t$, with $\z_{t}$, $t=1,\cdots, T$ being independent and identically distributed. We adopt diverse settings for $\z_t$ and $\Sigma_0$ respectively to compare the sizes of four test statistics.

As for $\z_t$, we use two models to represent different distributions for $\z_t$:
\begin{itemize}
	\item [(I)]  $\z_t\sim \N_p({\bf 0},{\bf I}_p), i.i.d. ~t=1,\cdots, T$;
	\item[(II)] $\z_t$ with i.i.d. components $z_{it}\sim Gamma(4,0.5)-2, ~i=1,\cdots,p,~t=1,\cdots, T$, $\E(z_{it})=0$, $\var(z_{it})=1$, $\nu_4(z_{it})=4.5$;
\end{itemize}

As for $\Sigma_0$, we use two different models as follows.

\begin{itemize}
	\item[(III)] $\Sigma_0={\bf I}_p$;
	\item[(IV)] $\Sigma_0=\frac{4}{p}A_0A_0^*$, $A_0$ is $p\times p$ matrix with entries $a_{ij} ~ \sim U(-1,1)\text{ i.i.d..}$
\end{itemize}

Table \ref{Tab:cLSizeCom} compares the sizes of the four tests for two different $q$ when $\Sigma_0={\bf I}_p$. Cases when $p>T$ are not considered here since $\widetilde{Q}_q$ and $Q_q^*$ are not applicable then.

The main information from Table \ref{Tab:cLSizeCom} is that classical test procedures derived using large sample scheme, namely
by letting the sample size $T\rightarrow\infty$ while the dimension $p$ remains fixed, are heavily biased when the dimension $p$ is in fact
not negligible with respect to the sample size. To be more precise, these biases are clearly present when the dimension-to-sample ratio
$p/T$ is not ``small enough", say greater than 0.1. Such high-dimensional traps for classical procedures have already been reported
in other testing problems, see for example \cite{Bai09}  and \cite{WY13}.
Here we observe that the empirical sizes of the Hosking and the Li-McLeod
tests quickly degenerate to 0 as the ratio $p/T$ increases from 0.1 to
0.5. In other words, the critical values from their  $\chi^2_{q p^2}$
asymptotic limits are seemingly  {\em too large}. On the other hand, the statistics $G_q$ and $G_{q,1}$ have reasonable sizes when compared to the 5\% nominal level across all the tested $(p,T)$ combinations. Various $(p,T)$ combinations are accommodated to testify the adaptability of our test statistics, $G_q$ and $G_{q,1}$. Test sizes in both high and low dimension cases are shown in Table \ref{Tab:HDSize}. It can be seen that both $G_q$ and $G_{q,1}$ attain the nominal level accurately under various scenarios.

\subsection{Empirical powers and adjusted powers}\label{subsec:powercomparison}
In this section, we compare the empirical powers of the tests by
assuming  that $\x_t=\Sigma_0^{1/2}\y_t$, $\y_t$ follows a vector autoregressive process of order
1,
$$\x_t=\Sigma_0^{1/2}\y_t,~\y_t=A\y_{t-1}+\z_t,$$
where $A=a {\bf I}_p$, $\z_{t}\sim N_p({\bf 0},{\bf I}_p)$ being
independent of each other for $t=1,\cdots, T$. 
{First we check the power of two classic test procedures,
  $\widetilde{Q}_q$ and $Q^*_q$. Table~A 
  in the supplemental paper~\cite{supp} gives these empirical powers for
  $a=0.1$ and various combinations $(p,T)$.
}

From Table \ref{Tab:cLSizeCom} we know that the two classic tests
become seriously biased when the dimension $p$ is large compared to
the sample size $T$. Their sizes approach zero when $p/T$ becomes
larger. From Table~A 
of \cite{supp}, we see that due to such biased critical values used in $\widetilde{Q}_{q}$ and $Q_q^*$, their powers are driven downward. This is particularly severe when the ratio $p/T$ is larger than 0.5.

To explore more these two traditional tests, 
we also examine their 
{\em intrinsic powers}  when $\Sigma_0=I_p$. Namely, we
empirically find the 95 percentiles of
$\widetilde{Q}_{q}$ and $Q_q^*$ \ggai{under the null} and use these
values as the corrected critical value for the power comparison.
Empirical values are reported in Table~B of the supplemental paper~\cite{supp}. 
It is interesting to observe
that after such
correction, both $\widetilde{Q}_{q}$ and $Q_q^*$ show
very reasonable powers which
all increase to 1 when the
dimension and the sample size increases.Our test statistics $G_q$ and $G_{q,1}$ also maintain  comparably high power in all the tested
$(p,T)$ combinations. Table \ref{Tab:HDPower}  demonstrates the
feasibility of our test statistics under both high and low dimension
cases.
Interestingly enough, $G_{q,1}$ shows  slightly better power 
than $G_q$ under the present AR(1) alternative which is not
intuitive. 
Comparison with the Hosking and the Li-McLeod tests sheds new light on the superiority of our test statistics in both low and high dimensional cases.

%
\subsection{Why both the Hosking and the Li-Mcleod tests fail in high dimension}

The experiments here are designed to explore the reasons behind the
failure of the Hosking and the Li-McLeod tests in high dimension.
For the  test statistics
$\widetilde{Q}_q$ and  $Q_q^*$ as well as our test statistic $\phi_\tau$,
we  consider  their empirical mean, variance and the  95\% quantile,
say $\theta_{emp}$,
with their theoretical values predicted by their respective asymptotic
distributions (denoted as $\theta_{theo}$).
 As for the two classical tests,
we have often observed very large discrepancy between these values  so it
is more convenient to report the corresponding relative errors
$
  (\theta_{theo} -\theta_{emp})/ {\theta_{emp}}
$
(in percentage).
Empirical values are reported in 
Table~C  of the supplemental paper~\cite{supp}.
It clearly appears from this table that for both
statistics $\widetilde{Q}_q$ and  $Q_q^*$,
the  traditional asymptotic theory
severely overestimated their variances, that is
their empirical  means
are close to the degree of freedom $p^2(q-u-v)$ of the asymptotic
chi-squared distribution
while
their empirical variances are much smaller than  $2p^2(q-u-v)$
as suggested by the same  chi-squared limit.
This leads to an  inflated  95th  percentiles which,  although in a  lesser
proportion,  is enough to create a high downward-bias in the empirical
sizes of these two classical tests with high-dimensional data; See Table~\ref{Tab:cLSizeCom}.

\subsection{Comparison with other test statistics}\label{subsec:comp}

In this section, we compare our test statistics with some others in recent literature. \citet{CYZ17} proposed an omnibus test for vector white noise using the maximum absolute autocorrelations and cross-correlations  of the component series. Let
\[\hat{\Gamma}(k)=\{\hat{\rho}_{ij}(k)\}_{1\leq i,j\leq p}=\text{diag} \{\hat{\Sigma}(0)\}^{-1/2}\hat{\Sigma}(k)\text{diag} \{\hat{\Sigma}(0)\}^{-1/2}\]
be the sample autocorrelation matrix at lag $k$, where $\hat{\Sigma}(k)=\frac{1}{T}\sum_{t=1}^{T-k}\x_{t+k}\x_t^*$. Their test statistic $T_n$ is defined as
\[T_n=\max_{1\leq k\leq q}T_{n,k},\]
where $T_{n,k}=\max_{1\leq i,j\leq p} T^{1/2}|\hat{\rho}_{ij}(k)|$. Another test statistic $T_n^*$ is defined in the same manner as $T_n$, only that the time series principal component analysis proposed by \citet{ChangGuoYao_2015} is applied to the data $\{\x_t\}$ first.

Here we fix $p=20,~ T=100$ and adopt the spherical AR(1) process for power comparison, i.e. $\x_t=\Sigma_0^{1/2}\y_t,~\y_t=A\y_{t-1}+\z_t$, $A=a{\bf I}_p$, where $\z_t$ and $\Sigma_0$ follow different combinations of settings.
Power values of all the five test statistics, i.e. $G_q$, $G_{q,1}$,
$G_{q,1}^*$, $T_n$ and $T_n^*$, are compared when VAR coefficient $a$
grows from 0 to 0.5. Here $G_{q,1}^*$ is our test statistic with finte
sample correction as demonstrated in \eqref{eq:Gq1test}.  Empirical
statistics are obtained using 2000 independent replicates. Results are
shown in Fig.~\ref{Fig:SF2}. Notice that on these 
displays, $G_{q,1}$ and $G_{q,1}^*$ coincide almost everywhere showing
a high accuracy of the parameter estimates  used in $G_{q,1}^*$.

 It can be seen that our test statistics show better performance under this spherical AR(1) model setting. Designed via Frobenius norm of sample autocovariance matrices, the strength of our test statistics are fully demonstrated in such VAR(1) settings. While $T_n$ and $T_n^*$ are more adapted to settings where majority coordinates of the test sequence $\x_t$ or their linear transformations remain to be white noise, see the model settings in \citet{CYZ17}. Moreover, it can be seen that test size of $T_n$ is a little biased when $p=20,~T=100$. Actually, such bias appears to be more significant when we increase the dimension-to-sample ratio $p/T$ to a relative higher level, say $0.5$. On the contrary, our test statistics maintain the nominal level accurately in both low and high dimensional settings. $T_n^*$ shows very resilient powerful performance while it is quite time-consuming due to its relatively complicated bootstrap procedures. All in all, our test statistics $G_q$, $G_{q,1}$ and $G_{q,1}^*$ provide very satisfactory alternatives for high dimensional diagnostic checking.

{
\subsection{Performance under VMA(1) model}\label{subsec:VMAComp}
In this section we compare performance of the tests when $\x_t$ follows a vector moving average process of order 1, i.e.
\[\x_t=A_0\z_t+A_1\z_{t-1}.\]
We use different settings for $\z_t$ and $A_0,A_1$ respectively to compare our test statistic $G_q$ as defined in \eqref{eq:Gqtest} and $G_{q,1}$ in \eqref{eq:Gq1test0} under nominal level $\alpha=5\%$.

As for $\z_t$, we use the same two models as defined in (I) and (II) in Section \ref{subsec:empiricalsize}. As for $A_0$ and $A_1$, we use two different models as follows.
\begin{itemize}
	\item[(V)] $A_0={\bf I}_p$ and   $A_1=a{\bf I}_p$, $0<a<1$. 
	\item[(VI)] $A_0={\bf I}_p$ and  for $0<r<1$, take $d=[pr]$, here $[\cdot]$ means to take the closest integer to the given value. $A_1=\lb \frac{4}{p}E_0E_0^*\rb^{1/2}$, where $E_0$ is $p\times d$ matrix with entries $e_{ij}\sim U(-1,1)$ i.i.d., thus $\mbox{rank}(A_1)\leq d<p$.
\end{itemize}

To evaluate the performance of our test statistics $G_q$ and $G_{q,1}$ under VMA(1) models, we assign $a=0.07$ and  $r=0.01$, $d=\max(1,[pr])$
respectively for Scenario (V) and (VI). Testing power of $G_q$ and $G_{q,1}$ are shown in Table~\ref{Tab:VMAPower} for $q=1$ under various $(p,T)$ combinations. The asymptotic power $\beta(G_{1,1})$ of the test statistic $G_{1,1}$ derived in Proposition \ref{prop:power} are also listed for comparison. All empirical results are obtained using 2000 independent replicates.

Similarly as in Section \ref{subsec:comp}, we further compare our test statistics with others, i.e. $T_n$ and $T_n^*$ in \citet{CYZ17} under the VMA(1) settings. Here we fix $p=20, ~T=100$ and let $\x_t=A_0\z_t+A_1\z_{t_1}$ where $A_1$ follows model (V) or (VI) and $\z_t$ is either Gaussian or Non-Gaussian.  Power values of all the five test statistics, i.e. $G_q$, $G_{q,1}$,
$G_{q,1}^*$, $T_n$ and $T_n^*$, are compared under model (V) and (VI)
separately. Figure~\ref{Fig:SF4} shows the results under model (V)
when coefficient $a$ of $A_1$  grows from 0 to
0.5 (top rows), and for model (VI) when parameter $r$ varies
from 0 to 0.5 (bottom rows). All results  are based on 2000 independent
experiments.

From Table~\ref{Tab:VMAPower}, it can be seen that our test statistics
$G_1$ and $G_{1,1}$ consistently show reasonable powers for various
$(p,T)$ combinations under both VMA(1) model settings. Especially
$G_{1,1}$ performs surprisingly well under VMA model (VI) even when
$d~(\mbox{rank}(A_1))$ is very small.  Meanwhile the empirical power
of $G_{1,1}$ is consistent with the asymptotic values $\beta(G_{1,1})$
derived in Proposition~\ref{prop:power}.  As for comparison with $T_n$
and $T_n^*$ in Figure~\ref{Fig:SF4}, our test 
statistics in general show better performance under VMA(1) model
settings.  The test sizes of $T_n$ and $T_n^*$ are a little biased
when $p=20,~T=100$, especially for non-Gaussian cases. While our test
statistics  maintain the nominal level accurately and uphold higher
detection power even when the signals are relatively weak.

}
\section{Proofs}\label{sec:proofs}

\subsection{Proof of Theorem \ref{MainThmMult}}

To derive the null distribution of $G_q$ when  $\x_t=\Sigma_0^{1/2}\z_t$, we looked into the Free probability and moment method proposed by \citet{BB16}. In Section 4.2.3 of \citet{BB16}, they have proved the following result:

\begin{proposition}\label{prop:asyNormal}
Let $\Pi:=\Pi(\widehat{\Sigma}_{\tau},\widehat{\Sigma}_{\tau}^*:\tau\geq0)$ be a symmetric polynomial in $\{\widehat{\Sigma}_{\tau},\widehat{\Sigma}_{\tau}^*:\tau\geq 0\}$, $$\sigma^2_{\Pi}=\lim \E(\Tr(\Pi)-\E(\Tr(\Pi)))^2.$$
They have,
\[
\lim \E (\Tr(\Pi)-\E(\Tr(\Pi)))^k=\left\{
\begin{array}{ll}
0,& \mbox{ if } k=2d-1,\\
\left(\prod_{l=1}^d(2d-2l+1)\right)\sigma_{\Pi}^{2d}, & \mbox{ if } k=2d.
\end{array}
\right.
\]
therefore, as $p,T\rightarrow\infty,~c_p=p/T\rightarrow c\in(0,\infty)$,
\[\Tr(\Pi)-\E\Tr(\Pi)\xrightarrow{d}\N(0,\sigma^2_{\Pi}).\]
\end{proposition}

Since  $G_q$ is a symmetric polynomial in
$\{\widehat{\Sigma}_{\tau},\widehat{\Sigma}_{\tau}^*:\tau\geq 0\}$,
its  asymptotic normality directly results from the proposition above.
It remains to determine its first two  moments in order to  get the
null distribution.  This is done in the following corollary which is a
direct consequence of moment calculations presented in 
Section~1 of the supplemental paper~\cite{supp}.

\begin{corollary}\label{Cor:Gq}
	Let the assumptions for $\z_t$ in Theorem \ref{MainThmMult}
    hold. Under the framework $p/T \rightarrow c > 0$, assume that
    $\norm{\Sigma_0} = O(1)$. Then as $p,T\rightarrow \infty$,
    \begin{align*}
     & \E (G_q)  \sim q p^2 s_1^2/T  ,\\
	&\text{\em Var}(G_q)  \rightarrow qc^2(s_2^2 + b^2(s_2')^2) \\
	&\quad +    4q^2c^3(\nu_4-b-2)s_1^2s_{d,2} + 8q^2c^3s_1^2s_{r,2} +
    4q^2c^3(b-1)s_1^2s_2',
  \end{align*}
  where $s_2' = \lim_{p\rightarrow \infty} \Tr(\Sigma_0\Sigma_0^T)/p$, $s_{r,2} = \lim_{p\rightarrow \infty}\Tr(\Re^2(\Sigma_0))/p$.
\end{corollary}
If the $z_{it}$'s are real, then $\Sigma_0$ is real symmetric and $b=1$, $s_2' = s_{r,2} = s_2$. The asymptotic formula for Var$(G_q)$ then reduces to
\[ 2qc^2s_2^2 + 4q^2c^3(\nu_4-3)s_1^2s_{d,2} + 8q^2c^3s_1^2s_2, \]
which further reduces to $2qc^2s_2^2 + 8q^2c^3s_1^2s_2$ if all the $z_{it}$'s are Gaussian.

\subsection{Proof of Theorem \ref{Thm:s1}} The proof of Theorem
\ref{Thm:s1} is similar to that of Theorem \ref{MainThmMult}, while in
this proof we only consider the real value cases. Both $G_q$ and
$p(\hat{s}_1^2-s_1^2)$ are symmetric polynomials in
$\{\widehat{\Sigma}_{\tau},\widehat{\Sigma}_{\tau}^*:\tau\geq 0\}$,
thus the asymptotic normality of any linear combinations of these two
statistics have been proven by Proposition \ref{prop:asyNormal}. We
can directly calculate the first two order moments and covariance of
these two statistics to obtain the joint limiting distribution. By
directly conducting moment calculations
as in Section~1 of the supplemental paper~\cite{supp},
we have the following proposition. 

\begin{proposition}\label{prop:joint}
	Let the assumptions for $\z_t$ in Theorem \ref{MainThmMult}
hold. Under the framework $p/T \rightarrow c > 0$, assume that
$\norm{\Sigma_0} = O(1)$. Then as $p,T\rightarrow \infty$,
\begin{gather*}
\mathbb{E}\lb p\hat{s}_1^2\rb=\frac{1}{p}\Tr^2(\Sigma_0)+\frac{1}{pT}\lb 2\Tr(\Sigma_0^2)+(\nu_4-3)\Tr(D^2(\Sigma_0))\rb,\\
\var(p\hat{s}_1^2)=\frac{8}{p^2T}\Tr(\Sigma_0^2)\Tr^2(\Sigma_0)+\frac{4}{p^2T}(\nu_4-3)\Tr^2(\Sigma_0)\Tr(D(\Sigma_0))+o(\frac{1}{T}),\\
\mathbb{E}(\hat{s}_2)=\frac{1}{p}\Tr(\Sigma_0^2)+\frac{1}{pT}\Tr^2(\Sigma_0)+\frac{1}{pT}\lb \Tr(\Sigma_0^2)+(\nu_4-3)\Tr(D^2(\Sigma_0))\rb,\\
\mathbb{E}(G_q)=\frac{q}{T}\Tr^2(\Sigma_0),~\var(G_q)=\frac{4q^2}{T^3}\Tr^2(\Sigma_0)\lb 2\Tr(\Sigma_0^2)+(\nu_4-3)\Tr(D^2(\Sigma_0)) \rb\\
\quad+\frac{2q}{T^2}\Tr^2(\Sigma_0^2)+\frac{q}{T^3}\lb 2\Tr(\Sigma_0^2)+(\nu_4-3)\Tr(D^2(\Sigma_0))\rb^2+o(\frac{1}{T}),\\
\cov\lb G_q, ~p\hat{s}_1^2\rb=\frac{4q}{pT^2}\Tr^2(\Sigma_0)\lb 2\Tr(\Sigma_0^2)+(\nu_4-3)\Tr(D^2(\Sigma_0)) \rb+o(\frac{1}{T}).
\end{gather*}	
\end{proposition}

Results in Theorem~\ref{Thm:s1} and Proposition~\ref{prop:Gq1}
naturally follows from  Proposition~\ref{prop:joint}.  The proof of
Proposition~\ref{prop:joint} is postponed to
Section~2 of the supplemental paper~\cite{supp}.

\subsection{Proof of Theorem~\ref{Thm:jointH1}} The proof of Theorem \ref{Thm:jointH1} is similar to that of Theorem \ref{Thm:s1} while the calculations are more complicated. When $\x_t=A_0\z_t+A_1\z_{t-1}$, both $G_q$ and $p(\hat{s}_1^2-s_1^2)$ are still symmetric polynomials in $\{\widehat{\Sigma}_{\tau},\widehat{\Sigma}_{\tau}^*:\tau\geq 0\}$, thus the asymptotic normality of any linear combinations of these two statistics have been proven by Proposition \ref{prop:asyNormal}. We can directly calculate the first two order moments and covariance of these two statistics to obtain the joint limiting distribution. 

To elucidate the calculations of moments, we implement the following decompositions on both $G_q$ and $qTc_p^2\hat{s}_1^2$ when $\x_t=A_0\z_t+A_1\z_{t-1}$ for $q=1$. Actually,
{\small
\begin{align*}
&G_1=\frac{1}{T^2}\sum_{s,t=1}^T\lb A_0\z_s+A_1\z_{s-1}\rb^*\lb A_0\z_t+A_1\z_{t-1}\rb\lb A_0\z_{t-1}+A_1\z_{t-2}\rb^*\lb A_0\z_{s-1}+A_1\z_{s-2}\rb\\
&=G(I)+G(II)+G(III),\\
&Tc_p^2\hat{s}_1^2=\frac{1}{T^3}\sum_{s,t=1}^T\lb A_0\z_s+A_1\z_{s-1}\rb^*\lb A_0\z_s+A_1\z_{s-1}\rb\lb A_0\z_{t}+A_1\z_{t-1}\rb^*\lb A_0\z_{t}+A_1\z_{t-1}\rb\\
&=S(I)+S(II)+S(III),
\end{align*}}
where 
\begin{align*}
G(I)&=\frac{1}{T^2}\sum_{s,t=1}^T\left(\z_s^*A_0^*A_0\z_t\z_{t-1}^*A_0^*A_0\z_{s-1}+\z_{s-1}^*A_1^*A_1\z_{t-1}\z_{t-2}^*A_1^*A_1\z_{s-2} \right.\\
&+\left.\z_s^*A_0^*A_0\z_t\z_{t-2}^*A_1^*A_1\z_{s-2}+\z_{s-1}^*A_1^*A_1\z_{t-1}\z_{t-1}^*A_0^*A_0\z_{s-1}\right),\\
G(II)&=\frac{1}{T^2}\sum_{s,t=1}^T\left(\z_s^*A_0^*A_1\z_{t-1}\z_{t-1}^*A_0^*A_0\z_{s-1}+\z_{s-1}^*A_1^*A_0\z_{t}\z_{t-1}^*A_0^*A_0\z_{s-1} \right.\\
&+\z_{s-1}^*A_1^*A_1\z_{t-1}\z_{t-1}^*A_0^*A_1\z_{s-2}+\z_{s-1}^*A_1^*A_1\z_{t-1}\z_{t-2}^*A_1^*A_0\z_{s-1}\\
&+\z_{s}^*A_0^*A_0\z_{t}\z_{t-2}^*A_1^*A_0\z_{s-1}+\z_{s}^*A_0^*A_0\z_{t}\z_{t-1}^*A_0^*A_1\z_{s-2}\\
&\left.+\z_{s}^*A_0^*A_1\z_{t-1}\z_{t-2}^*A_1^*A_1\z_{s-2}+\z_{s-1}^*A_1^*A_0\z_{t}\z_{t-2}^*A_1^*A_1\z_{s-2}\right),\\
G(III)&=\frac{1}{T^2}\sum_{s,t=1}^T\left( \z_s^*A_0^*A_1\z_{t-1}\z_{t-1}^*A_0^*A_1\z_{s-2}+\z_{s-1}^*A_1^*A_0\z_t\z_{t-2}^*A_1^*A_0\z_{s-1}\right.\\
&\left.+\z_s^*A_0^*A_1\z_{t-1}\z_{t-2}^*A_1^*A_0\z_{s-1}+\z_{s-1}^*A_1^*A_0\z_t\z_{t-1}^*A_0^*A_1\z_{s-2}\right),
\end{align*}
\begin{align*}
S(I)&=\frac{1}{T^3}\sum_{s,t=1}^T\left(\z_t^*A_0^*A_0\z_t\z_{s}^*A_0^*A_0\z_{s}+\z_{t-1}^*A_1^*A_1\z_{t-1}\z_{s-1}^*A_1^*A_1\z_{s-1} \right.\\
&+\left.\z_t^*A_0^*A_0\z_t\z_{s-1}^*A_1^*A_1\z_{s-1}+\z_{t-1}^*A_1^*A_1\z_{t-1}\z_{s}^*A_0^*A_0\z_{s}\right),\\
S(II)&=\frac{1}{T^2}\sum_{s,t=1}^T\left(\z_t^*A_0^*A_1\z_{t-1}\z_{s}^*A_0^*A_0\z_{s}+\z_{t-1}^*A_1^*A_0\z_{t}\z_{s}^*A_0^*A_0\z_{s} \right.\\
&+\z_{t-1}^*A_1^*A_1\z_{t-1}\z_{s}^*A_0^*A_1\z_{s-1}+\z_{t-1}^*A_1^*A_1\z_{t-1}\z_{s-1}^*A_1^*A_0\z_{s}\\
&+\z_{t}^*A_0^*A_0\z_{t}\z_{s-1}^*A_1^*A_0\z_{s}+\z_{t}^*A_0^*A_0\z_{t}\z_{s}^*A_0^*A_1\z_{s-1}\\
&\left.+\z_{t}^*A_0^*A_1\z_{t-1}\z_{s-1}^*A_1^*A_1\z_{s-1}+\z_{t-1}^*A_1^*A_0\z_{t}\z_{s-1}^*A_1^*A_1\z_{s-1}\right),\\
S(III)&=\frac{1}{T^2}\sum_{s,t=1}^T\left( \z_t^*A_0^*A_1\z_{t-1}\z_{s}^*A_0^*A_1\z_{s-1}+\z_{t-1}^*A_1^*A_0\z_t\z_{s-1}^*A_1^*A_0\z_{s}\right.\\
&\left.+\z_t^*A_0^*A_1\z_{t-1}\z_{s-1}^*A_1^*A_0\z_{s}+\z_{t-1}^*A_1^*A_0\z_t\z_{s}^*A_0^*A_1\z_{s-1}\right),
\end{align*}
By conducting moment calculations
similar to Section~1 of the supplemental paper~\cite{supp}, 
we have the following proposition.

\begin{proposition}\label{prop:H1Lemma}
	Let the assumptions in Theorem~\ref{Thm:jointH1} hold, as $p,T\rightarrow \infty, ~p/T\rightarrow c>0$, we have
	\begin{align*}
\mathbb{E}\lb G(I)\rb&=\frac{1}{T}\left[ \Tr\lb\widetilde{\Sigma}_0\widetilde{\Sigma}_1\rb+(\nu_4-3)\Tr\lb D(\widetilde{\Sigma}_0)D(\widetilde{\Sigma}_1)\rb\right]\\
&+
\frac{1}{T}\Tr^2\lb\widetilde{\Sigma}_0+\widetilde{\Sigma}_1\rb+\Tr\lb\widetilde{\Sigma}_0\widetilde{\Sigma}_1\rb,\\
\mathbb{E}\lb G(II)\rb&=0,\quad\mathbb{E}\lb G(III)\rb =\frac{2}{T}\Tr^2\lb\widetilde{\Sigma}_{01}\rb,\\
\mathbb{E}\lb S(I)\rb&=\frac{1}{T^2}\left[2\Tr\lb\widetilde{\Sigma}_0+\widetilde{\Sigma}_1\rb^2+(\nu_4-3)\Tr\lb D^2(\widetilde{\Sigma}_0+\widetilde{\Sigma}_1)\rb\right]\\
&+\frac{1}{T}\Tr^2\lb\widetilde{\Sigma}_0+\widetilde{\Sigma}_1\rb,\\
\mathbb{E}\lb S(II)\rb&=0,\quad \mathbb{E}\lb S(III)\rb=\frac{4}{T^2}\Tr\lb\widetilde{\Sigma}_{01}\widetilde{\Sigma}_{01}^*\rb,
	\end{align*}
	and
{\small
	\begin{gather*}
	\var\lb G(I)\rb=\frac{4}{T^3}\Tr^2\lb \widetilde{\Sigma}_0+\widetilde{\Sigma}_1\rb\left[2\Tr\lb \widetilde{\Sigma}_0+\widetilde{\Sigma}_1\rb^2+(\nu_4-3)\Tr\lb D^2\lb \widetilde{\Sigma}_0+\widetilde{\Sigma}_1\rb\rb\right]\\
	+\frac{8}{T^2}\Tr\lb \widetilde{\Sigma}_0+\widetilde{\Sigma}_1\rb \left[2\Tr\lb \widetilde{\Sigma}_0\widetilde{\Sigma}_1\lb\widetilde{\Sigma}_0+\widetilde{\Sigma}_1\rb\rb+(\nu_4-3)\Tr\lb D\lb\widetilde{\Sigma}_0\widetilde{\Sigma}_1\rb D\lb \widetilde{\Sigma}_0+\widetilde{\Sigma}_1\rb \rb\right]\\
	+\frac{2}{T^2}\Tr^2\lb \widetilde{\Sigma}_0^2+\widetilde{\Sigma}_1^2\rb+\frac{6}{T^2}\Tr^2\lb\widetilde{\Sigma}_0\widetilde{\Sigma}_1 \rb+\frac{4}{T}\left[ 2\Tr\lb \widetilde{\Sigma}_0\widetilde{\Sigma}_1\rb^2+(\nu_4-3)\Tr\lb D^2\lb \widetilde{\Sigma}_0\widetilde{\Sigma}_1\rb\rb\right]+R_n,
\end{gather*}
\begin{gather*}
\var\lb G(III)\rb=\frac{4}{T}\Tr\lb\widetilde{\Sigma}_{01}\widetilde{\Sigma}_{01}^*\widetilde{\Sigma}_{01}^*\widetilde{\Sigma}_{01}\rb+\frac{12}{T^2}\Tr^2\lb\widetilde{\Sigma}_{01}\widetilde{\Sigma}_{01}^*\rb+\frac{16}{T^2}\Tr(\widetilde{\Sigma}_{01})\Tr\lb\widetilde{\Sigma}_{01}\widetilde{\Sigma}_{01}^*\widetilde{\Sigma}_{01}^* \rb\\
+\frac{16}{T^3}\Tr^2\lb \widetilde{\Sigma}_{01}\rb \left[ \Tr\lb\widetilde{\Sigma}_{01}\rb^2+2\Tr\lb\widetilde{\Sigma}_{01}\widetilde{\Sigma}_{01}^*\rb+(\nu_4-3)\Tr\lb D^2\lb \widetilde{\Sigma}_{01}\rb\rb\right]+R_n,
\end{gather*}
\begin{gather*}
\var\lb G(II)\rb=
\frac{8}{T^2}\Tr\lb \widetilde{\Sigma}_{01}\widetilde{\Sigma}_{01}^*\rb\Tr\lb\widetilde{\Sigma}_0^2+\widetilde{\Sigma}_1^2\rb+\frac{16}{T^2}\Tr\lb \widetilde{\Sigma}_{01}\widetilde{\Sigma}_1\rb\Tr\lb \widetilde{\Sigma}_{01}\widetilde{\Sigma}_0\rb\\
+\frac{16}{T^2}\Tr\lb \widetilde{\Sigma}_0+\widetilde{\Sigma}_1\rb \left[ \Tr\lb \widetilde{\Sigma}_{01}^*\widetilde{\Sigma}_{01}\widetilde{\Sigma}_0\rb+ \Tr\lb \widetilde{\Sigma}_{01}\widetilde{\Sigma}_{01}^*\widetilde{\Sigma}_1\rb\right]\\
+\frac{16}{T^2}\Tr(\widetilde{\Sigma}_{01})\left[\Tr\lb\widetilde{\Sigma}_0^2\widetilde{\Sigma}_{01}^*\rb+\Tr\lb\widetilde{\Sigma}_1^2\widetilde{\Sigma}_{01}\rb+2\Tr\lb \widetilde{\Sigma}_1\widetilde{\Sigma}_{01}\widetilde{\Sigma}_0\rb\right]\\
+\frac{4}{T}\Tr\lb \widetilde{\Sigma}_{01}^*\widetilde{\Sigma}_{01}\widetilde{\Sigma}_0^2+\widetilde{\Sigma}_{01}\widetilde{\Sigma}_{01}^*\widetilde{\Sigma}_1^2+2\widetilde{\Sigma}_{01}^*\widetilde{\Sigma}_1\widetilde{\Sigma}_{01}\widetilde{\Sigma}_0\rb\\
+\frac{16}{T^3}\Tr^2\lb \widetilde{\Sigma}_0+\widetilde{\Sigma}_1 \rb\Tr\lb \widetilde{\Sigma}_{01}\widetilde{\Sigma}_{01}^*\rb+\frac{16}{T^3}\Tr^2(\widetilde{\Sigma}_{01})\Tr\lb \widetilde{\Sigma}_0+\widetilde{\Sigma}_1\rb^2\\
+\frac{32}{T^3}\Tr\lb\widetilde{\Sigma}_0+\widetilde{\Sigma}_1\rb\Tr\lb\widetilde{\Sigma}_{01}\rb\Tr\lb \widetilde{\Sigma}_{01}\lb\widetilde{\Sigma}_0+\widetilde{\Sigma}_1\rb\rb+R_n,
\end{gather*}
\begin{gather*}
\cov\lb G(I),~G(III) \rb=\frac{4}{T^2}\Tr^2\lb \widetilde{\Sigma}_0 \widetilde{\Sigma}_{01}\rb +\frac{4}{T^2}\Tr^2\lb \widetilde{\Sigma}_1 \widetilde{\Sigma}_{01}\rb\\
+\frac{8}{T^3}\Tr\lb\widetilde{\Sigma}_{01}\rb\Tr\lb\widetilde{\Sigma}_0+\widetilde{\Sigma}_1\rb\left[2\Tr\lb \widetilde{\Sigma}_{01}\lb\widetilde{\Sigma}_0+\widetilde{\Sigma}_1\rb\rb+(\nu_4-3)\Tr\lb D\lb \widetilde{\Sigma}_{01} \rb D\lb \widetilde{\Sigma}_0+\widetilde{\Sigma}_1 \rb\rb\right]\\
+\frac{8}{T^2}\Tr\lb\widetilde{\Sigma}_{01}\rb\left[2\Tr\lb\widetilde{\Sigma}_0\widetilde{\Sigma}_1\widetilde{\Sigma}_{01}\rb+(\nu_4-3)\Tr\lb D\lb\widetilde{\Sigma}_0\widetilde{\Sigma}_1 \rb D\lb \widetilde{\Sigma}_{01}\rb\rb\right]+R_n,\\
\cov\lb G(I),~G(II)\rb=o(1),\quad\cov\lb G(II),~G(III)\rb=o(1);
\end{gather*}
\begin{gather*}
\var\lb S(I)\rb=\frac{4}{T^3}\Tr^2\lb\widetilde{\Sigma}_0+\widetilde{\Sigma}_1\rb\left[2\Tr\lb\widetilde{\Sigma}_0+\widetilde{\Sigma}_1\rb^2+(\nu_4-3)\Tr\lb D^2\lb \widetilde{\Sigma}_0+\widetilde{\Sigma}_1\rb\rb\right]+R_n,\\
\var\lb S(II)\rb= \frac{16}{T^3}\Tr^2\lb\widetilde{\Sigma}_0+\widetilde{\Sigma}_1\rb\Tr\lb\widetilde{\Sigma}_{01}\widetilde{\Sigma}_{01}^*\rb+R_n,\\
\var\lb S(III)\rb=\frac{32}{T^4}\Tr^2\lb\widetilde{\Sigma}_{01}\widetilde{\Sigma}_{01}^*\rb,~\cov\lb S(I),~S(II)\rb=o(1),\\
~\cov\lb S(I),~S(III)\rb=o(1),~\cov\lb S(II),~S(III)\rb=o(1);
\end{gather*}
\begin{gather*}
\cov\lb G(I),~S(I)\rb=\frac{4}{T^3}\Tr^2\lb\widetilde{\Sigma}_0+\widetilde{\Sigma}_1\rb\left[2\Tr\lb \widetilde{\Sigma}_0+\widetilde{\Sigma}_1\rb^2+(\nu_4-3)\Tr\lb D^2\lb \widetilde{\Sigma}_0+\widetilde{\Sigma}_1\rb\rb\right]\\
+\frac{4}{T^2}\Tr\lb\widetilde{\Sigma}_0+\widetilde{\Sigma}_1\rb\left[2\Tr\lb \widetilde{\Sigma}_0\widetilde{\Sigma}_1\lb\widetilde{\Sigma}_0+\widetilde{\Sigma}_1\rb\rb+(\nu_4-3)\Tr\lb D\lb\widetilde{\Sigma}_0\widetilde{\Sigma}_1\rb D\lb \widetilde{\Sigma}_0+\widetilde{\Sigma}_1\rb \rb\right]+R_n,\\
\cov\lb G(II),~S(II)\rb=\frac{8}{T^2}\Tr\lb\widetilde{\Sigma}_0+\widetilde{\Sigma}_1\rb\left[ \Tr\lb \widetilde{\Sigma}_{01}^*\widetilde{\Sigma}_{01}\widetilde{\Sigma}_0\rb+ \Tr\lb \widetilde{\Sigma}_{01}\widetilde{\Sigma}_{01}^*\widetilde{\Sigma}_1\rb\right]\\
+\frac{16}{T^3}\Tr^2\lb \widetilde{\Sigma}_0+\widetilde{\Sigma}_1 \rb\Tr\lb \widetilde{\Sigma}_{01}\widetilde{\Sigma}_{01}^*\rb+\frac{16}{T^3}\Tr\lb\widetilde{\Sigma}_0+\widetilde{\Sigma}_1\rb\Tr\lb\widetilde{\Sigma}_{01}\rb\Tr\lb \widetilde{\Sigma}_{01}\lb\widetilde{\Sigma}_0+\widetilde{\Sigma}_1\rb\rb+R_n,\\
\cov\lb G(III),S(I)\rb=\frac{8}{T^3}\Tr\lb\widetilde{\Sigma}_{01}\rb\Tr\lb\widetilde{\Sigma}_0+\widetilde{\Sigma}_1\rb\left[2\Tr\lb \widetilde{\Sigma}_{01}\lb\widetilde{\Sigma}_0+\widetilde{\Sigma}_1\rb\rb\right.\\
\left.+(\nu_4-3)\Tr\lb D\lb \widetilde{\Sigma}_{01} \rb D\lb \widetilde{\Sigma}_0+\widetilde{\Sigma}_1 \rb\rb\right]+R_n,\\
\cov\lb G(III),~S(III)\rb=0,~\cov\lb G(I),~S(II)\rb=o(1),~~\cov\lb G(I),~S(III)\rb=0,\\
\cov\lb G(II),~S(I)\rb=o(1),~~\cov\lb G(II),~S(III)\rb=o(1),~\cov\lb G(III),~S(II)\rb=o(1).
\end{gather*}}
Here the $R_n$'s are possibly different: they  represent remainder
terms with smaller orders than the others listed in each variance covariance items.
\end{proposition}

Theorem~\ref{Thm:jointH1} naturally follows from Proposition~\ref{prop:H1Lemma}.

\begin{supplement}[id=suppA]
  \stitle{Supplement to ``On testing for high-dimensional white noise''}
  \slink[doi]{COMPLETED BY THE TYPESETTER}
  \sdatatype{.pdf}
  \sdescription{This supplemental article contains some technical
    lemmas, the proof of Proposition 4.2 of the main article, and some
    additional  simulation results.}\footnote{For the version posted
    on arXiv, the supplemental article appears as an appendix of
    this main paper.}
\end{supplement}




{\scriptsize
\begin{table}[!h]
\caption{Empirical sizes for our tests $G_{q}$ and $G_{q,1}$, the  Hosking
  test  $\widetilde{Q}_{q}$ and the Li-McLeod test $Q_{q}^{*}$.\label{Tab:cLSizeCom}}
\resizebox{0.98\textwidth}{!}{
	\def\arraystretch{1.1}
\begin{tabular}{ccc|cc|cc|cc|cc}
	\hline 
	&  &  & \multicolumn{2}{c|}{$G_{q}$} & \multicolumn{2}{c|}{$G_{q,1}$} & \multicolumn{2}{c|}{$\widetilde{Q}{}_{q}$} & \multicolumn{2}{c}{$Q_{q}^{*}$}\tabularnewline
	\cline{4-11} 
	$p$ & $T$ & $p/T$ & $q=1$ & $q=3$ & $q=1$ & $q=3$ & $q=1$ & $q=3$ & $q=1$ & $q=3$\tabularnewline
	\hline 
	5 & 1000 & 0.005 & 0.0630 & 0.0615 & 0.0610 & 0.0645 & 0.0490 & 0.0478 & 0.0488 & 0.0476\tabularnewline
	10 & 2000 & 0.005 & 0.0630 & 0.0580 & 0.0615 & 0.0575 & 0.0492 & 0.0440 & 0.0492 & 0.0436\tabularnewline
	25 & 5000 & 0.005 & 0.0520 & 0.0470 & 0.0575 & 0.0535 & 0.0498 & 0.0528 & 0.0498 & 0.0528\tabularnewline
	40 & 8000 & 0.005 & 0.0565 & 0.0395 & 0.0540 & 0.0430 & 0.0508 & 0.0520 & 0.0508 & 0.0520\tabularnewline
	\hline
	10 & 1000 & 0.01 & 0.0740 & 0.0565 & 0.0675 & 0.0570 & 0.0472 & 0.0468 & 0.0470 & 0.0464\tabularnewline
	20 & 2000 & 0.01 & 0.0500 & 0.0555 & 0.0540 & 0.0540 & 0.0502 & 0.0530 & 0.0502 & 0.0530\tabularnewline
	50 & 5000 & 0.01 & 0.0455 & 0.0555 & 0.0450 & 0.0580 & 0.0488 & 0.0498 & 0.0488 & 0.0498\tabularnewline
	80 & 8000 & 0.01 & 0.0500 & 0.0490 & 0.0510 & 0.0520 & 0.0464 & 0.0406 & 0.0464 & 0.0404\tabularnewline
	\hline
	50 & 1000 & 0.05 & 0.0375 & 0.0495 & 0.0410 & 0.0475 & 0.0408 & 0.0466 & 0.0408 & 0.0466\tabularnewline
	100 & 2000 & 0.05 & 0.0570 & 0.0525 & 0.0560 & 0.0515 & 0.0432 & 0.0414 & 0.0432 & 0.0414\tabularnewline
	250 & 5000 & 0.05 & 0.0500 & 0.0480 & 0.0495 & 0.0500 & 0.0456 & 0.0436 & 0.0456 & 0.0434\tabularnewline
	400 & 8000 & 0.05 & 0.0410 & 0.0480 &0.0455  &0.0505  & 0.0418 & 0.0410 & 0.0418 & 0.0410\tabularnewline
	\hline
	10 & 100 & 0.1 & 0.0570 & 0.0555 & 0.0555 & 0.0570 & 0.0300 & 0.0400 & 0.0280 & 0.0362\tabularnewline
	40 & 400 & 0.1 & 0.0560 & 0.0590 & 0.0575 & 0.0525 & 0.0362 & 0.0342 & 0.0358 & 0.0338\tabularnewline
	60 & 600 & 0.1 & 0.0465 & 0.0585 & 0.0550 & 0.0595 & 0.0340 & 0.0340 & 0.0340 & 0.0338\tabularnewline
	100 & 1000 & 0.1 & 0.0515 & 0.0500 & 0.0435 & 0.0480 & 0.0370 & 0.0268 & 0.0366 & 0.0264\tabularnewline
	\hline
	50 & 100 & 0.5 & 0.0520 & 0.0465 & 0.0480 & 0.0520 & 0.0006 & 0.0018 & 0.0006 & 0.0018\tabularnewline
	200 & 400 & 0.5 & 0.0400 & 0.0415 & 0.0505 & 0.0545 & 0.0010 & 0.0004 & 0.0010 & 0.0004\tabularnewline
	300 & 600 & 0.5 & 0.0390 & 0.0480 & 0.0455 & 0.0480 & 0.0002 & 0.0008 & 0.0002 & 0.0008\tabularnewline
	500 & 1000 & 0.5 & 0.0470 & 0.0470 & 0.0430 & 0.0545 & 0 & 0 & 0 & 0\tabularnewline
	\hline
	90 & 100 & 0.9 & 0.0555 & 0.0580 & 0.0460 & 0.0455 & 0 & 0 & 0 & 0\tabularnewline
	360 & 400 & 0.9 & 0.0475 & 0.0520 & 0.0535 & 0.0405 & 0 & 0 & 0 & 0\tabularnewline
	540 & 600 & 0.9 & 0.0535 & 0.0550 & 0.0550 & 0.0540 & 0 & 0 & 0 & 0\tabularnewline
900&1000  &0.9  & 0.0495 &0.0505  &0.0545  &0.0515  &  0&  0&  0&0 \tabularnewline
	\hline
\end{tabular}}
\end{table}
}




\renewcommand{\arraystretch}{1.15}
\begin{table}[!h]
	\centering
		\caption{Test sizes of our tests $G_q$ and $G_{q,1}$}\label{Tab:HDSize}
	\resizebox{1.05\textwidth}{!}{
\begin{tabular}{ccc|cc|cc|cc|ccc}
	&  & \multicolumn{1}{c}{} & \multicolumn{4}{c}{Gaussian (I)} & \multicolumn{4}{c}{Non-Gaussian (II)} & \tabularnewline
	\cline{1-11} 
	&  &  & \multicolumn{2}{c|}{$G_{q}$} & \multicolumn{2}{c|}{$G_{q,1}$} & \multicolumn{2}{c|}{$G_{q}$} & \multicolumn{2}{c}{$G_{q,1}$} & \tabularnewline
	\cline{4-11} 
	$p$ & $T$ & $p/T$ & $q=1$ & $q=3$ & $q=1$ & $q=3$ & $q=1$ & $q=3$ & $q=1$ & $q=3$ & \tabularnewline
	\cline{1-11} 
	5 & 500 & 0.01 & 0.0500 & 0.0545 & 0.0465 & 0.0485 & 0.0650 & 0.0655 & 0.0655 & 0.0540 & \tabularnewline
	10 & 1000 & 0.01 & 0.0565 & 0.0420 & 0.0575 & 0.0400 & 0.0515 & 0.0615 & 0.0600 & 0.0575 & \tabularnewline
	20 & 2000 & 0.01 & 0.0545 & 0.0570 & 0.0515 & 0.0525 & 0.0610 & 0.0595 & 0.0600 & 0.0510 & \tabularnewline
	\cline{1-11} 
	25 & 500 & 0.05 & 0.0550 & 0.0570 & 0.0630 & 0.0510 & 0.0570 & 0.0645 & 0.0520 & 0.0565 & \tabularnewline
	50 & 1000 & 0.05 & 0.0520 & 0.0515 & 0.0510 & 0.0455 & 0.0500 & 0.0485 & 0.0495 & 0.0455 & \tabularnewline
	100 & 2000 & 0.05 & 0.0565 & 0.0410 & 0.0545 & 0.0355 & 0.0500 & 0.0595 & 0.0440 & 0.0530 & (III)\tabularnewline
	\cline{1-11} 
	100 & 100 & 1 & 0.0515 & 0.0545 & 0.0565 & 0.0520 & 0.0515 & 0.0520 & 0.0395 & 0.0420 & \tabularnewline
	200 & 200 & 1 & 0.0540 & 0.0460 & 0.0485 & 0.0395 & 0.0475 & 0.0495 & 0.0450 & 0.0520 & \tabularnewline
	400 & 400 & 1 & 0.0570 & 0.0565 & 0.0505 & 0.0450 & 0.0385 & 0.0420 & 0.0505 & 0.0510 & \tabularnewline
	\cline{1-11} 
	200 & 100 & 2 & 0.0530 & 0.0480 & 0.0560 & 0.0380 & 0.0560 & 0.0545 & 0.0370 & 0.0420 & \tabularnewline
	400 & 200 & 2 & 0.0480 & 0.0500 & 0.0510 & 0.0420 & 0.0545 & 0.0515 & 0.0470 & 0.0390 & \tabularnewline
	800 & 400 & 2 & 0.0505 & 0.0485 & 0.0480 & 0.0520 & 0.0475 & 0.0470 & 0.0405 & 0.0445 & \tabularnewline
	\cline{1-11} 
	\tabularnewline
	\cline{1-11}
	5 & 500 & 0.01 & 0.0630 & 0.0715 & 0.0585 & 0.0665 & 0.0670 & 0.0560 & 0.0650 & 0.0585 & \tabularnewline
	10 & 1000 & 0.01 & 0.0680 & 0.0645 & 0.0695 & 0.0580 & 0.0555 & 0.0540 & 0.0545 & 0.0565 & \tabularnewline
	20 & 2000 & 0.01 & 0.0590 & 0.0545 & 0.0575 & 0.0540 & 0.0655 & 0.0520 & 0.0635 & 0.0560 & \tabularnewline
	\cline{1-11} 
	25 & 500 & 0.05 & 0.0510 & 0.0545 & 0.0505 & 0.0505 & 0.0635 & 0.0590 & 0.0595 & 0.0580 & \tabularnewline
	50 & 1000 & 0.05 & 0.0435 & 0.0425 & 0.0475 & 0.0405 & 0.0550 & 0.0555 & 0.0535 & 0.0465 & \tabularnewline
	100 & 2000 & 0.05 & 0.0480 & 0.0460 & 0.0470 & 0.0420 & 0.0600 & 0.0460 & 0.0595 & 0.0520 & (IV)\tabularnewline
	\cline{1-11} 
	100 & 100 & 1 & 0.0500 & 0.0525 & 0.0455 & 0.0455 & 0.0545 & 0.0485 & 0.0595 & 0.0530 & \tabularnewline
	200 & 200 & 1 & 0.0510 & 0.0530 & 0.0530 & 0.0505 & 0.0495 & 0.0460 & 0.0480 & 0.0520 & \tabularnewline
	400 & 400 & 1 & 0.0535 & 0.0495 & 0.0530 & 0.0390 & 0.0450 & 0.0440 & 0.0510 & 0.0520 & \tabularnewline
	\cline{1-11} 
	200 & 100 & 2 & 0.0550 & 0.0545 & 0.0480 & 0.0605 & 0.0480 & 0.0485 & 0.0415 & 0.0450 & \tabularnewline
	400 & 200 & 2 & 0.0470 & 0.0485 & 0.0540 & 0.0525 & 0.0545 & 0.0525 & 0.0460 & 0.0520 & \tabularnewline
	800 & 400 & 2 & 0.0415 & 0.0505 & 0.0450 & 0.0495 & 0.0480 & 0.0490 & 0.0510 & 0.0495 & \tabularnewline
	\cline{1-11} 
\end{tabular}}
\end{table}

\renewcommand{\arraystretch}{1.15}
\begin{table}[!h]
	\centering
		\caption{Test power of our tests $G_q$ and $G_{q,1}$ under VAR(1)}\label{Tab:HDPower}
	\resizebox{1.05\textwidth}{!}{
\begin{tabular}{cccc|cc|cc|cc|ccc}
	&  &  & \multicolumn{1}{c}{} & \multicolumn{4}{c}{Gaussian (I)} & \multicolumn{4}{c}{Non-Gaussian (II)} & \tabularnewline
	\cline{1-12} 
	&  &  &  & \multicolumn{2}{c|}{$G_{q}$} & \multicolumn{2}{c|}{$G_{q,1}$} & \multicolumn{2}{c|}{$G_{q}$} & \multicolumn{2}{c}{$G_{q,1}$} & \tabularnewline
	\cline{5-12} 
	$p$ & $T$ & $p/T$ & $a$ & $q=1$ & $q=3$ & $q=1$ & $q=3$ & $q=1$ & $q=3$ & $q=1$ & $q=3$ & \tabularnewline
	\cline{1-12} 
	5 & 500 & 0.01 & 0.05 & 0.2355 & 0.1535 & 0.2500 & 0.1540 & 0.2485 & 0.1475 & 0.2465 & 0.1505 & \tabularnewline
	10 & 1000 & 0.01 & 0.05 & 0.5280 & 0.2770 & 0.5335 & 0.2935 & 0.5135 & 0.2645 & 0.5265 & 0.2930 & \tabularnewline
	20 & 2000 & 0.01 & 0.05 & 0.9460 & 0.6620 & 0.9495 & 0.6995 & 0.9355 & 0.6010 & 0.9500 & 0.6670 & \tabularnewline
	\cline{1-12} 
	25 & 500 & 0.05 & 0.05 & 0.2260 & 0.1300 & 0.2485 & 0.1770 & 0.2315 & 0.1395 & 0.2585 & 0.1810 & \tabularnewline
	50 & 1000 & 0.05 & 0.05 & 0.5410 & 0.2800 & 0.5995 & 0.3785 & 0.5105 & 0.2495 & 0.5960 & 0.3750 & \tabularnewline
	100 & 2000 & 0.05 & 0.05 & 0.9580 & 0.6550 & 0.9815 & 0.8275 & 0.9500 & 0.5895 & 0.9805 & 0.8385 & (III)\tabularnewline
	\cline{1-12} 
	100 & 100 & 1 & 0.1 & 0.2615 & 0.2205 & 0.6170 & 0.8190 & 0.2100 & 0.1750 & 0.6165 & 0.8285 & \tabularnewline
	200 & 200 & 1 & 0.1 & 0.6010 & 0.4720 & 0.9870 & 0.9995 & 0.4460 & 0.3370 & 0.9865 & 1 & \tabularnewline
	400 & 400 & 1 & 0.1 & 0.9745 & 0.9230 & 1 & 1 & 0.9025 & 0.7875 & 1 & 1 & \tabularnewline
	\cline{1-12} 
	200 & 100 & 2 & 0.1 & 0.3275 & 0.2710 & 0.9375 & 0.9980 & 0.2420 & 0.2135 & 0.9390 & 0.9995 & \tabularnewline
	400 & 200 & 2 & 0.1 & 0.7415 & 0.6745 & 1 & 1 & 0.5715 & 0.4830 & 1 & 1 & \tabularnewline
	800 & 400 & 2 & 0.1 & 0.9995 & 0.9930 & 1 & 1 & 0.9710 & 0.9350 & 1 & 1 & \tabularnewline
	\cline{1-12} 
	 \tabularnewline
	\cline{1-12} 
	5 & 500 & 0.01 & 0.05 & 0.2540 & 0.1680 & 0.2590 & 0.1700 & 0.2355 & 0.1505 & 0.2450 & 0.1615 & \tabularnewline
	10 & 1000 & 0.01 & 0.05 & 0.4650 & 0.2870 & 0.4730 & 0.2850 & 0.4650 & 0.2885 & 0.4825 & 0.2970 & \tabularnewline
	20 & 2000 & 0.01 & 0.05 & 0.8750 & 0.6170 & 0.8815 & 0.6285 & 0.8880 & 0.5980 & 0.8950 & 0.6190 & \tabularnewline
	\cline{1-12} 
	25 & 500 & 0.05 & 0.05 & 0.2580 & 0.1630 & 0.2555 & 0.1710 & 0.2475 & 0.1415 & 0.2655 & 0.1750 & \tabularnewline
	50 & 1000 & 0.05 & 0.05 & 0.5215 & 0.2650 & 0.5525 & 0.3110 & 0.5165 & 0.2575 & 0.5450 & 0.3270 & \tabularnewline
	100 & 2000 & 0.05 & 0.05 & 0.9450 & 0.6500 & 0.9555 & 0.7320 & 0.9345 & 0.6240 & 0.9635 & 0.7405 & (IV)\tabularnewline
	\cline{1-12} 
	100 & 100 & 1 & 0.1 & 0.2145 & 0.1690 & 0.3700 & 0.4695 & 0.1970 & 0.1470 & 0.3765 & 0.4495 & \tabularnewline
	200 & 200 & 1 & 0.1 & 0.4910 & 0.3470 & 0.8335 & 0.9005 & 0.4355 & 0.2935 & 0.8430 & 0.9150 & \tabularnewline
	400 & 400 & 1 & 0.1 & 0.9205 & 0.7690 & 1 & 1 & 0.8655 & 0.6735 & 1 & 1 & \tabularnewline
	\cline{1-12} 
	200 & 100 & 2 & 0.1 & 0.2450 & 0.2035 & 0.6255 & 0.8115 & 0.2240 & 0.1745 & 0.6425 & 0.8235 & \tabularnewline
	400 & 200 & 2 & 0.1 & 0.5815 & 0.4790 & 0.9915 & 1 & 0.5000 & 0.3770 & 0.9880 & 1 & \tabularnewline
	800 & 400 & 2 & 0.1 & 0.9705 & 0.9205 & 1 & 1 & 0.9425 & 0.8525 & 1 & 1 & \tabularnewline
	\cline{1-12} 
\end{tabular}}

\end{table}

\renewcommand{\arraystretch}{1.2}
\begin{table}[!h]
	\centering
	\caption{Test power of our tests $G_1$ and $G_{1,1}$ under VMA(1)}\label{Tab:VMAPower}
	\resizebox{\textwidth}{!}{
\begin{tabular}{cccc|cc|c|cc|cc}
	&  &  & \multicolumn{1}{c}{} & \multicolumn{3}{c}{Gaussian (I)} & \multicolumn{3}{c}{Non-Gaussian (II)} & \tabularnewline
	\cline{1-10} 
	$p$ & $T$ & $p/T$ & $a$ & $G_{1}$ & $G_{1,1}$ & $\beta(G_{1,1})$ & $G_{1}$ & $G_{1,1}$ & $\beta(G_{1,1})$ & \tabularnewline
	\cline{1-10} 
	10 & 200 & 0.05 & 0.07 & 0.2085 & 0.2260 & 0.2144 & 0.1865 & 0.1990 & 0.2159 & \tabularnewline
	20 & 400 & 0.05 & 0.07 & 0.4135 & 0.4410 & 0.4530 & 0.3805 & 0.4315 & 0.4548 & \tabularnewline
	40 & 800 & 0.05 & 0.07 & 0.8350 & 0.8985 & 0.8903 & 0.7910 & 0.8885 & 0.8910 & \tabularnewline
	\cline{1-10} 
	20 & 200 & 0.1 & 0.07 & 0.1830 & 0.2120 & 0.2235 & 0.1755 & 0.2165 & 0.2250 & \tabularnewline
	40 & 400 & 0.1 & 0.07 & 0.3915 & 0.4925 & 0.5015 & 0.3605 & 0.4800 & 0.5034 & \tabularnewline
	80 & 800 & 0.1 & 0.07 & 0.8480 & 0.9395 & 0.9372 & 0.7995 & 0.9485 & 0.9377 & \tabularnewline
	\cline{1-10} 
	50 & 100 & 0.5 & 0.07 & 0.1185 & 0.1705 & 0.1790 & 0.1070 & 0.1730 & 0.1804 & \tabularnewline
	100 & 200 & 0.5 & 0.07 & 0.2070 & 0.3820 & 0.3958 & 0.1600 & 0.3850 & 0.3977 & (V)\tabularnewline
	200 & 400 & 0.5 & 0.07 & 0.4940 & 0.8395 & 0.8521 & 0.3660 & 0.8400 & 0.8531 & \tabularnewline
	\cline{1-10} 
	100 & 100 & 1 & 0.07 & 0.1305 & 0.2670 & 0.2754 & 0.1120 & 0.2715 & 0.2771 & \tabularnewline
	200 & 200 & 1 & 0.07 & 0.2540 & 0.6605 & 0.6485 & 0.1925 & 0.6470 & 0.6502 & \tabularnewline
	400 & 400 & 1 & 0.07 & 0.5520 & 0.9900 & 0.9903 & 0.4110 & 0.9895 & 0.9904 & \tabularnewline
	\cline{1-10} 
	200 & 100 & 2 & 0.07 & 0.1510 & 0.4990 & 0.5157 & 0.1225 & 0.5000 & 0.5177 & \tabularnewline
	400 & 200 & 2 & 0.07 & 0.3005 & 0.9480 & 0.9500 & 0.2385 & 0.9500 & 0.9504 & \tabularnewline
	800 & 400 & 2 & 0.07 & 0.7310 & 1 & 0.9999 & 0.5500 & 1 & 0.9999 & \tabularnewline
	\cline{1-10} 
	&  &  & \multicolumn{1}{c}{} &  & \multicolumn{1}{c}{} & \multicolumn{1}{c}{} &  & \multicolumn{1}{c}{} &  & \tabularnewline
	\cline{1-10} 
	$p$ & $T$ & $p/T$ & $r$ & $G_{1}$ & $G_{1,1}$ & $\beta(G_{1,1})$ & $G_{1}$ & $G_{1,1}$ & $\beta(G_{1,1})$ & \tabularnewline
	\cline{1-10} 
	10 & 200 & 0.05 & 0.01 & 0.9955 & 0.9995 & 0.9838 & 1 & 1 & 0.9884 & \tabularnewline
	20 & 400 & 0.05 & 0.01 & 1 & 1 & 0.9995 & 1 & 1 & 0.9994 & \tabularnewline
	40 & 800 & 0.05 & 0.01 & 1 & 1 & 0.9999 & 1 & 1 & 0.9999 & \tabularnewline
	\cline{1-10} 
	20 & 200 & 0.1 & 0.01 & 0.9700 & 0.9935 & 0.9705 & 0.9875 & 0.9980 & 0.9815 & \tabularnewline
	40 & 400 & 0.1 & 0.01 & 0.9995 & 1 & 0.9970 & 0.9980 & 1 & 0.9979 & \tabularnewline
	80 & 800 & 0.1 & 0.01 & 1 & 1 & 0.9999 & 0.9995 & 1 & 0.9999 & \tabularnewline
	\cline{1-10} 
	50 & 100 & 0.5 & 0.01 & 0.0530 & 0.0445 & 0.0500 & 0.0615 & 0.0510 & 0.0500 & \tabularnewline
	100 & 200 & 0.5 & 0.01 & 0.3255 & 0.5855 & 0.6185 & 0.2765 & 0.5925 & 0.6155 & (VI)\tabularnewline
	200 & 400 & 0.5 & 0.01 & 0.6080 & 0.9390 & 0.9439 & 0.5150 & 0.9565 & 0.9544 & \tabularnewline
	\cline{1-10} 
	100 & 100 & 1 & 0.01 & 0.1135 & 0.1910 & 0.2132 & 0.1110 & 0.2575 & 0.2759 & \tabularnewline
	200 & 200 & 1 & 0.01 & 0.2200 & 0.5665 & 0.5690 & 0.1780 & 0.5450 & 0.5770 & \tabularnewline
	400 & 400 & 1 & 0.01 & 0.5110 & 0.9755 & 0.9709 & 0.3810 & 0.9620 & 0.9628 & \tabularnewline
	\cline{1-10} 
	200 & 100 & 2 & 0.01 & 0.0910 & 0.2430 & 0.2553 & 0.1020 & 0.2660 & 0.2772 & \tabularnewline
	400 & 200 & 2 & 0.01 & 0.1625 & 0.5785 & 0.5972 & 0.1320 & 0.5575 & 0.5917 & \tabularnewline
	800 & 400 & 2 & 0.01 & 0.3695 & 0.9755 & 0.9714 & 0.2615 & 0.9785 & 0.9746 & \tabularnewline
	\cline{1-10} 
\end{tabular}	
	}
\end{table}


\begin{figure}[!h]
  \centering
  \includegraphics[width=\textwidth,trim={1.5cm 1cm 1.5cm 1cm},clip]{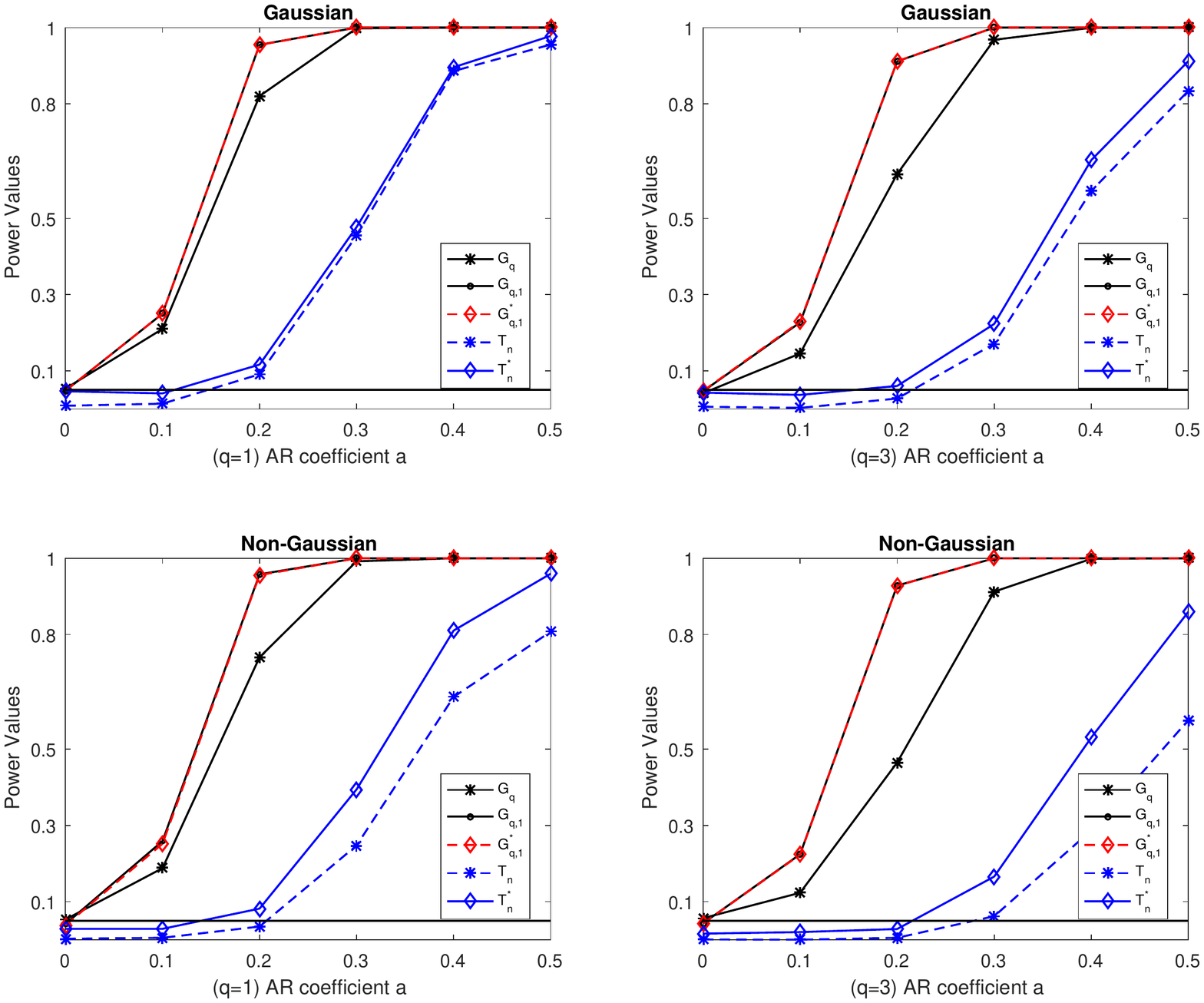}\\
  \includegraphics[width=\textwidth,trim={1.5cm 1cm 1.5cm 1cm},clip]{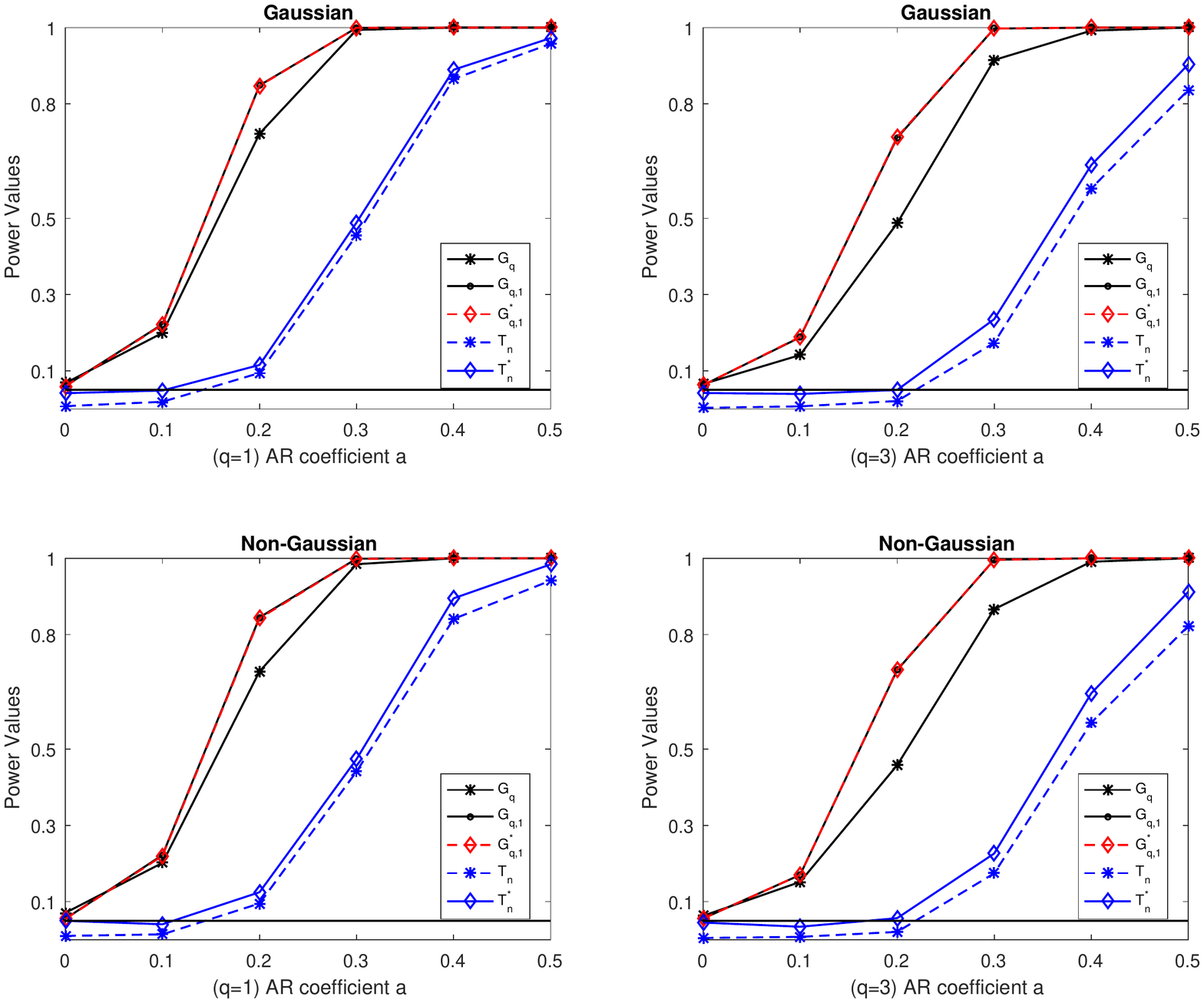}\\
  \caption{Power comparison under VAR(1)  with
    $(p,T)=(20,100)$.
    Left column with $q=1$ and right column with $q=3$. First two
    rows under alternative model (III); last two rows  under alternative model (IV). 
  }\label{Fig:SF2}
\end{figure}

\begin{figure}[!h]
  \centering
  \includegraphics[width=\textwidth,trim={1.5cm 1cm 1.5cm 1cm},clip]{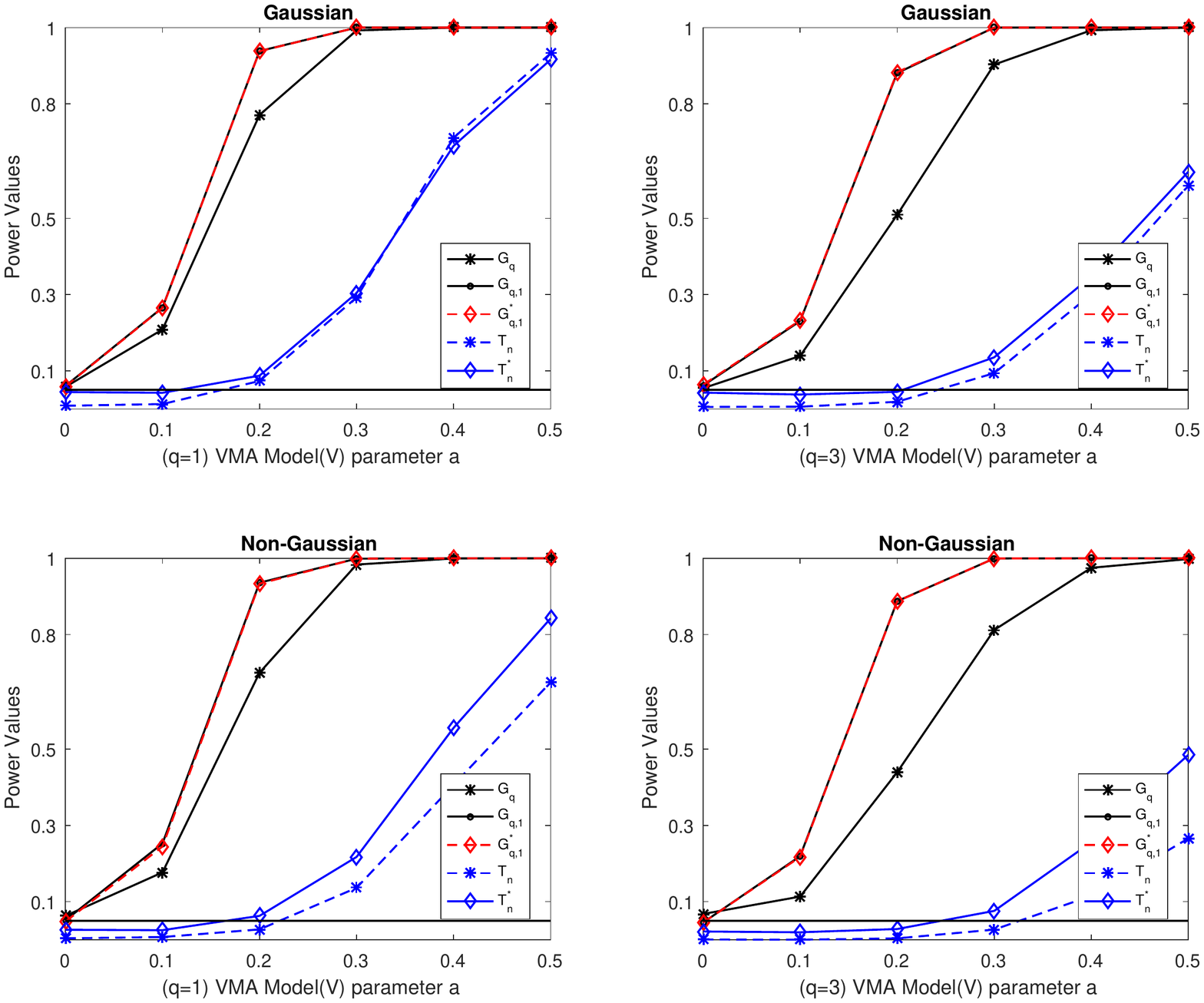}\\
  \includegraphics[width=\textwidth,trim={1.5cm 1cm 1.5cm 1cm},clip]{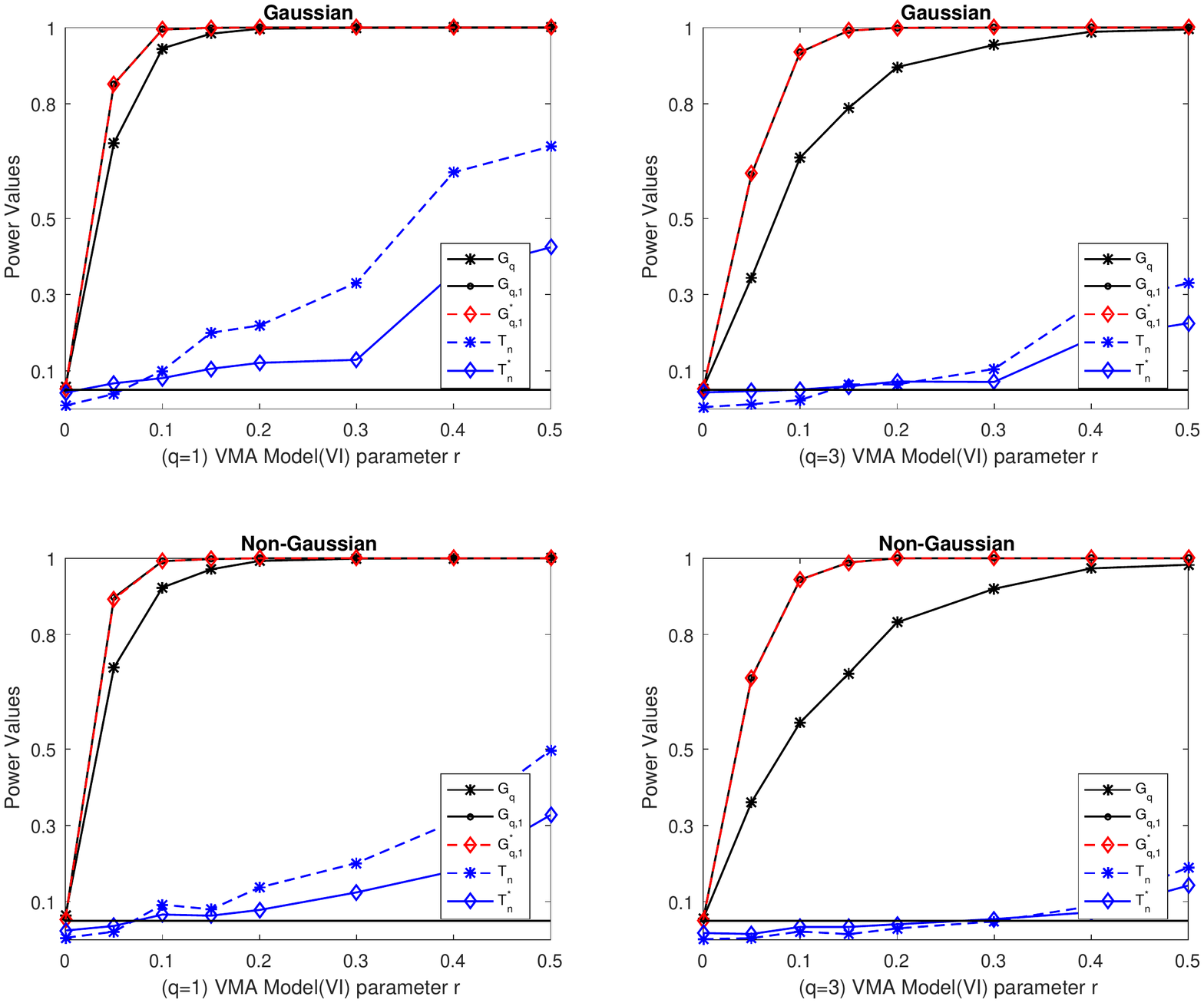}\\
  \caption{Power comparison under VMA(1) with
    $(p,T)=(20,100)$.
    Left column with $q=1$ and right column with $q=3$. First two
    rows under alternative model (V); last two rows  under alternative model (VI). 
  }\label{Fig:SF4}
\end{figure}


\clearpage

\appendix

\begin{center}
  {\bf SUPPLEMENT TO ``ON TESTING FOR
    HIGH-DIMENSIONAL WHITE NOISE''
  }
\end{center}

This supplemental article contains some technical
lemmas, the proof of Proposition 4.2 of the main article, and some
detailed tables of  simulation results which are discussed in the main paper.

\section{Technical lemmas}
\label{app:lemmas}

Hereafter, for a matrix $A = (a_{ij})$, define $\Re(A) = (\Re(a_{ij}))$ and $\Im(A) = (\Im(a_{ij}))$, where $\Re(z)$ and $\Im(z)$ are the real and imaginary parts of the complex number $z$.

\begin{lemma}\label{lem:1}
	Let $A$ be Hermitian. Then $\Tr(A\Re(A)) = \Tr(\Re^2(A))$.
\end{lemma}

\begin{proof}
Let $a_{ij}$ be the $(i,j)$th element of $A$. Then
\begin{align*}
\Tr(A \Re(A)) &=  \sum_{j,k} a_{jk}\Re(a_{kj}) = \sum_{j,k} \Re^2(a_{kj}) - i\sum_{j,k}\Re(a_{kj})\Im(a_{kj}) = \Tr(\Re^2(A)),
\end{align*}
where we used $\Re(a_{jk}) = \Re(a_{kj})$ and $\Im(a_{jk}) =
-\Im(a_{kj})$.
\end{proof}

\begin{lemma}\label{lem:2}
	Let $A$ be Hermitian and $\z_t$ has independent components with identical second and fourth order moments. Furthermore, $\E(z_{it}) = 0$, $\E|z_{it}|^2=1$, $\E|z_{it}|^4 = \nu_4 < \infty$ and $b = |\E(z_{it}^2)|^2$. Then
	\[\E(\z_t\z_t^*A\z_t\z_t^*) = \Tr(A)I_p + (\nu_4 - b - 2)\text{\em diag}(A) + 2\Re(A) + (b-1)A^T.\]
	If $A$ is symmetric, then by denoting the $k$th moment of $z_{it}$ as $\E(z^k)$, $k=2,4$,
	\[\E({\z}_t{\z}_t^TA\z_t\z_t^T) = \E^2(z^2)\Tr(A)I_p + (\E(z^4)-3\E^2(z^2))\text{\em diag}(A) + 2\E^2(z^2)A.\]
\end{lemma}

\begin{proof}
  The $(i,j)$th entry of $\E(\z_t\z_t^*A\z_t\z_t^*)$ is
\begin{align*}
\big(\E(\z_t\z_t^*A\z_t\z_t^*)\big)_{ij} &= \E\Big(\sum_{k,\ell}z_{it}z_{kt}^*a_{k\ell}z_{\ell t}z_{jt}^*\Big) = \sum_{k}a_{kk}\E(z_{it}|z_{kt}|^2z_{jt}^*) + \sum_{k\neq \ell}a_{k\ell}\E(z_{it}z_{kt}^*z_{\ell t}z_{j t}^*)\\
&= \left\{
\begin{array}{ll}
a_{ij} + a_{ji}\E(z_{it}^2)\E((z_{jt}^*)^2), & \hbox{$i\neq j$;} \\
(\nu_4-1)a_{ii} + \Tr(A), & \hbox{$i=j$}
\end{array}
                            \right.\\
  &= \left\{
\begin{array}{ll}
2\Re(a_{ij}) + a_{ji}(|\E(z_{it}^2)|^2-1), & \hbox{$i\neq j$;} \\
(\nu_4-1)a_{ii} + \Tr(A), & \hbox{$i=j$.}
\end{array}
\right.
\end{align*}
This completes the proof of the first part. For the second part,
\begin{align*}
\big(\E({\z}_t{\z}_t^TA \z_t\z_t^T)\big)_{ij} &= \E\Big(\sum_{k,\ell}{z}_{it}z_{kt}a_{k\ell}z_{\ell t}{z}_{jt}\Big) = \sum_{k}a_{kk}\E(z_{kt}^2{z}_{it}{z}_{jt}) + \sum_{k\neq \ell}a_{k\ell}\E({z}_{it}z_{kt}z_{\ell t}{z}_{jt})\\
&=\left\{
\begin{array}{ll}
(a_{ij} + a_{ji})\E^2(z^2), & \hbox{$i\neq j$;} \\
(\E(z^4)-\E^2(z^2))a_{ii} + \E^2(z^2)\Tr(A), & \hbox{$i=j$}
\end{array}
\right.\\ &=\left\{
\begin{array}{ll}
2a_{ij}\E^2(z^2), & \hbox{$i\neq j$;} \\
(\E(z^4)-\E^2(z^2))a_{ii} + \E^2(z^2)\Tr(A), & \hbox{$i=j$}
\end{array}
\right.
\end{align*}
This completes the proof of the lemma.
\end{proof}

\begin{lemma}\label{lem:3}
	Let the assumptions for $\z_t$ in Lemma \ref{lem:2} hold. Then for $s\neq t$ and $\tau=1,\ldots,q$,
	\begin{align*}
	V_1 &= \E(\z_t^*\Sigma_0\z_t) = \Tr(\Sigma_0),\\
	V_2 &= \E(\z_t^*\Sigma_0\z_t)^2 = \Tr^2(\Sigma_0) + (\nu_4-b-2)\Tr(\text{\em diag}^2(\Sigma_0)) + 2\Tr(\Re^2(\Sigma_0)) + (b-1)\Tr(\Sigma_0\Sigma_0^T),\\
	V_3 &= \E(\z_t^*\Sigma_0\z_s)^2 = b\Tr(\Sigma_0\Sigma_0^T),\\
	V_3' &= \E|\z_t^*\Sigma_0\z_s|^2 = \Tr(\Sigma_0^2),\\
	V_4 &= \E\big((\z_{t}^*\Sigma_0\z_{t+\tau})^2(\z_t^*\Sigma_0\z_{t-\tau})^2\big) \\
	&= b^2\Tr^2(\Sigma_0\Sigma_0^T) + (\E(\bar{z}^4)\E^2(z^2)-3b^2)\Tr(\text{\em diag}^2(\Sigma_0\Sigma_0^T)) + 2b^2\Tr(\Sigma_0\Sigma_0^T)^2,\\
	V_4' &= \E\big( |\z_{t+\tau}^*\Sigma_0\z_{t}|^2|\z_{t}^*\Sigma_0\z_{t-\tau}|^2 \big)\\
	&= \Tr^2(\Sigma_0^2) + (\nu_4-b-2)\Tr(\text{\em diag}^2(\Sigma_0^2)) + 2\Tr(\text{\em Re}^2(\Sigma_0^2)) + (b-1)\Tr(\Sigma_0^2(\Sigma_0^2)^T).
	\end{align*}
\end{lemma}

\begin{proof}
 We have
\[V_1 = \Tr[\E(\z_t\z_t^*)\Sigma_0] = \Tr(\Sigma_0).\]
For $V_2$, using Lemma \ref{lem:1} and the first part of Lemma \ref{lem:2},
\begin{align*}
V_2 &= \Tr[\E(\z_t\z_t^*\Sigma_0\z_t\z_t^*)\Sigma_0]\\
&= \Tr\left( \Tr(\Sigma_0)\Sigma_0 + (\nu_4-b-2)\diag(\Sigma_0)\Sigma_0 + 2\Re(\Sigma_0)\Sigma_0
+ (b-1)\Sigma_0^T\Sigma_0 \right)\\
&= \Tr^2(\Sigma_0) + (\nu_4-b-2)\Tr(\diag^2(\Sigma_0)) + 2\Tr(\Re^2(\Sigma_0)) + (b-1)\Tr(\Sigma_0\Sigma_0^T).
\end{align*}
Also,
\[V_3 = \E(\z_t^*\Sigma_0\z_s\z_s^T\Sigma_0^T\bar{\z}_t) = \Tr(\E(\bar{\z}_t\z_t^*)\Sigma_0\E(\z_s\z_s^T)\Sigma_0^T) = b\Tr(\Sigma_0\Sigma_0^T).\]
For $V_3'$,
\begin{align*}
V_3' &= \E(\z_t^*\Sigma_0\z_s\z_s^*\Sigma_0\z_t) = \Tr(\E(\z_t\z_t^*)\Sigma_0\E(\z_s\z_s^*)\Sigma_0) = \Tr(\Sigma_0^2).
\end{align*}

Using the second part of Lemma \ref{lem:2} with $A = \Sigma_0{\z}_{t+\tau}\z_{t+\tau}^T\Sigma_0^T$ and $z$ replaced by $\bar{z}$,
\begin{align*}
V_4 &= \E(\z_{t}^*\Sigma_0\z_{t+\tau}\z_{t+\tau}^T\Sigma_0^T\bar{\z}_t\z_t^*\Sigma_0\z_{t-\tau}\z_{t-\tau}^T\Sigma_0^T\bar{\z}_{t}) = \Tr\big(\E(\bar{\z}_t\bar{\z}_t^T\Sigma_0{\z}_{t+\tau}\z_{t+\tau}^T\Sigma_0^T\bar{\z}_t\bar{\z}_t^T)\Sigma_0\E(\z_{t-\tau}\z_{t-\tau}^T)\Sigma_0^T\big)\\
&= \E(z^2)\Tr\Big( \E^2(\bar{z}^2)\Tr(\E(\Sigma_0{\z}_{t+\tau}\z_{t+\tau}^T\Sigma_0^T))\Sigma_0\Sigma_0^T + (\E(\bar{z}^4)-3\E^2(\bar{z}^2))\diag(\E(\Sigma_0{\z}_{t+\tau}\z_{t+\tau}^T\Sigma_0^T))\Sigma_0\Sigma_0^T\\
&\quad\quad\quad\quad\quad + 2\E^2(\bar{z}^2)\E(\Sigma_0{\z}_{t+\tau}\z_{t+\tau}^T\Sigma_0^T)\Sigma_0\Sigma_0^T \Big)\\
&= b^2\Tr^2(\Sigma_0\Sigma_0^T) + (\E(\bar{z}^4)\E^2(z^2)-3b^2)\Tr(\diag^2(\Sigma_0\Sigma_0^T)) + 2b^2\Tr(\Sigma_0\Sigma_0^T)^2.
\end{align*}

Finally, using the first part of Lemma \ref{lem:2} with $A = \Sigma_0\z_{t+\tau}\z_{t+\tau}^*\Sigma_0$,
\begin{align*}
V_4' &= \E(\z_t^*\Sigma_0\z_{t+\tau}\z_{t+\tau}^*\Sigma_0\z_t\z_t^*\Sigma_0\z_{t-\tau}\z_{t-\tau}^*\Sigma_0\z_t)
= \Tr\big( \E(\z_t\z_t^*\Sigma_0\z_{t+\tau}\z_{t+\tau}^*\Sigma_0\z_t\z_t^*)\Sigma_0\E(\z_{t+\tau}\z_{t+\tau}^*)\Sigma_0 \big)\\
&= \Tr\Big( \Tr(\E(\Sigma_0\z_{t+\tau}\z_{t+\tau}^*\Sigma_0))\Sigma_0^2 + (\nu_4-b-2)\diag(\E(\Sigma_0\z_{t+\tau}\z_{t+\tau}^*\Sigma_0))\Sigma_0^2\\
&\quad\quad\quad\; + 2\Re(\E(\Sigma_0\z_{t+\tau}\z_{t+\tau}^*\Sigma_0))\Sigma_0^2 + (b-1)(\E(\Sigma_0\z_{t+\tau}\z_{t+\tau}^*\Sigma_0))^T\Sigma_0^2 \Big)\\
&=\Tr^2(\Sigma_0^2) + (\nu_4-b-2)\Tr(\diag^2(\Sigma_0^2)) + 2\Tr(\Re^2(\Sigma_0^2)) + (b-1)\Tr(\Sigma_0^2(\Sigma_0^2)^T).
\end{align*}
The proof is now complete.
\end{proof}

\begin{lemma}\label{lem:4}
  Let the assumptions for $\z_t$ in Lemma \ref{lem:2} hold. Then
  $G_q=\sum_{\tau=1}^q\Tr(\wh\Sigma_{\tau}\wh\Sigma_{\tau}^*)$ has
  expectation and variance given by
  \begin{align*}
    \E (G_q) & = q V_1^2/T, \\
    \text{\em Var}(G_q) & = \frac{q(V_2-V_1^2)^2+2q(\Re(V_4) + V_4') +
      q(T-2)(V_3^2 + (V_3')^2)}{T^3} +
    \frac{4q^2V_1^2(V_2-V_1^2)}{T^3}.
  \end{align*}
\end{lemma}

\begin{proof}
  Write 
  \[     G_q = \sum_{\tau=1}^q Q_\tau, \quad \text{with}\quad Q_\tau = \frac{1}{T^2}\sum_{t,s=1}^T\E(\x_s^*\x_t\x_{t-\tau}^*\x_{s-\tau}).
  \] 
Since $\x_t = \Sigma_0^{1/2}\z_t$ with the $\z_t$'s independent of each other, we have
\begin{align*}
\E(Q_{\tau}) & =\frac{1}{T}\E(\x_t^*\x_t\x_{t-\tau}^*\x_{t-\tau}) = \frac{1}{T}\E^2(\x_t^*\x_t) = \frac{\Tr^2(\Sigma_0)}{T} = V_1^2/T.
\end{align*}
The value of $\E (G_q)$ follows.

To evaluate $\E(Q_\tau^2)$, observe that
\begin{align*}
\E(Q_\tau^2) &= \frac{1}{T^4}\sum_{t_1,t_2,s_1,s_2=1}^T \E(\z_{s_1}^*\Sigma_0\z_{t_1}\z_{t_1-\tau}^*\Sigma_0\z_{s_1-\tau}\z_{s_2}^*\Sigma_0\z_{t_2}\z_{t_2-\tau}^*\Sigma_0\z_{s_2-\tau}).
\end{align*}
Let $E(t_1,s_1,t_2,s_2)$ be the expectation on the right hand side above. We detail the cases where this expectation is non-zero below.
\begin{itemize}
	\item[(I)] $s_1=t_1$, $s_2=t_2$. Sub-cases:
	\begin{itemize}
		\item[i.] $t_1=t_2(=t)$:
		
		$T^{-4}\sum_{t=1}^TE(t,t,t,t) = \frac{1}{T^4}\sum_{t=1}^T\E(\z_t^*\Sigma_0\z_t)^2\E(\z_{t-\tau}^*\Sigma_0\z_{t-\tau})^2 = \frac{1}{T^4}\sum_{t=1}^TV_2^2 = V_2^2/T^3$;
		
		\item[ii.] $t_1=t_2-\tau(=t)$:
		
		$T^{-4}\sum_{t=1}^TE(t,t,t+\tau,t+\tau) = \frac{1}{T^4}\sum_{t=1}^T\E(\z_t^*\Sigma_0\z_t)^2\E(\z_{t+\tau}^*\Sigma_0\z_{t+\tau})\E(\z_{t-\tau}^*\Sigma_0\z_{t-\tau})
		=V_2V_1^2/T^3$.
		
		\item[iii.] $t_1=t_2+\tau(=t)$:
		
		$T^{-4}\sum_{t=1}^TE(t,t,t-\tau,t-\tau) = V_2V_1^2/T^3$.
		
		\item[iv.] Otherwise:
		\begin{align*}
          &		T^{-4}\sum_{s\neq t,t+\tau,t-\tau}E(t,t,s,s) \\
          &= \frac{1}{T^4}\sum_{s\neq t,t+\tau,t-\tau}^T\E(\z_{t}^*\Sigma_0\z_{t})\E(\z_{t-\tau}^*\Sigma_0\z_{t-\tau})\E(\z_{s}^*\Sigma_0\z_{s})\E(\z_{s-\tau}^*\Sigma_0\z_{s-\tau})\\
		&=(T^2-3T)V_1^4/T^4.
		\end{align*}
	\end{itemize}
	
	\item[(II)] $s_1=s_2, t_1=t_2$. Sub-cases (not overlapping with (I))
	\begin{itemize}
		\item[i.] $s_1=t_1+\tau(=t)$:
		
		$T^{-4}\sum_{t=1}^TE(t-\tau,t,t-\tau,t) = \frac{1}{T^4}\sum_{t=1}^T\E[(\z_{t-2\tau}^*\Sigma_0\z_{t-\tau})^2(\z_{t}^*\Sigma_0\z_{t-\tau})^2] = \bar{V}_4/T^3$.
		
		\item[ii.] $s_1=t_1-\tau(=t)$:
		
		$T^{-4}\sum_{t=1}^TE(t+\tau,t,t+\tau,t) = V_4/T^3$.
		
		\item[iii.] Otherwise:
		
		$T^{-4}\sum_{s\neq t+\tau,t-\tau}E(t,s,t,s) = (T^2-2T)V_3^2/T^4$.
	\end{itemize}

	\item[(III)] $s_1=t_2, s_2=t_1$. Sub-cases (not overlapping with (I) and (II))
	\begin{itemize}
		\item[i.] $s_2=t_2+\tau(=t)$:
		
		$T^{-4}\sum_{t=1}^TE(t,t-\tau,t-\tau,t) = V_4'/T^3$.
		
		\item[ii.] $s_2=t_2-\tau(=t)$:
		
		$T^{-4}\sum_{t=1}^TE(t,t+\tau,t+\tau,t) = V_4'/T^3$.
		
		\item[iii.] Otherwise:
		
		$T^{-4}\sum_{s\neq t+\tau,t-\tau}E(t,s,s,t) = (T^2-2T)(V_3')^2/T^4$.
	\end{itemize}
\end{itemize}

With the above,
\begin{align*}
\E(Q_{\tau}^2) &= \frac{V_2^2 + 2V_1^2V_2 + (T-3)V_1^4}{T^3} + \frac{2\Re(V_4) + (T-2)V_3^2}{T^3}
+ \frac{2V_4' + (T-2)(V_3')^2}{T^3}, \text{ so that }\\
\var(Q_\tau) &= \frac{V_2^2 + 2V_1^2V_2 -3V_1^4 + 2(\Re(V_4) + V_4') + (T-2)(V_3^2+(V_3')^2)}{T^3}.
\end{align*}

For $k \neq \ell$ and both $k,\ell$ are non-zero, consider
\begin{align*}
\E(Q_kQ_{\ell}) &= \frac{1}{T^4}\sum_{t_1,s_1,t_2,s_2=1}^T\E(\z_{s_1}^*\Sigma_0\z_{t_1}\z_{t_1-k}^*\Sigma_0\z_{s_1-k}\z_{s_2}^*\Sigma_0\z_{t_2}\z_{t_2-\ell}^*\Sigma_0\z_{s_2-\ell})\\
&= \frac{1}{T^4}\sum_{t,s=1}^T\E(\z_{s}^*\Sigma_0\z_{s}\z_{s-k}^*\Sigma_0\z_{s-k}\z_{t}^*\Sigma_0\z_{t}\z_{t-\ell}^*\Sigma_0\z_{t-\ell})\\
&= \frac{4V_2V_1^2}{T^3} + \frac{(T^2-4T)V_1^4}{T^4},
\end{align*}
where the last equality used the fact that when $t=s,t=s-k, s=t-\ell$ or $s-k=t-\ell$, the resulting expectation in the summation in the second equality is $V_2V_1^2$, otherwise it is $V_1^4$. Hence
\begin{align*}
\cov(Q_k,Q_{\ell}) &= \frac{4V_1^2(V_2-V_1^2)}{T^3}.
\end{align*}
With the above, we have
\begin{align*}
\var(G_q) &= \sum_{\tau=1}^q\var(Q_\tau) + \sum_{k\neq\ell}\cov(Q_k,Q_\ell)\\
&= q\var(Q_\tau) + q(q-1)\cov(Q_k,Q_\ell)\\
&= \frac{q(V_2-V_1^2)^2+2q(\Re(V_4) + V_4') + q(T-2)(V_3^2 + (V_3')^2)}{T^3} + \frac{4q^2V_1^2(V_2-V_1^2)}{T^3}.
\end{align*}
This completes the proof of the lemma.
\end{proof}

\begin{lemma}\label{lem:V5V6}
		Assume that $\Sigma_0$ is semipositive definite and $\z_t$ has real-valued independent components with identical eight order moments, $\E(z_{it}) = 0$, $\E|z_{it}|^2=1$, $\E|z_{it}|^k = \nu_k < \infty,~k=1,\cdots,8$. Let $\imath$ denotes a $p\times 1$ vector with unit elements and $\odot$ denote Hadamard product, i.e. for any  two given matrices $A,~ B$ with same size,  $\lb A\odot B\rb_{ij}=A_{ij}B_{ij}$,  then
\begin{align*}
	V_5 &= \E(\z_t^*\Sigma_0\z_t)^3 = \Tr^3(\Sigma_0)+6\Tr(\Sigma_0)\Tr(\Sigma_0^2)+8\Tr(\Sigma_0^3)\\
	&+(\nu_4-3)\left[3\Tr\lb D^2(\Sigma_0)\rb\Tr(\Sigma_0)+4\Tr\lb  D(\Sigma_0^2)\Sigma_0\rb+8\Tr\lb D(\Sigma_0)\Sigma_0^2\rb\right]\\
	&+\nu_3^2\left[ 4\Tr\lb \Sigma_0(\Sigma_0\odot \Sigma_0)\rb+2\imath^* D(\Sigma_0)\Sigma_0D(\Sigma_0)\imath +4\imath^*  D^2(\Sigma_0)\Sigma_0\imath\right]\\
	&+\lb \nu_6-10\nu_3^2-15(\nu_4-3)-15\rb\Tr\lb \Sigma_0\lb \Sigma_0\odot \Sigma_0\rb\rb,\\
V_6 &= \E(\z_t^*\Sigma_0\z_t)^4=\Tr^4(\Sigma_0)+12\Tr(\Sigma_0^2)\Tr^2(\Sigma_0)+12\Tr^2(\Sigma_0^2)+32\Tr(\Sigma_0)\Tr(\Sigma_0^3)\\
&+48\Tr(\Sigma_0^4)+(\nu_4-3)\left\{6\Tr^2(\Sigma_0)\Tr(\Sigma_0\odot\Sigma_0)+12\Tr(\Sigma_0^2)\Tr(\Sigma_0\odot\Sigma_0)\right.\\
&\left. +48\Tr(\Sigma_0)\Tr(\Sigma_0\odot\Sigma_0^2)+48\Tr(D^2(\Sigma_0^2))+96\Tr(D(\Sigma_0)\Sigma_0^3)\right\}\\
&+(\nu_4-3)^2\left\{3\Tr^2(\Sigma_0\odot\Sigma_0)+24\imath^* D(\Sigma_0)\lb\Sigma_0 \odot \Sigma_0\rb D(\Sigma_0)\imath+8\imath^*\lb\Sigma_0\odot\Sigma_0\odot\Sigma_0\odot\Sigma_0\rb\imath\right\}
\\
&+\lb\nu_6-15(\nu_4-3)-10\nu_3^2-15\rb\left\{4\Tr(\Sigma_0)\Tr(\Sigma_0\odot\Sigma_0\odot\Sigma_0)+24\Tr\lb\Sigma_0\odot\Sigma_0\odot(\Sigma_0^2)\rb\right\}\\
&+2\nu_3^2 \left\{ 12\imath^*\lb D(\Sigma_0)\Sigma_0D(\Sigma_0)\rb\imath\cdot \Tr(\Sigma_0) +24 \imath^*\lb D(\Sigma_0)\Sigma_0^2D(\Sigma_0)\rb\imath \right.\\
&\left. +8\imath^*\lb \Sigma_0\odot\Sigma_0\odot\Sigma_0\rb\imath\cdot\Tr(\Sigma_0) +48\imath^*(\Sigma_0\odot\Sigma_0)\Sigma_0D(\Sigma_0)\imath+48\Tr\lb\Sigma_0^2\lb \Sigma_0\odot\Sigma_0\rb\rb\right\}\\
&+2\nu_3(\nu_5-10\nu_3)\left\{12\imath^* \lb D(\Sigma_0)\Sigma_0D^2\lb \Sigma_0\rb\rb\imath+16\imath^*\lb\Sigma_0\odot\Sigma_0\odot\Sigma_0\rb D(\Sigma_0) \imath\right\}+\\
&\lb \nu_8-28\nu_6+210(\nu_4-3)-35(\nu_4-3)^2-56\nu_3(\nu_5-10\nu_3)+315\rb\Tr\lb \Sigma_0\odot\Sigma_0\odot\Sigma_0\odot\Sigma_0\rb,\\
V_7&=\mathbb{E}\lb \z_s^*\Sigma_0\z_{s+\tau}\z_{s+\tau}^*\Sigma_0\z_{s+\tau}\z_s^*\Sigma_0\z_{s-\tau}\z_{s-\tau}^*\Sigma_0\z_{s-\tau}\rb=\nu_3^2\cdot\imath^*\lb D(\Sigma_0)\Sigma_0^2D(\Sigma_0)\rb \imath.
\end{align*}
\end{lemma}

\begin{proof}
  Moments of quadratic forms are well studied in the fields of econometrics and statistics. Specifically, there have been long interest in deriving $\mathbb{E}\lb \prod_{i=1}^n Q_i\rb$, where $Q_i=y^*A_iy$, $A_i$ are $p\times p$ non-stochastic symmetric matrices and  $y$ is an $p\times 1$ random vector with mean $\mu$ and identity covariance matrix. $y^*$ represents transpose of $y$. Both $V_5$ and $V_6$ are moments of quadratic forms as a special case where all $A_i's$ equal to $\Sigma_0$, thus we can write down the results by directly referring to calculations and techniques presented in previous works of \citep{U04} and \citep{BU10}.  As for $V_7$,
  \begin{align*}
    V_7&=\E\lb \z_{s+\tau}^*\Sigma_0\z_{s+\tau}\z_{s+\tau}^*\rb\Sigma_0\E(\z_s\z_s^*)\Sigma_0\E \lb\z_{s-\tau}\z_{s-\tau}^*\Sigma_0\z_{s-\tau} \rb\\
         &=\nu_3^2\cdot\imath^*\lb D(\Sigma_0)\Sigma_0^2D(\Sigma_0)\rb \imath,
  \end{align*}
  because  according to \cite{U04}, 
  $\E \lb\z_{t}\z_{t}^*\Sigma_0\z_{t} \rb=\nu_3\lb {\bf I}_p\odot\Sigma_0 \rb\imath.$ 

\end{proof}

\begin{lemma}\label{lem:cov}
	Assume the same conditions as in Lemma~\ref{lem:V5V6}, then $p\hat{s}_1^2=\frac{1}{p}\Tr^2(\Sigma_0)$ has expectation and variance given by
	\begin{align*}
	\mathbb{E}(p\hat{s}_1^2)&=\frac{V_1^2}{p}-\frac{1}{pT}\lb V_1^2-V_2\rb,\\
	\var(p\hat{s}_1^2)&=\frac{1}{p^2T^3}V_6+\lb \frac{4}{p^2T^2}- \frac{4}{p^2T^3}\rb V_1V_5+\lb \frac{2}{p^2T^2}- \frac{3}{p^2T^3}\rb V_2^2\\
	&+\lb \frac{4}{p^2T}- \frac{16}{p^2T^2}+\frac{12}{p^2T^3}\rb V_1^2V_2+\lb -\frac{4}{p^2T}+\frac{10}{p^2T^2}-\frac{6}{p^2T^3}\rb V_1^4.
	\end{align*}
	The covariance between $G_q$ and $p\hat{s}_1^2$ is
	\begin{align*}
	\cov \lb G_q,p\hat{s}_1^2\rb&=\lb \frac{4q}{pT^2}- \frac{10q}{pT^3}\rb V_1^2\lb V_2-V_1^2\rb-\frac{4q}{pT^3}V_1^4+\frac{2q}{pT^3}V_1V_5\\
	&+\frac{2q}{pT^3}V_2^2+\frac{4q}{pT^3}V_7.
	\end{align*}
\end{lemma}

\begin{proof}
	Write
	\[p\hat{s}_1^2=\frac{1}{pT^2}\sum_{s,t=1}^T\x_t^*\x_t\x_s^*\x_s,\]
	since $\x_t=\Sigma_0^{1/2}\z_t$ with $\z_t's$ independent of each other,
	\begin{align*}
	\mathbb{E}\lb p\hat{s}_1^2\rb&=\frac{1}{pT^2}\sum_{t=1}^T\mathbb{E}\lb \x_t^*\x_t\x_t^*\x_t\rb+\frac{1}{pT^2}\sum_{t\neq s}\mathbb{E}(\x_t^*\x_t)\mathbb{E}(\x_s^*\x_s)\\
	&=\frac{1}{pT}V_2+\lb\frac{1}{p}-\frac{1}{pT}\rb V_1^2.
	\end{align*}
	To evaluate $\mathbb{E}(p^2\hat{s}_1^4)$, observe that
	\[\mathbb{E}(p^2\hat{s}_1^4)=\frac{1}{p^2T^4}\sum_{t_1,t_2,s_1,s_2=1}^T\mathbb{E}\lb \z_{t_1}^*\Sigma_0\z_{t_1}\z_{s_1}^*\Sigma_0\z_{s_1}\z_{t_2}^*\Sigma_0\z_{t_2}\z_{s_2}^*\Sigma_0\z_{s_2}\rb.\]
Denote $\mathbb{E}\lb \z_{t_1}^*\Sigma_0\z_{t_1}\z_{s_1}^*\Sigma_0\z_{s_1}\z_{t_2}^*\Sigma_0\z_{t_2}\z_{s_2}^*\Sigma_0\z_{s_2}\rb$ by $F(t_1,s_1,t_2,s_2)$,	we detail the cases where $F(t_1,s_1,t_2,s_2)$ is non-zero as follows.
\begin{itemize}
	\item[(I)] $t_1=s_1=t_2=s_2$:
		\[\frac{1}{p^2T^4}\sum_{t=1}^T F(t,t,t,t)=\frac{T}{p^2T^4}\E(\z_t^*\Sigma_0\z_t)^4=\frac{1}{p^2T^3}V_6;\]
	
	\item[(II)] $t_1=s_1=t_2\neq s_2$:
			\begin{align*}
		\frac{4}{p^2T^4}\sum_{t_1\neq t_2}^T F(t_1,t_1,t_1,t_2)&=\frac{4(T^2-T)}{p^2T^4}\E(\z_t^*\Sigma_0\z_t)\E(\z_t^*\Sigma_0\z_t)^3\\
		&=\lb \frac{4}{p^2T^2}-\frac{4}{p^2T^3}\rb V_1V_5;
		\end{align*}
		
				\item[(III)] $t_1=s_1\neq t_2= s_2$:
		\begin{align*}
		\frac{3}{p^2T^4}\sum_{t_1\neq t_2}^T F(t_1,t_1,t_2,t_2)&=\frac{3(T^2-T)}{p^2T^4}\E(\z_t^*\Sigma_0\z_t)^2\E(\z_t^*\Sigma_0\z_t)^2\\
		&=\lb \frac{3}{p^2T^2}-\frac{3}{p^2T^3}\rb V_2^2;
		\end{align*}

	\item[(IV)] $t_1=s_1\neq t_2\neq s_2$:
				\begin{align*}
	\frac{6}{p^2T^4}\sum_{t_1\neq t_2\neq s_2}^T F(t_1,t_1,t_2,s_2)&=\frac{6T(T-1)(T-2)}{p^2T^4}\E(\z_t^*\Sigma_0\z_t)^2\E(\z_t^*\Sigma_0\z_t)\E(\z_t^*\Sigma_0\z_t)\\
	&=\lb \frac{6}{p^2T}-\frac{18}{p^2T^2}+\frac{12}{p^2T^3}\rb V_2V_1^2;
		\end{align*}

\item[(V)] Otherwise, $t_1\neq s_1\neq t_2\neq s_2$:
	\begin{align*}
\frac{1}{p^2T^4}\sum_{t_1\neq s_1\neq t_2\neq s_2}^T F(t_1,s_1,t_2,s_2)&=\frac{T^4-6T(T-1)(T-2)-7T(T-1)-T}{p^2T^4}\lb\E(\z_t^*\Sigma_0\z_t)\rb^4\\
&=\lb \frac{1}{p^2}-\frac{6}{p^2T}+\frac{11}{p^2T^2}-\frac{6}{p^2T^3}\rb V_1^4.
\end{align*}
	
\end{itemize}

With the above,
	\begin{align*}
\mathbb{E}(p\hat{s}_1^2)&=\frac{V_1^2}{p}-\frac{1}{pT}\lb V_1^2-V_2\rb,\\
\var(p\hat{s}_1^2)&=\frac{1}{p^2T^3}V_6+\lb \frac{4}{p^2T^2}- \frac{4}{p^2T^3}\rb V_1V_5+\lb \frac{2}{p^2T^2}- \frac{3}{p^2T^3}\rb V_2^2\\
&+\lb \frac{4}{p^2T}- \frac{16}{p^2T^2}+\frac{12}{p^2T^3}\rb V_1^2V_2+\lb -\frac{4}{p^2T}+\frac{10}{p^2T^2}-\frac{6}{p^2T^3}\rb V_1^4.
\end{align*}

As for covariance between $G_q$ and $p\hat{s}_1^2$, since
 \[     G_q = \sum_{\tau=1}^q Q_\tau,~ Q_\tau = \frac{1}{T^2}\sum_{t,s=1}^T\E(\x_s^*\x_t\x_{t-\tau}^*\x_{s-\tau}),~ \E(Q_\tau)=V_1^2/T,
\] 
\[\cov\lb G_q, p\hat{s}_1^2\rb=q\lb \E(Q_\tau\cdot p\hat{s}_1^2)-\E(Q_\tau)\E(p\hat{s}_1^2)\rb, \]
we only need to consider
\[\E(Q_\tau\cdot p\hat{s}_1^2)=\frac{1}{pT^4}\sum_{t_1,t_2,s_1,s_2=1}^T \mathbb{E}\lb \z_{s_1}^*\Sigma_0\z_{t_1}\z_{t_1-\tau}^*\Sigma_0\z_{s_1-\tau}\z_{t_2}^*\Sigma_0\z_{t_2}\z_{s_2}^*\Sigma_0\z_{s_2}\rb. \]
In the following we detail the cases where the expectation on the right hand side above is non-zero.

\begin{itemize}
	\item[(I)] $s_1=t_1$:  $\mathbb{E}\lb \z_{s_1}^*\Sigma_0\z_{t_1}\z_{t_1-\tau}^*\Sigma_0\z_{s_1-\tau}\z_{t_2}^*\Sigma_0\z_{t_2}\z_{s_2}^*\Sigma_0\z_{s_2}\rb=F(t_1,t_1-\tau,t_2,s_2)$. Sub-cases:
	\begin{itemize}
		\item[(i)] $t_1=t_2=s_2$:		
		\[\frac{T}{pT^4} F(t,t-\tau,t,t)=\frac{1}{pT^3}V_1V_5;\]
		
			\item[(ii)] $t_1-\tau=t_2=s_2$:
		\[\frac{T}{pT^4} F(t,t-\tau,t-\tau,t-\tau)=\frac{1}{pT^3}V_1V_5;\]
			\item[(iii)] $t_1=t_2,~ t_1-\tau=s_2$:		
		\[\frac{T}{pT^4} F(t,t-\tau,t,t-\tau)=\frac{1}{pT^3}V_2^2;\]
		
		\item[(iv)] $t_1-\tau=t_2~, t_1=s_2$:
		\[\frac{T}{pT^4} F(t,t-\tau,t-\tau,t)=\frac{1}{pT^3}V_2^2;\]
		
		\item[(v)] $t_1=t_2,~ s_2\neq t_1,~ s_2\neq t_1-\tau$:
		\[\frac{T(T-2)}{pT^4} F(t_1,t_1-\tau,t_1,s_2)=\lb\frac{1}{pT^2}-\frac{2}{pT^3}\rb V_2V_1^2;\]
		
				\item[(vi)] $t_1=s_2,~ t_2\neq t_1,~ t_2\neq t_1-\tau$:
		\[\frac{T(T-2)}{pT^4} F(t_1,t_1-\tau,t_2,t_1)=\lb\frac{1}{pT^2}-\frac{2}{pT^3}\rb V_2V_1^2;\]
		
				\item[(vii)] $t_1-\tau=t_2,~ s_2\neq t_1-\tau,~ s_2\neq t_1$:
		\[\frac{T(T-2)}{pT^4} F(t_1,t_1-\tau,t_1-\tau,s_2)=\lb\frac{1}{pT^2}-\frac{2}{pT^3}\rb V_2V_1^2;\]
			
				\item[(viii)] $t_1-\tau=s_2,~ t_2\neq t_1,~ t_2\neq t_1-\tau$:
		\[\frac{T(T-2)}{pT^4} F(t_1,t_1-\tau,t_2,t_1-\tau)=\lb\frac{1}{pT^2}-\frac{2}{pT^3}\rb V_2V_1^2;\]
				
				\item[(ix)] $t_2=s_2,~ t_2\neq t_1,~ t_2\neq t_1-\tau$:
		\[\frac{T(T-2)}{pT^4} F(t_1,t_1-\tau,t_2,t_2)=\lb\frac{1}{pT^2}-\frac{2}{pT^3}\rb V_2V_1^2;\]
		
		\item[(x)] Otherwise, $t_1\neq t_1-\tau \neq t_2\neq s_2$:
		\[\frac{T^3-4T-5T(T-2)}{pT^4} F(t_1,t_1-\tau,t_2,s_2)=\lb\frac{1}{pT}-\frac{5}{pT^2}+\frac{6}{pT^3}\rb V_1^4;\]
		
	\end{itemize}

\item[(II)] $s_1\neq t_1, s_1=t_1-\tau$: $\quad \mathbb{E}\lb \z_{s_1}^*\Sigma_0\z_{s_1+\tau}\z_{s_1}^*\Sigma_0\z_{s_1-\tau}\z_{t_2}^*\Sigma_0\z_{t_2}\z_{s_2}^*\Sigma_0\z_{s_2}\rb$.

Subcases:

\begin{itemize}
	\item[(i)] $t_2=s_1-\tau, s_2=s_1+\tau$:
	\[\frac{T}{pT^4}\E\lb \z_{s_1}^*\Sigma_0\z_{s_1+\tau}\z_{s_1}^*\Sigma_0\z_{s_1-\tau}\z_{s_1-\tau}^*\Sigma_0\z_{s_1-\tau}\z_{s_1+\tau}^*\Sigma_0\z_{s_1+\tau}\rb=\frac{1}{pT^3}V_7; \]
		\item[(ii)] $s_2=s_1-\tau, t_2=s_1+\tau$:
	\[\frac{T}{pT^4}\E\lb \z_{s_1}^*\Sigma_0\z_{s_1+\tau}\z_{s_1}^*\Sigma_0\z_{s_1-\tau}\z_{s_1+\tau}^*\Sigma_0\z_{s_1+\tau}\z_{s_1-\tau}^*\Sigma_0\z_{s_1-\tau}\rb=\frac{1}{pT^3}V_7; \]
\end{itemize}

\item[(III)] $s_1\neq t_1, t_1=s_1-\tau$: $\quad \mathbb{E}\lb \z_{s_1}^*\Sigma_0\z_{s_1-\tau}\z_{s_1-2\tau}^*\Sigma_0\z_{s_1-\tau}\z_{t_2}^*\Sigma_0\z_{t_2}\z_{s_2}^*\Sigma_0\z_{s_2}\rb$.

Subcases:

\begin{itemize}
	\item[(i)] $t_2=s_1, s_2=s_1-2\tau$:
	\[\frac{T}{pT^4}\mathbb{E}\lb \z_{s_1}^*\Sigma_0\z_{s_1-\tau}\z_{s_1-2\tau}^*\Sigma_0\z_{s_1-\tau}\z_{s_1}^*\Sigma_0\z_{s_1}\z_{s_1-2\tau}^*\Sigma_0\z_{s_1-2\tau}\rb=\frac{1}{pT^3}V_7; \]
	\item[(ii)] $t_2=s_1-2\tau, s_2=s_1$:
	\[\frac{T}{pT^4}\mathbb{E}\lb \z_{s_1}^*\Sigma_0\z_{s_1-\tau}\z_{s_1-2\tau}^*\Sigma_0\z_{s_1-\tau}\z_{s_1-2\tau}^*\Sigma_0\z_{s_1-2\tau}\z_{s_1}^*\Sigma_0\z_{s_1}\rb=\frac{1}{pT^3}V_7; \]
\end{itemize}
	
\end{itemize} 

With all the above, we have, the covariance between $Q_\tau$ and $p\hat{s}_1^2$ is
\begin{align*}
\cov \lb Q_\tau,p\hat{s}_1^2\rb&=\lb \frac{4}{pT^2}- \frac{10}{pT^3}\rb V_1^2\lb V_2-V_1^2\rb-\frac{4}{pT^3}V_1^4+\frac{2}{pT^3}V_1V_5\\
&+\frac{2}{pT^3}V_2^2+\frac{4}{pT^3}V_7.
\end{align*}

This completes the proof of the lemma.

\end{proof}

\section{Proof of Proposition 4.2}

Considering $\Sigma_0$ is with bounded spectral norm, by implementing the results in Lemma~\ref{lem:3} and Lemma~\ref{lem:V5V6} to Lemma~\ref{lem:4} and  Lemma~\ref{lem:cov}, we can evaluate the order of each term and select terms of orders $O(1)$ and $O(\frac{1}{T})$. Therefore, the leading order terms of $\E(p\hat{s}_1^2)$, $\var(p\hat{s}_1^2)$, $\E(G_q)$, $\var(G_q)$ and $\cov(G_q, p\hat{s}_1^2)$ can be selected out accordingly. 	
As for $\E(\hat{s}_2)$,	
\begin{align*}
  \E(\hat{s}_2)&=\frac{1}{pT^2}\sum_{t_1,t_2=1}^T\E\lb \z_{t_1}\Sigma_0\z_{t_2}\z_{t_2}^*\Sigma_0\z_{t_1}\rb\\
               &=\frac{1}{pT^2}\sum_{t=1}^T\E \lb \z_t^*\Sigma_0\z_t\rb^2+\frac{1}{pT^2}\sum_{t\neq s}\E |\z_t^*\Sigma_0\z_s|^2\\
               &=\frac{1}{pT}V_2+\lb\frac{1}{p}-\frac{1}{pT}\rb V_3'\\
               &=\frac{1}{p}\Tr(\Sigma_0^2)+\frac{1}{pT}\Tr^2(\Sigma_0)+\frac{1}{pT}\lb \Tr(\Sigma_0^2)+(\nu_4-3)\Tr(D^2(\Sigma_0))\rb.
\end{align*}	
This completes the proof of Proposition~4.2.  

\clearpage

\section{Empirical  results related to the Hosking test  $\tilde Q_q$  and 
the Li-McLeod test $Q_q^*$}

The following three tables are discussed in 
Sections 3.2 and 3.3 of the main paper.
Table~\ref{Tab:cLPower} reports on empirical 
powers  of the Hosking test $\tilde Q_q$ and 
the Li-McLeod test $Q_q^*$.

\begin{table}[hb]
  \caption{Power performance of the Hosking test  $\widetilde{Q}_{q}$ and the Li-McLeod test $Q_{q}^{*}$.\label{Tab:cLPower}}
  \begin{tabular}{ccc|cc|cc}
\hline
 &  &   & \multicolumn{2}{c|}{$\widetilde{Q}_{q}$} & \multicolumn{2}{c}{$Q_q^*$}\tabularnewline
\cline{4-7}
$p$ & $T$ & $p/T$   & $q=1$ & $q=3$ & $q=1$ & $q=3$\tabularnewline
\hline
10 & 100 & 0.1  & 0.0952 & 0.0952 & 0.0914 & 0.0926\tabularnewline
20 & 200 & 0.1   & 0.2392 & 0.1994 & 0.2362 & 0.1958\tabularnewline
40 & 400 & 0.1  & 0.6638 & 0.5410 & 0.6622 & 0.5380\tabularnewline
60 & 600 & 0.1  & 0.9406 & 0.8452 & 0.9404 & 0.8448\tabularnewline
100 & 1000 & 0.1  & 1 & 0.9982 & 1 & 0.9982\tabularnewline
\hline
50 & 100 & 0.5  & 0.0014 & 0.0060 & 0.0014 & 0.0052\tabularnewline
100 & 200 & 0.5   & 0.0036 & 0.0208 & 0.0030 & 0.0194\tabularnewline
200 & 400 & 0.5  & 0.0330 & 0.2022 & 0.0328 & 0.1994\tabularnewline
300 & 600 & 0.5  & 0.1156 & 0.6348 & 0.1138 & 0.6306\tabularnewline
500 & 1000 & 0.5  & 0.5816 & 0.9974 & 0.5806 & 0.9972\tabularnewline
\hline
80 & 100 & 0.8   & 0 & 0 & 0 & 0\tabularnewline
160 & 200 & 0.8  & 0 & 0 & 0 & 0\tabularnewline
320 & 400 & 0.8  & 0 & 0 & 0 & 0\tabularnewline
480 & 600 & 0.8  & 0 & 0 & 0 & 0\tabularnewline
800 & 1000 & 0.8 &  0.0004 & 0.0038 & 0.0004 & 0.0032\tabularnewline
\hline
90 & 100 & 0.9  & 0 & 0 & 0 & 0\tabularnewline
180 & 200 & 0.9 &  0 & 0 & 0 & 0\tabularnewline
360 & 400 & 0.9 &  0 & 0 & 0 & 0\tabularnewline
540 & 600 & 0.9 &  0 & 0 & 0 & 0\tabularnewline
900 & 1000 & 0.9 & 0 & 0 & 0 & 0\tabularnewline
\hline
\end{tabular}
\end{table}

Table~\ref{Tab:cLAdPower} reports adjusted powers of the Hosking test $\widetilde{Q}_{q}$ and the
Li-McLeod test $Q_{q}^{*}$
in comparison to 
powers of our tests $G_q$ and $G_{q,1}$ when $\Sigma_0={\bf I}_p$.


\begin{table}[!h]
\caption{Adjusted powers of the Hosking test $\widetilde{Q}_{q}$ and the
  Li-McLeod test $Q_{q}^{*}$
  as compared to
  powers of our tests $G_q$ and $G_{q,1}$ when $\Sigma_0={\bf I}_p$}\label{Tab:cLAdPower}

\resizebox{0.98\textwidth}{!}{
	\def\arraystretch{1.1}
\begin{tabular}{cccc|cc|cc|cc|cc}
	\hline 
	&  &  &  & \multicolumn{2}{c|}{$G_{q}$} & \multicolumn{2}{c|}{$G_{q,1}$} & \multicolumn{2}{c|}{$\widetilde{Q}_{q}$} & \multicolumn{2}{c}{$Q_{q}^{*}$}\tabularnewline
	\cline{5-12} 
	$p$ & $T$ & $p/T$ & $a$ & $q=1$ & $q=3$ & $q=1$ & $q=3$ & $q=1$ & $q=3$ & $q=1$ & $q=3$\tabularnewline
	\hline 
	5 & 1000 & 0.005 & 0.05 & 0.4865 & 0.2820 & 0.4900 & 0.2820 & 0.4290 & 0.2535 & 0.4290 & 0.2540\tabularnewline
	10 & 2000 & 0.005 & 0.05 & 0.9235 & 0.6080 & 0.9305 & 0.6355 & 0.9010 & 0.6495 & 0.9010 & 0.6500\tabularnewline
	25 & 5000 & 0.005 & 0.05 & 1 & 0.9985 & 1 & 0.9995 & 1 & 0.9985 & 1 & 0.9985\tabularnewline
	40 & 8000 & 0.005 & 0.05 & 1 & 1 & 1 & 1 & 1 & 1 & 1 & 1\tabularnewline
	\hline
	10 & 1000 & 0.01 & 0.05 & 0.4965 & 0.2650 & 0.5145 & 0.2850 & 0.4715 & 0.2450 & 0.4715 & 0.2450\tabularnewline
	20 & 2000 & 0.01 & 0.05 & 0.9310 & 0.6425 & 0.9375 & 0.6855 & 0.9315 & 0.6530 & 0.9315 & 0.6530\tabularnewline
	50 & 5000 & 0.01 & 0.05 & 1 & 0.9995 & 1 & 1 & 1 & 0.9990 & 1 & 0.9990\tabularnewline
	80 & 8000 & 0.01 & 0.05 & 1 & 1 & 1 & 1 & 1 & 1 & 1 & 1\tabularnewline
	\hline
	50 & 1000 & 0.05 & 0.05 & 0.5360 & 0.2825 & 0.6010 & 0.3800 & 0.4795 & 0.2910 & 0.4795 & 0.2905\tabularnewline
	100 & 2000 & 0.05 & 0.05 & 0.9670 & 0.6780 & 0.9865 & 0.8365 & 0.9470 & 0.7450 & 0.9470 & 0.7460\tabularnewline
	250 & 5000 & 0.05 & 0.05 & 1 & 1 & 1 & 1 & 1 & 1 & 1 & 1\tabularnewline
	400 & 8000 & 0.05 & 0.05 & 1 & 1 & 1 & 1 & 1 & 1 & 1 & 1\tabularnewline
	\hline
	10 & 100 & 0.1 & 0.1 & 0.1960 & 0.1360 & 0.2235 & 0.1695 & 0.1392 & 0.1218 & 0.1392 & 0.1214\tabularnewline
	40 & 400 & 0.1 & 0.1 & 0.8435 & 0.5380 & 0.9355 & 0.8050 & 0.7082 & 0.5968 & 0.7082 & 0.5988\tabularnewline
	60 & 600 & 0.1 & 0.1 & 0.9905 & 0.8140 & 0.9990 & 0.9830 & 0.9598 & 0.8808 & 0.9598 & 0.8816\tabularnewline
	100 & 1000 & 0.1 & 0.1 & 1 & 0.9940 & 1 & 1 & 1 & 0.9992 & 1 & 0.9992\tabularnewline
	\hline
	50 & 100 & 0.5 & 0.1 & 0.2110 & 0.1610 & 0.3810 & 0.4660 & 0.1004 & 0.1376 & 0.1004 & 0.1380\tabularnewline
	200 & 400 & 0.5 & 0.1 & 0.9245 & 0.7665 & 0.9990 & 1 & 0.4012 & 0.7708 & 0.4012 & 0.7708\tabularnewline
	300 & 600 & 0.5 & 0.1 & 0.9975 & 0.9725 & 1 & 1 & 0.6626 & 0.9746 & 0.6626 & 0.9748\tabularnewline
	500 & 1000 & 0.5 & 0.1 & 1 & 1 & 1 & 1 & 0.9666 & 1 & 0.9666 & 1\tabularnewline
	\hline
	90 & 100 & 0.9 & 0.1 & 0.2505 & 0.1920 & 0.5970 & 0.7620 & 0.1384 & 0.0992 & 0.1384 & 0.1002\tabularnewline
	360 & 400 & 0.9 & 0.1 & 0.9705 & 0.9080 & 1 & 1 & 0.7138 & 0.5172 & 0.7138 & 0.5176\tabularnewline
	540 & 600 & 0.9 & 0.1 & 1 & 1 & 1 & 1 & 0.9496 & 0.8368 & 0.9496 & 0.8368\tabularnewline
	900 & 1000 & 0.9 & 0.1 & 1 & 1 & 1 & 1 & 0.9998 & 0.9966 & 0.9998 & 0.9966\tabularnewline
	\hline
\end{tabular}}
\end{table}

Table~\ref{Tab:HoskingLitau3} reports on relative errors for the mean, variance and 95 percentile for the 
Hosking test statistic   $\widetilde{Q}_{q}$ and the Li-McLeod test statistic
$Q_q^*$ (with $q=3$).

\begin{table}[!h]
\caption{Relative errors for the mean, variance and 95 percentile for the 
  Hosking test statistic   $\widetilde{Q}_{q}$ and the Li-McLeod test statistic
  $Q_q^*$ (with $q=3$)}\label{Tab:HoskingLitau3}
\resizebox{0.98\textwidth}{!}{
	\def\arraystretch{1.1}
\begin{tabular}{ccc|ccc|ccc}
\hline
\multicolumn{3}{c|}{ } & \multicolumn{3}{c|}{$\widetilde{Q}_{q}$} & \multicolumn{3}{c}{$Q_{q}^{*}$}\tabularnewline
\cline{4-9}
$p$ & $T$ & $p/T$ &\bf Mean & \bf Variance  & \bf95\%Quantile &\bf Mean & \bf Variance  &\bf95\%Quantile\tabularnewline
\hline
10 & 100 & 0.1 & 0.234\% & 19.976\% & 1.366\% & 0.234\% & 24.922\% & 1.547\%\tabularnewline
20 & 200 & 0.1 & 0.067\% & 30.862\% & 0.993\% & 0.067\% & 33.526\% & 1.049\%\tabularnewline
40 & 400 & 0.1 & -0.015\% & 22.057\% & 0.253\% & -0.015\% & 23.286\% & 0.265\%\tabularnewline
60 & 600 & 0.1 & 0.000\% & 21.457\% & 0.162\% & 0.000\% & 22.269\% & 0.162\%\tabularnewline
100 & 1000 & 0.1 & 0.007\% & 20.666\% & 0.125\% & 0.007\% & 21.153\% & 0.125\%\tabularnewline
\hline
50 & 100 & 0.5 & 0.041\% & 267.179\% & 1.322\% & 0.041\% & 282.546\% & 1.354\%\tabularnewline
100 & 200 & 0.5 & 0.007\% & 284.025\% & 0.655\% & 0.007\% & 291.875\% & 0.662\%\tabularnewline
200 & 400 & 0.5 & 0.000\% & 289.080\% & 0.330\% & 0.000\% & 292.998\% & 0.330\%\tabularnewline
300 & 600 & 0.5 & 0.000\% & 297.059\% & 0.222\% & 0.000\% & 299.734\% & 0.222\%\tabularnewline
500 & 1000 & 0.5 & 0.000\% & 296.364\% & 0.134\% & 0.000\% & 297.941\% & 0.134\%\tabularnewline
\hline
80 & 100 & 0.8 & 0.010\% & 1742.257\% & 1.289\% & 0.005\% & 1820.096\% & 1.300\%\tabularnewline
160 & 200 & 0.8 & 0.000\% & 2020.024\% & 0.655\% & 0.000\% & 2063.959\% & 0.657\%\tabularnewline
320 & 400 & 0.8 & 0.000\% & 2214.386\% & 0.332\% & 0.000\% & 2237.811\% & 0.332\%\tabularnewline
480 & 600 & 0.8 & 0.001\% & 2266.151\% & 0.223\% & 0.001\% & 2282.093\% & 0.223\%\tabularnewline
800 & 1000 & 0.8 & 0.000\% & 2348.823\% & 0.137\% & 0.000\% & 2358.701\% & 0.137\%\tabularnewline
\hline
90 & 100 & 0.9 & 0.004\% & 5382.234\% & 1.292\% & 0.000\% & 5618.993\% & 1.297\%\tabularnewline
180 & 200 & 0.9 & 0.000\% & 6906.920\% & 0.657\% & 0.000\% & 7053.897\% & 0.658\%\tabularnewline
360 & 400 & 0.9 & 0.000\% & 8110.500\% & 0.332\% & 0.000\% & 8195.108\% & 0.332\%\tabularnewline
540 & 600 & 0.9 & 0.000\% & 8705.234\% & 0.222\% & 0.000\% & 8764.569\% & 0.222\%\tabularnewline
900 & 1000 & 0.9 & 0.000\% & 9170.563\% & 0.133\% & 0.000\% & 9208.205\% & 0.133\%\tabularnewline
\hline
\end{tabular}}
\end{table}



\begin{thebibliography}{99}
\bibitem[{Bai and Ng(2002)}]{BaiNg_Econometrica_2002}
  \textsc{Bai, J. and Ng, S.} (2002). Determining the number of factors in approximate factor models. \textsl{Econometrica}, {\bf 70}, 191--221.


\bibitem[Bai et al.(2010)]{Bai10}
  \textsc{Bai, Z.,  Chen, J., and  Yao, J.} (2010).
  On estimation of the population spectral distribution from a
  high-dimensional sample covariance matrix.     {\em Australian \&
    New Zeland Journal of Statistics} {\bf 52}, 423-437 

\bibitem[Bai and Silverstein(2004)]{Bai04}
\textsc{Bai, Z., and  Silverstein, J.W.}(2004). CLT for linear
spectral statistics of large-dimensional sample covariance
matrices. {\em Annals of Probability} {\bf 32(1A)},  553-605.

\bibitem[Bai et al.(2009)]{Bai09}
\textsc{Bai, Z., Jiang, D.,  Yao, J. and Zheng, S.}(2009). Corrections to
LRT on Large Dimensional Covariance Matrix by RMT. {\em Annals of
Statistics}, {\bf 37}(6B), 3822-3840

\bibitem[Bai and Wang(2015)]{Bai15}
\textsc{Bai, Z., and Wang, C.}(2015). A note on the limiting spectral distribution of a symmetrized auto-cross covariance matrix.  {\em Statistics \& Probability Letters}, {\bf 96}, 333-340.


\bibitem[{Basu and Michailidis(2015)}]{BM2015}
\textsc{Basu, S. and Michailidis, G. } (2015).
Regularized estimation in sparse high-dimensional time series models. {\em The Annals of Statistics,} {\bf 43(4)}, 1535-1567.


\bibitem[{Bickel and Gel(2011)}] {bickel2011}
\textsc{Bickel, P.~J. and Gel, Y.~R.} (2011).
\newblock{Banded regularization of autocovariance matrices
in application to parameter estimation and forecasting of time series.}
\newblock{{\em Journal of the  Royal Statistical  Society B}, {\bf 73}, 549-92.}

\bibitem[Bhattacharjee and Bose(2016)]{BB16}
\textsc{Bhattacharjee, M. and Bose, A.} (2016). Large sample behaviour of high dimensional autocovariance matrices. {\em The Annals of Statistics}, {\bf 44(2)}, pp.598-628.


\bibitem[{Chang {\sl et al.}(2015)}]{ChangGuoYao_2015}
Chang, J., Guo, B. and Yao, Q. (2015). High dimensional stochastic
regression with latent factors, endogeneity and nonlinearity. {\em J. Econometrics},
{\bf 189}, 297-312.

\bibitem[Chang, Yao and Zhou(2017)]{CYZ17}
\textsc{Chang, J., Yao, Q. and  Zhou, W.} (2017). Testing for high-dimensional white noise using maximum cross-correlations. {\em Biometrika}, {\bf 104(1)}, 111-127.


\bibitem[{Forni {\sl et al.}(2000)}]{Fornietal_2000}
\textsc{Forni, M., Hallin, M., Lippi, M., and  Reichlin, L.} (2000). The generalized dynamic-factor model: Identification and estimation. {\em Review of Economics and statistics}, {\bf 82(4)}, 540-554.

\bibitem[{Forni {\sl et al.}(2005)}]{Fornietal_2005}
\textsc{Forni, M., Hallin, M., Lippi, M., and Reichlin, L.} (2005). The
generalized dynamic factor model: One-sided estimation and
forecasting, {\em Journal of the  American  Statistical  Association}, {\bf 100}, 830--840.

\bibitem[Guo et al.(2016)]{GuoWangYao16}
\textsc{Guo, S., Wang, Y. \& Yao, Q.}(2016).
High dimensional and banded vector autoregressions. {\em Biometrika},
{\bf 103}, 889-903.

\bibitem[Gray(2006)]{Gray06}
\textsc{Gray, R.M. }(2006). Toeplitz and Circulant Matrices: A
Review.
{\em Now Publishers Inc.}


\bibitem[Grenander and Szeg\"{o}(1958)]{Grenander58}
\textsc{Grenander, U., Szeg\"{o}, G.}(1958). Toeplitz Forms and Their Applications. In: California Monographs in Mathematical Sciences. {\em University of California Press}, Berkeley.



\bibitem[{Han and Liu(2015)}]{HL2013}
\textsc{Han, F. and Liu, H.}(2015). A direct estimation of high dimensional stationary vector autoregressions. {\em Journal of Machine Learning Research}, {\bf 16}, 3115-3150.

\bibitem[{Haufe et al.(2009)}]{HNMK2009}
\textsc{Haufe, S., Nolte, G., Mueller, K. ~R., and Kr$\ddot{a}$mer, N.} (2009).
Sparse causal discovery in multivariate time series. {\em In Proceedings of the 2008th International Conference on Causality: Objectives and Assessment}, Volume 6 (pp. 97-106). JMLR.org.

\bibitem[Hosking(1980)]{Hosking80}
\textsc{Hosking, J. R. }(1980). The multivariate portmanteau
statistic. {\em Journal of the American Statistical Association},  {\bf 75}(371), 602-608.

\bibitem[{Hsu et al.(2008)}] {HHC2008}
\textsc{Hsu, N. J., Hung, H. L., and Chang, Y. M.} (2008).
\newblock{Subset selection for vector autoregressive processes using lasso. }
\newblock{{\it Computational  Statistics and Data  Analysis},  {\bf 52}, 3645-3657.}

\bibitem[Jing et al.(2014)]{Jin14}
\textsc{Jin, B., Wang, C., Bai, Z. D., Nair, K. K., Harding, M.}(2014). Limiting spectral distribution of a symmetrized auto-cross covariance matrix. {\em The Annals of Applied Probability},  {\bf 24(3)}, 1199-1225.

\bibitem[Johnstone (2007)]{John07}
  \textsc{Johnstone, I. M.} (2007).
  High dimensional statistical inference and random matrices.
  \textit{International  Congress of  Mathematicians}, \textbf{I}, 307-333.
  Z$\ddot{u}$rich, Switzerland: European Mathematical
  Society.

\bibitem[Lam, C.(2016)]{Lam15}
\textsc{Lam, C. }(2016). Nonparametric Eigenvalue-Regularized Precision or Covariance Matrix Estimator. {\em The Annals of Statistics}, {\bf 44}(3), 928-953.


\bibitem[{Lam and Yao(2012)}]{LamYao_AOS_2012}
Lam, C. and Yao, Q. (2012). Factor modeling for high-dimensional time
series: inference for the number of factors. {\em Annals of Statistics}, {\bf 40}, 694--726.


\bibitem[Li(2004)]{Li03}
\textsc{Li, W. K. }(2004). {\em Diagnostic Checks in Time Series.}
{Chapman \& Hall/CRC}.

\bibitem[Li et al.(2018)]{supp}
  \textsc{Li, Z., Lam, C., Yao, J., Yao, Q.}(2018).
  Supplement to ``On testing for high-dimensional white noise''. DOI:
  xxxxx.

\bibitem[Li and McLeod(1981)]{Li81}
\textsc{Li, W. K., \& McLeod, A. I. }(1981). Distribution of the
residual autocorrelations in multivariate ARMA time series
models. {\em Journal of the Royal Statistical Society. Series B}, {\bf 43},  231-239.

\bibitem[Li and Yao(2016)]{LiYao15}
\textsc{Li, Z., and Yao, J. }(2016). Testing the Sphericity of a covariance matrix when the dimension is much larger than the sample size. {\em Electronic Journal of Statistics}, {\bf 10(2)}, 2973-3010.

\bibitem[Liu et al. (2015)]{LiuAuePaul15}
  \textsc{Liu, H.,  Aue, A. and  Paul, D.} (2015).
  On the Marcenko-Pastur law for linear time series.
  {\em The Annals of Statistics}, {\bf 43(2)},  675-712.


\bibitem[{L\"utkepohl(2005)}]{Lutkepohl_2005}
\textsc{L\"utkepohl, H.} (2005). {\sl New Introduction to Multiple Time Series Analysis}. Springer, Berlin.

\bibitem[Paul and Aue (2014)]{PaulAue14}
\textsc{Paul, D. and Aue, A.} (2014). Random matrix theory in statistics: A review.
{\em Journal of  Statistical Planning and Inference}, {\bf 150}, 1-29.

\bibitem[Sarkar and Chang(1997)]{Sarkar97}
\textsc{Sarkar, S. K. and Chang, C. K.} (1997). The Simes method for multiple hypothesis testing with positively dependent test statistics. {\em Journal of the American Statistical Association}, {\bf 92(440)}, 1601-1608.

\bibitem[{Shojaie and Michailidis(2010)}]{SM2010}
\textsc{Shojaie, A. and Michailidis, G.} (2010).
\newblock{Discovering graphical Granger causality using the truncating LASSO penalty.}
\newblock{{\it Bioinformatics}, {\bf 26}, 517-523.}


\bibitem[Simes(1986)]{Simes86}
\textsc{Simes, R. J. }(1986). An improved Bonferroni procedure for multiple tests of significance. {\em Biometrika}, {\bf 73}(3), 751-754.

\bibitem[Srivastava(2005)]{Srivastava05}
\textsc{M.S. Srivastava}(2005). Some tests concerning the covariance matrix in high-dimensional data, {\em Journal of The Japan Statistical Society}, {\bf 35},251-272.

\bibitem[Stock and Watson(1989)]{SW89}
\textsc{Stock, J. H., Watson, M. W. }(1989). New indexes of coincident and leading economic indicators.{\em NBER macroeconomics annual}, {\bf 4}, 351-394.

\bibitem[Stock and Watson(1998)]{SW98}
\textsc{Stock, J. H., Watson, M. W. }(1998). Diffusion indexes (No. w6702). {\em National bureau of economic research}.

\bibitem[Stock and Watson(1999)]{SW99}
\textsc{Stock, J. H., Watson, M. W. }(1999). Forecasting inflation. {\em Journal of Monetary Economics}, {\bf 44}(2), 293-335.

\bibitem[Tsay(2017)]{Tsay17}
\textsc{Tsay, R.} (2017). Testing for serial correlations in high-dimensional
time series via extreme value theory. {\sl A preprint}.


\bibitem[Wang and Yao(2013)]{WY13}
\textsc{Wang, Q.  and Yao, J.}(2013). On the sphericity test with large-dimensional observations.
{\em Electronic Journalof  Statistics} {\bf 7},2164-2192.


\bibitem[Yao et al. (2015)]{YZB15a}
\textsc{Yao, J., Zheng, S. and Bai. Z.} (2015).
{\em Large Sample Covariance Matrices and High-Dimensional Data Analysis}.
Cambridge University Press, New York and London.

\bibitem[Zheng (2012)]{Zheng12}
\textsc{Zheng, S. }(2012). 
Central limit theorems for linear spectral statistics of large dimensional $F$-matrices.
{\em
  Annales de l'Institut Henri Poincar\'e - Probabilit\'es et Statistiques},
{\bf 48}(2), 444-476

Substitution principle for CLT of linear spectral statistics of high-dimensional sample covariance matrices with applications to hypothesis testing. {\em The Annals of Statistics}, {\bf 43}(2), 546-591.


\bibitem[Zheng et al.(2015)]{Zheng15}
\textsc{Zheng, S., Bai, Z., and Yao, J. }(2015). Substitution principle for CLT of linear spectral statistics of high-dimensional sample covariance matrices with applications to hypothesis testing. {\em The Annals of Statistics}, {\bf 43}(2), 546-591.


\end{thebibliography}

\begin{thebibliography}{99}
\bibitem[Bao and Ullah(2010)]{BU10}
  \textsc{Bao, Y., and Ullah, A. }(2010). Expectation of quadratic forms in normal and nonnormal variables with applications. {\em Journal of Statistical Planning and Inference}, 140(5), 1193-1205.

\bibitem[Ullah(2004)]{U04}
\textsc{Ullah, A. }(2004). Finite sample econometrics. New York: Oxford University Press.

  
\end{thebibliography}
\end{document}